\documentclass[a4paper,10pt]{article}

\usepackage{a4wide}
\usepackage{amsfonts,amssymb,amsmath,amsthm,bm}
\usepackage{mathrsfs, yfonts}
\usepackage{srcltx,tabularx}
\usepackage[utf8]{inputenc} 
\usepackage[cyr]{aeguill}
\usepackage[alphabetic, nobysame]{amsrefs}
\usepackage{enumerate}
\usepackage[english]{babel}
\usepackage{authblk}
\usepackage[colorlinks,pagebackref=false]{hyperref}
\usepackage{cleveref}
\usepackage{graphicx}
\usepackage{pstricks,pst-plot,pst-node}
\usepackage[bottom]{footmisc}
\usepackage[all]{hypcap}
\usepackage[toc,page]{appendix}

\usepackage{mathtools}

\usepackage{tikz}
\usetikzlibrary{matrix, positioning, calc}
\usetikzlibrary{arrows}

\allowdisplaybreaks

\usepackage{breqn}

\newtheorem{theorem}{Theorem}[section]
\newtheorem*{theorem*}{Theorem}
\newtheorem{lemma}[theorem]{Lemma}

\newtheorem{proposition}[theorem]{Proposition}
\newtheorem{corollary}[theorem]{Corollary}
\newtheorem*{corollary*}{Corollary}
\newtheorem*{conjecture*}{Conjecture}
\newtheorem*{question*}{Question}
\newtheorem{claim}{Claim}
\newtheorem*{claim*}{Claim}

\newtheorem{maintheorem}{Theorem}

  \newenvironment{claimproof}[1]{\par\noindent\textit{Proof of the claim:}\space#1}{\hfill $\blacksquare$}

\newcommand{\boundellipse}[3]
{(#1) ellipse (#2 and #3)
}

\theoremstyle{definition}
\newtheorem{definition}[theorem]{Definition}
\newtheorem*{definition*}{Definition}
\newtheorem{example}[theorem]{Example}

\newtheorem{remark}[theorem]{Remark}

\newcommand{\suchthat}{\;\ifnum\currentgrouptype=16 \middle\fi|\;}

\newcommand\algqs{\operatorname{qs-alg}}

\DeclareMathOperator{\Aut}{\mathrm{Aut}}
\DeclareMathOperator{\N}{\mathbf{N}}
\DeclareMathOperator{\Z}{\mathbf{Z}}
\DeclareMathOperator{\R}{\mathbf{R}}
\DeclareMathOperator{\PGL}{\mathrm{PGL}}

\DeclareMathOperator{\SL}{\mathrm{SL}}

\DeclareMathOperator{\SU}{\mathrm{SU}}
\DeclareMathOperator{\SO}{\mathrm{SO}}
\DeclareMathOperator{\Spec}{\mathrm{Spec}}

\DeclareMathOperator{\Sym}{\mathrm{Sym}}

\DeclareMathOperator{\Gal}{\mathrm{Gal}}
\DeclareMathOperator{\Hom}{\mathrm{Hom}}

\DeclareMathOperator{\Aff}{\mathrm{Aff}}
\DeclareMathOperator{\Id}{\mathrm{Id}}
\DeclareMathOperator{\Alg}{\mathrm{alg}}
\DeclareMathOperator{\odd}{\mathrm{odd}}
\DeclareMathOperator{\even}{\mathrm{even}}
\DeclareMathOperator{\trace}{\mathrm{trace}}
\DeclareMathOperator{\Mor}{\mathrm{Mor}}
\DeclareMathOperator{\Ad}{\mathrm{Ad}}
\DeclareMathOperator{\Sub}{\mathbf{Sub}}
\DeclareMathOperator{\End}{\mathrm{End}}
\DeclareMathOperator{\Nrd}{\mathrm{Nrd}}
\DeclareMathOperator{\diag}{\mathrm{diag}}
\DeclareMathOperator{\uram}{\mathrm{ur}}
\DeclareMathOperator{\ram}{\mathrm{ram}}

\DeclareMathOperator{\opp}{\mathrm{op}}

\begin{document}
\title{Chabauty limits of algebraic groups acting on trees \\ \large The quasi-split case}
\author[1]{Thierry Stulemeijer\thanks{Postdoctoral fellow at MPIM Bonn}}

\affil[1]{\small{Max Planck Institute for Mathematics, 53111 Bonn, Germany}}

\date{January 31, 2018}

\maketitle

\begin{abstract}
Given a locally finite leafless tree $ T $, various algebraic groups over local fields might appear as closed subgroups of $ \Aut (T) $. We show that the set of closed cocompact subgroups of $ \Aut (T) $ that are isomorphic to a quasi-split simple algebraic group is a closed subset of the Chabauty space of $ \Aut (T) $. This is done via a study of the integral Bruhat--Tits model of $ \SL_2 $ and $ \SU_3^{L/K} $, that we carry on over arbitrary local fields, without any restriction on the (residue) characteristic. In particular, we show that in residue characteristic $ 2 $, the Tits index of simple algebraic subgroups of $ \Aut (T) $ is not always preserved under Chabauty limits.
\end{abstract}

\maketitle
\tableofcontents

\section{Introduction}
\begin{flushright}
\begin{minipage}[t]{0.41\linewidth}\itshape\small
Ta vague monte avec la rumeur d'un prodige\\ 
C'est ici ta limite. Arr\^ete-toi, te dis-je.\\
\hfill\upshape (Victor Hugo, \emph{L'ann\'ee terrible}, 1872)
\end{minipage}
\end{flushright}

According to well-known rigidity results of J.\ Tits (see \cite{Tits74}*{Theorem~5.8}, together with \cite{Tits84}*{Théorème~2} or \cite{W09}*{Theorem~27.6}), a Bruhat--Tits building of rank $ \geq 2 $ determines uniquely the simple algebraic group and the underlying ground field to which it is associated. In particular, two simply connected absolutely simple algebraic groups of relative rank $ \geq 2 $ over a local field have isomorphic Bruhat--Tits buildings if and only if they are isomorphic as locally compact groups. This contrasts drastically with the rank $ 1 $ case, where infinitely many pairwise non-isomorphic simple algebraic groups of relative rank $ 1 $ can have the same Bruhat--Tits tree. Therefore, given a locally finite leafless tree T, the set $ \Sub (\Aut (T)) $ of closed subgroups of the locally compact group $ \Aut (T) $ may contain infinitely many pairwise non-isomorphic algebraic groups. For example, the Bruhat--Tits tree of the split group $ \SL_2(K) $ is completely determined by the order of the residue field of $ K $, while the isomorphism type of $ \SL_2(K) $ depends on the isomorphism type of the local field $ K $. Since $ \Sub (\Aut (T)) $ carries a natural compact Hausdorff topology, namely the Chabauty topology, we are naturally led to the following question: \textit{what are the Chabauty limits of algebraic groups in $ \Sub (\Aut (T))$?} The goal of this paper is to initiate the study of that problem. In particular, we provide a complete solution in the case of quasi-split groups. 

In order to be more precise, for $ T $ a tree, let us define a \emph{topologically simple algebraic group acting on $ T $} to be a locally compact group isomorphic to $H(K)/Z$, where $K$ is a local field, $H$ is an absolutely simple, simply connected, algebraic group over $K$ of relative rank $ 1 $ whose Bruhat--Tits tree is isomorphic to $ T $, and $Z$ is the centre of $H(K)$.

The first thing to observe is that, given a topologically simple algebraic group $ G $ acting on $ T $, the action homomorphism $ G \to \Aut (T) $ is not canonical, but depends on some choices. There is however a natural way to resolve this issue of canonicity, explained in \cite{CR16}. Following that paper, we shall denote by $ \mathcal{S}_T $ the space of (topological) isomorphism classes of topologically simple closed subgroups of $ \Aut (T) $ acting $ 2 $-transitively on the set of ends. According to \cite{CR16}*{Theorem~1.2}, the space $ \mathcal{S}_T $ endowed with the quotient topology induced from the Chabauty space $ \Sub (\Aut (T)) $ is compact Hausdorff.

We can therefore reformulate the question mentioned above as follows. Let $ \mathcal{S}_T^{\Alg} $ be the set of isomorphism classes of topologically simple algebraic groups acting on $ T $. \textit{What are the accumulation points in $\mathcal{S}_T $ of the elements of $\mathcal{S}_T^{\Alg} $ ?} It seems reasonable to conjecture that $\mathcal{S}_T^{\Alg} $ is closed in $\mathcal{S}_T $. Our main theorem is a partial result in this direction.

\begin{theorem}\label{Thm:mainThm}
Let $ T $ be a locally finite leafless tree, and let $\mathcal{S}_T^{\algqs} $ be the set of isomorphism classes of topologically simple algebraic groups acting on $ T $ that are furthermore quasi-split. Then $\mathcal{S}_T^{\algqs} $ is closed in $\mathcal{S}_T $.
\end{theorem}

As recalled in Section~\ref{SSec:Q S groups}, absolutely simple, simply connected, quasi-split algebraic groups over $K$ of relative rank $ 1 $ are of the form $ \SL_2(K) $ or $ \SU_{3}^{L/K}(K) $ (see Lemma~\ref{Lem:descriptionofqu-splitgroups}). So that in effect, the main goal of the paper is only to dispose of those two ``types'' of groups.

Since the Bruhat--Tits tree of $ \SL_2(K) $ or $ \SU_{3}^{L/K}(K) $ for $ L $ a ramified extension of $ K $ (respectively $ \SU_{3}^{L/K}(K) $ for $ L $ an unramified extension of $ K $) is isomorphic to the $ (p^{n}+1) $-regular tree (respectively the semiregular tree of bidegree ($ p^{3n}+1;p^{n}+1 $)), where $ p^{n} $ is the order of the residue field of $ K $, the space $\mathcal{S}_{T}^{\algqs} $ is empty unless $ T $ is one of those trees. 

It should also be noted that for some trees $ T $, every algebraic group having $ T $ as Bruhat--Tits tree is actually quasi-split. According to the classification tables in \cite{Tits77}*{4.2 and 4.3}, this is the case if and only if $ T $ is the regular tree of degree $ p+1 $ or the semiregular tree of bidegree ($ p^{3n}+1;p^{n}+1 $). Combining this observation with Theorem~\ref{Thm:mainThm}, we get the following corollary.

\begin{corollary}\label{Cor:complete result for special degrees}
Let $ p $ be a prime number, and let $ T $ be the $ (p+1) $-regular tree, or the $ (p^{3n}+1;p^{n}+1) $-semiregular tree. Then the set $\mathcal{S}_T^{\Alg} $ coincides with $ \mathcal{S}_T^{\algqs} $, so that it is closed in $\mathcal{S}_T $.
\end{corollary}

In fact, our method yields an explicit description of the topological space $\mathcal{S}_T^{\algqs} $. To ease the statement of the explicit form of the main theorem, let us introduce some terminology. Recall that a countable totally disconnected topological space $ X $ is classified by two invariants (see \cite{MS20}*{Théorème~1}). More precisely, let $ \hat{\N} $ be the one point compactification of $ \N $ (or in other words, a topological space homeomorphic to $ \lbrace 1, \frac{1}{2}, \frac{1}{3},\dots,0\rbrace \subset \R $). If $ X^{(k)} $ is the last non-empty Cantor-Bendixson derivative of $ X $, and if $ X^{(k)} $ has $ n $ connected components, then $ X $ is homeomorphic to $ \hat{\N}^{k}\times \lbrace 1,\dots,n\rbrace $. In the statement of the following theorems, we use the notation $ \overline{K} $ for the residue field of a local field $ K $, and for any group $ G $, we write $ G/Z $ as a shorthand for the group $ G $ modulo its centre (so that the same letter $ Z $ stands for the centre of various groups). We also make a slight abuse of notation: we represent a point in $ \mathcal{S}_{T} $, which is an isomorphism class, by a representative of that class. This abuse should not cause any confusion, and will simplify notations throughout the rest of the paper.


\begin{theorem}\label{Thm:explicit form of main theorem odd}
Let $ p $ be an odd prime number, and let $ T $ be the $ (p^{n}+1) $-regular tree. Consider the following subsets of $ \mathcal{S}_T^{\algqs} $: 
\begin{align*}
&\mathcal{S}_{\SL_2} = \lbrace \SL_2(K)/Z~\vert~ K \text{ a local field with } \overline{K}\cong \mathbf{F}_{p^{n}}\rbrace\\
&\mathcal{S}_{\SU_3}^{\ram} = \lbrace \SU_3^{L/K}(K)/Z~\vert~ K \text{ a local field with } \overline{K}\cong \mathbf{F}_{p^{n}} \text{ and } L/K \text{ $ ( $separable$ ) $ quadratic ramified}\rbrace 
\end{align*}

Then $ \mathcal{S}_T^{\algqs} = \mathcal{S}_{\SL_2}\sqcup \mathcal{S}_{\SU_3}^{\ram} $ is a countable set. Furthermore, $\mathcal{S}_{\SL_2}$ $ ( $respectively $\mathcal{S}_{\SU_3}^{\ram})$ is a clopen subset of $ \mathcal{S}_T^{\algqs} $ which is homeomorphic to $ \hat{\N} $, the accumulation point being $ \SL_2(\mathbf{F}_{p^{n}}(\!(X)\!))/Z $ $ ( $respectively $ \SU_3^{L_0/\mathbf{F}_{p^{n}}(\!(X)\!)}(\mathbf{F}_{p^{n}}(\!(X)\!))/Z $, where $ L_0 $ is a $ ( $separable$ ) $ quadratic ramified extension of $ \mathbf{F}_{p^{n}}(\!(X)\!) )$.
\end{theorem}

Let us summarise Theorem~\ref{Thm:explicit form of main theorem odd} informally. For $ p $ an odd prime and $ T $ the $ (p^n+1) $-regular tree, the set $ \mathcal{S}_T^{\algqs} $ inside $ \Sub (\Aut (T)) $ can be pictured as follows:
\begin{center}
\begin{tikzpicture}
\foreach \x in {1,1.153,1.3636,1.6666,2.1429,3,4,5,6,7,8,9,10,11,12.5,14,16,18,20,22,24,26,28,30}
    \node (\x) [circle, draw=red, fill = red, scale = 0.2] at (10/\x,0) {};
    \node (L1)[circle, draw=red, fill = red, scale = 0.4] at (0.3,0) {};
\foreach \x in {1,1.153,1.3636,1.6666,2.1429,3,4,5,6,7,8,9,10,11,12.5,14,16,18,20,22,24,26,28,30}
    \node (\x) [circle, draw=black, fill = black, scale = 0.2] at (10/\x,-1) {};
    \node (L2)[circle, draw=black, fill = black, scale = 0.4] at (0.3,-1) {};
\end{tikzpicture}
\end{center}
\begin{center}
\begin{tikzpicture}
 \node (1) [circle, draw=red, fill = red, scale = 0.2] at (0,0) {};
 \node (2) at (0.2,0) [right]{groups of type $\SL_2(K)$};
 \node (3) at (0.2,-0.2) [below right]{char$(K)=0$, $ \overline{K}\cong \mathbf{F}_{p^n} $};
 \node (4) [circle, draw=red, fill = red, scale = 0.4] at (7,0) {};
 \node (5) at (7.2,0) [right]{$\SL_2(\mathbf{F}_{p^n}(\!(X)\!))$};
 \node (6) [circle, draw=black, fill = black, scale = 0.2] at (0,-1.2) {};
 \node (7) at (0.2,-1.2) [right]{groups of type $\SU_3^{L/K}(K)$};
 \node (8) at (0.2,-1.4) [below right]{char$(K)=0$, $ \overline{K}\cong \mathbf{F}_{p^n} $, $ L/K $ ramified};
 \node (9) [circle, draw=black, fill = black, scale = 0.4] at (7,-1.2) {};
 \node (10) at (7.2,-1.2) [right]{$\SU_3^{L_0/\mathbf{F}_{p^n}(\!(X)\!)}(\mathbf{F}_{p^n}(\!(X)\!))$};
 \node (11) at (7.2,-1.4) [below right]{$L_0/\mathbf{F}_{p^n}(\!(X)\!)$ ramified};
 \node [align=center,text width=8cm] at (5.2,-2.5)
        {
            Illustration of Theorem~\ref{Thm:explicit form of main theorem odd}
        };
\end{tikzpicture}
\end{center}

The case of $\SU_3^{L/K}$ for $L/K$ an unramified extension presents a similar behaviour.
\begin{theorem}\label{Thm:explicit form of main theorem ur}
Let $ p $ be any prime number, and let $ T $ be the $ (p^{3n}+1;p^{n}+1) $-semiregular tree. Then $\mathcal{S}_T^{\algqs} $ is homeomorphic to $ \hat{\N} $. More precisely, the countable set
\begin{align*}
\mathcal{S}_T^{\algqs} = \lbrace \SU_3^{L/K}(K)/Z~\vert&~K \text{ a local field with } \overline{K}\cong \mathbf{F}_{p^{n}}\\ &\text{ and } L \text{ $ ( $separable$ ) $ quadratic unramified}\rbrace
\end{align*}
has a unique accumulation point, namely $ \SU_3^{L/\mathbf{F}_{p^{n}}(\!(X)\!)}(\mathbf{F}_{p^{n}}(\!(X)\!))/Z $, where $ L $ is the $ ( $separable$ ) $ quadratic unramified extension of $ \mathbf{F}_{p^{n}}(\!(X)\!) $.
\end{theorem}

Let us summarise Theorem~\ref{Thm:explicit form of main theorem ur} informally. For $ p $ any prime and $ T $ the $ (p^{3n}+1;p^{n}+1) $-semiregular tree, the set $ \mathcal{S}_T^{\algqs} $ inside $ \Sub (\Aut (T)) $ can be pictured as follows:
\begin{center}
\begin{tikzpicture}
\foreach \x in {1,1.153,1.3636,1.6666,2.1429,3,4,5,6,7,8,9,10,11,12.5,14,16,18,20,22,24,26,28,30}
    \node (\x) [circle, draw=black, fill = black, scale = 0.2] at (10/\x,0) {};
    \node (L1)[circle, draw=black, fill = black, scale = 0.4] at (0.3,0) {};
\end{tikzpicture}
\end{center}

\begin{center}
\begin{tikzpicture}
 \node (1) [circle, draw=black, fill = black, scale = 0.2] at (0,-1.2) {};
 \node (2) at (0.2,-1.2) [right]{groups of type $\SU_3^{L/K}(K)$};
 \node (3) at (0.2,-1.4) [below right]{char$(K)=0$, $ \overline{K}\cong \mathbf{F}_{p^n} $, $ L/K $ unramified};
 \node (9) [circle, draw=black, fill = black, scale = 0.4] at (7,-1.2) {};
 \node (10) at (7.2,-1.2) [right]{$\SU_3^{\mathbf{F}_{p^{2n}}(\!(X)\!)/\mathbf{F}_{p^n}(\!(X)\!)}(\mathbf{F}_{p^n}(\!(X)\!))$};
  
   \node [align=center,text width=8cm] at (5.4,-2.5)
        {
            Illustration of Theorem~\ref{Thm:explicit form of main theorem ur}
        };
\end{tikzpicture}
\end{center}
 
The only remaining case when $\mathcal{S}_T^{\algqs} $ is not empty is the case of the regular tree of degree $ (2^{n}+1) $. In this case, the topological space $\mathcal{S}_T^{\algqs} $ cannot be cut into two copies of $ \hat{\N} $. Indeed, it exhibits a much richer structure.

\begin{theorem}\label{Thm:explicit form of main theorem rameven}
Let $ T $ be the $ (2^{n}+1) $-regular tree. Then $\mathcal{S}_T^{\algqs} $ is homeomorphic to $ \hat{\N}^{2} $. More precisely, 
\begin{align*}
 &\mathcal{S}_T^{\algqs} = \lbrace \SL_2(K)/Z~\vert~ K \text{ a local field with } \overline{K}\cong \mathbf{F}_{2^{n}}\rbrace \\
& \cup \lbrace \SU_3^{L/K}(K)/Z~\vert~ K \text{ a local field with } \overline{K}\cong \mathbf{F}_{2^{n}} \text{ and } L \text{ separable quadratic ramified}\rbrace
\end{align*}
is a countable set. The first Cantor-Bendixson derivative of $ \mathcal{S}_T^{\algqs} $ is
\begin{equation*}
\lbrace \SU_3^{L/\mathbf{F}_{2^{n}}(\!(X)\!)}(\mathbf{F}_{2^{n}}(\!(X)\!))/Z~\vert~L \text{ is separable 
quadratic ramified}\rbrace \cup \lbrace \SL_2(\mathbf{F}_{2^{n}}(\!(X)\!))/Z \rbrace
\end{equation*}
while its second Cantor-Bendixson derivative contains the single element $ \SL_2(\mathbf{F}_{2^{n}}(\!(X)\!))/Z $. Also, the subset $ \lbrace \SL_2(K)/Z~\vert~ K \text{ a local field with } \overline{K}\cong \mathbf{F}_{2^{n}}\rbrace $ is closed in $ \mathcal{S}_T^{\algqs} $, homeomorphic to $ \hat{\N} $ and with accumulation point $ \SL_2(\mathbf{F}_{2^{n}}(\!(X)\!))/Z $.
\end{theorem}

We also draw a picture illustrating Theorem~\ref{Thm:explicit form of main theorem rameven}. Let $ T $ be the $ (2^n+1) $-regular tree. The set $ \mathcal{S}_T^{\algqs} $ inside $ \Sub (\Aut (T)) $ can be pictured as follows:
\begin{center}
\begin{tikzpicture}
\foreach \x in {2.1703,3,4.8,7,9,11,14,18,22,26,30}
\foreach \y in {2.1703,3,4.8,7,9,11,14,18,22,26,30}
    \node (\x,\y) [circle, draw=black, fill = black, scale = 0.2] at (-10/\x,10/\y) {};
\foreach \y in {2.1703,3,4.8,7,9,11,14,18,22,26,30}
\node (L\y)[circle, draw=black, fill = black, scale = 0.4] at (-0.3,10/\y) {};
\foreach \x in {2.1703,3,4.8,7,9,11,14,18,22,26,30}
   \node (L\x)[circle, draw=black, fill = black, scale = 0.4] at (-10/\x,0.3) {};
\foreach \x in {1.6666,2.1429,3,4,5,6,7,8,9,10,11,12.5,14,16,18,20,22,24,26,28,30}
\node (\x) [circle, draw=red, fill = red, scale = 0.2] at (10/\x -0.5,0.27) {};
\node (LL)[circle, draw=yellow, fill = yellow, scale = 0.6] at (-0.25,0.25) {};
\end{tikzpicture}
\end{center}
\begin{center}
\begin{tikzpicture}
 \node (1) [circle, draw=red, fill = red, scale = 0.2] at (0,0) {};
 \node (2) at (0.2,0) [right]{groups of type $\SL_2(K)$};
 \node (3) at (0.2,-0.2) [below right]{char$(K)=0$, $ \overline{K}\cong \mathbf{F}_{2^n} $};
 \node (4) [circle, draw=yellow, fill = yellow, scale = 0.6] at (7,0) {};
 \node (5) at (7.2,0) [right]{$\SL_2(\mathbf{F}_{2^n}(\!(X)\!))$};
 \node (6) [circle, draw=black, fill = black, scale = 0.2] at (0,-1.2) {};
 \node (7) at (0.2,-1.2) [right]{groups of type $\SU_3^{L/K}(K)$};
 \node (8) at (0.2,-1.4) [below right]{char$(K)=0$, $ \overline{K}\cong \mathbf{F}_{2^n} $, $ L/K $ ramified};
 \node (9) [circle, draw=black, fill = black, scale = 0.4] at (7,-1.2) {};
 \node (10) at (7.2,-1.2) [right]{groups of type $\SU_3^{L/\mathbf{F}_{2^n}(\!(X)\!)}(\mathbf{F}_{2^n}(\!(X)\!))$};
 \node (11) at (7.2,-1.4) [below right]{$L/\mathbf{F}_{2^n}(\!(X)\!)$ separable ramified};
   \node [align=center,text width=8cm] at (6.2,-2.5)
        {
            Illustration of Theorem~\ref{Thm:explicit form of main theorem rameven}
        };
\end{tikzpicture}
\end{center}

It is also important to note that in the statement of Theorem~\ref{Thm:explicit form of main theorem odd}, Theorem~\ref{Thm:explicit form of main theorem ur} and Theorem~\ref{Thm:explicit form of main theorem rameven}, we do not describe precisely $ \mathcal{S}_T^{\algqs} $ as a set. Indeed, we do not give the criterion allowing one to know when two equivalence classes are the same, or in other words when two given groups appearing in those theorems are topologically isomorphic. A precise criterion could be easily stated using the work of J.\ Tits on abstract homomorphism of algebraic groups (this is discussed in more details in the proof of those theorems). For example, let $ L $ and $ L'$ be the two quadratic ramified extensions of $ \mathbf{F}_{p^n}(\!(T)\!) $ for $ p $ an odd prime. Then the pairs of fields ($ \mathbf{F}_{p^n}(\!(T)\!) $, $ L $) and ($ \mathbf{F}_{p^n}(\!(T)\!) $, $ L' $) are isomorphic (in the sense of Definition~\ref{Def:space L ramified}, see also Remark~\ref{Rem:one accumulation point ram}). This explains why in Theorem~\ref{Thm:explicit form of main theorem odd}, the accumulation point $\SU_3^{L_0/\mathbf{F}_{p^n}(\!(T)\!)}(\mathbf{F}_{p^n}(\!(T)\!))$ is represented by any (of the two) quadratic ramified extension $ L_0 $ of $ \mathbf{F}_{p^n}(\!(T)\!) $.
  
As one can see from Theorem~\ref{Thm:explicit form of main theorem rameven}, we face a more complex situation in residue characteristic $ 2 $. Indeed, this theorem implies that the split group $ SL_2 (\mathbf{F}_{2^{n}} (\!(X)\!))/Z  $ is a limit of unitary groups, thereby illustrating the fact that the Tits index need not be preserved under Chabauty limits in residue characteristic $ 2 $. In other words, the map associating to an isomorphism class in $ \mathcal{S}_T^{\Alg} $ its Tits index is not continuous. The specific features of Chabauty limits in residue characteristic $ 2 $ highlight the complexity of the aforementioned conjecture, which will be addressed in full generality in a forthcoming paper, but with different methods.
%

Despite the fact that $ \mathcal{S}_T^{\algqs} $ depends very much on $ T $, the strategy to prove our results is the same for all $ T $ and for all algebraic groups under consideration (i.e. $ \SL_2 $ or $ \SU_3 $). Let us outline it in the $ \SL_2 $ case (our notational conventions for local fields are spelled out at the beginning of Section~\ref{Sec:Def of alg. grps}). 
\begin{enumerate}
\item\label{Itm:first} In Definition~\ref{Def:buildingassociated to a data} and Section~\ref{SSec:B-T Tree SL_2(D)}, we recall the definition of the Bruhat--Tits tree: 
\begin{equation*}
\mathcal{I} = \SL_2(K)\times \mathbf{R}/\sim
\end{equation*}
\item In Lemma~\ref{Lem:B_0(r) is really the ball of radius r}, we observe that after renormalising the valuation so that $ \omega (\pi_K) = 1 $, the ball around $ 0 $ of radius $ r $ in $ \mathcal{I} $ is: 
\begin{equation*}
B_0(r) = \lbrace [(g,x)]\in \mathcal{I}~\vert~g\in \SL_2(\mathcal{O}_{K}) , x\in [-\omega (\pi_{K}^{r}), \omega (\pi_{K}^{r})]\subset \R \rbrace
\end{equation*}
\item\label{Itm:last} In Definition~\ref{Def:local tree of radius r SL2}, we define a local version (around $ 0 $ and of radius $ r $) of the Bruhat--Tits tree:
\begin{align*}
\mathcal{I}^{0,r} = \SL_2(\mathcal{O}_{K}/\mathfrak{m}_{K}^{r}) \times [-\omega (\pi_{K}^{r}),\omega (\pi_{K}^{r})]/\sim_{0,r}
\end{align*}
and we show in Theorem~\ref{Thm:localdescriptionoftheball} that the homomorphism $ \SL_2(\mathcal{O}_{K})
\to \SL_2(\mathcal{O}_{K}/\mathfrak{m}_{K}^{r}) $ induces an $ (\SL_2(\mathcal{O}_{K})\to 
\SL_2(\mathcal{O}_{K}/\mathfrak{m}_{K}^{r}))$-equivariant bijection $ B_{0}(r)\to \mathcal{I}^{0,r} 
$.
\item Following an idea dating back to M.\ Krasner (see \cite{D84} for references, this idea is also used in e.g. \cite{Ka86}), we define a metric $ d $ on the space $ \mathcal{K} $ of (isomorphism classes of) local fields by declaring that for $ r\in \N $ and $ K_1,K_2 \in \mathcal{K} $, $ d(K_1;K_2)\leq \frac{1}{2^{r}} $ if and only if $ \mathcal{O}_{K_{1}}/\mathfrak{m}_{K_1}^{r}\cong \mathcal{O}_{K_2}/\mathfrak{m}_{K_2}^{r} $ (see Lemma~\ref{Lem:Non-archim. metric space}). We observe in Proposition~\ref{Prop:explicit description of D} that the space $ \mathcal{K}_{p^{n}} $ of (isomorphism classes of) local fields having residue field $ \mathbf{F}_{p^{n}} $ is homeomorphic to $ \hat{\N} $.
\item Points \ref{Itm:first} to \ref{Itm:last} imply that if $ K_1 $ and $ K_2 $ are close to each other in $ \mathcal{K}_{p^{n}} $, then $ \SL_2(\mathcal{O}_{K_1}) $ and $ \SL_2(\mathcal{O}_{K_2}) $ are close to each other in the Chabauty space of $ \Aut (T_{p^{n}+1}) $ (where $ T_{p^{n}+1} $ is the ($ p^{n}+1 $)-regular tree). Indeed, up to isomorphism, they act in the same way on a large ball centred at $ 0 $. This is the key step in the proof of Theorem~\ref{Thm:continuity from D to Chab(Aut(T))}.
\item We are then able to conclude effortlessly, using a rigidity argument, that the map $ \mathcal{K}_{p^{n}}\to\mathcal{S}_{T_{p^{n}+1}}^{\Alg}\colon K\mapsto \SL_2(K)/Z $ is a homeomorphism onto its image.
\end{enumerate}

In light of this outline, it seems natural to consider Theorem~\ref{Thm:localdescriptionoftheball} (together with its variants for other types of groups) as the central result of this paper.

A key tool to implement our strategy is the existence of good functors from $ \mathcal{O}_{K} $-algebras (such as $ \mathcal{O}_{K}/\mathfrak{m}_{K}^{r} $) to groups (like $ \SL_2(\mathcal{O}_{K}/\mathfrak{m}_{K}^{r}) $). The integral model provided by Bruhat--Tits theory plays the role of this good functor. In the $ \SL_2 $ case, this is just the algebraic group $ \SL_2 $ considered over $ \mathcal{O}_{K} $. But a description of the integral model is not always so straightforward, and an important feature of this article is an explicit computation of Bruhat--Tits models for $ \SU_3^{L/K} $, especially in the more delicate case when the residue characteristic is $ 2 $ and $ L $ is ramified.
%

The complexity of the integral model of $ \SU_3^{L/K} $ when the residue characteristic is $ 2 $ and $ L $ is ramified also explains why we get a different behaviour for regular trees of degree $ 2^n+1 $ in Theorem~\ref{Thm:explicit form of main theorem rameven}. As often in the theory of algebraic groups, the characteristic $ 2 $ case is more involved to work out (and in our situation, it is again because of the presence of orthogonal groups in characteristic $ 2 $ lurking in the background, see Remark~\ref{Rem:reductive quotient of the closed fibre}). However, one should always make the effort of including this case, if only to avoid the wrath of J.\ Tits (see for example the introduction of \cite{KMRT98}).

It also appears that studying convergence of groups isomorphic to $ \SL_2(D)/Z $ (where $ D $ is a finite dimensional central division algebra over a local field $ K $) can be done in parallel to the $ \SL_2(K) $ case. Hence we decided to treat this case as well in this paper. We stress that this is only an opportunistic choice, and that the other cases should be settled by first considering similar questions in arbitrary rank for quasi-split groups, and then by applying a descent method.

Nevertheless, thanks to this treatment, we get the following results as well.
\begin{theorem}
Let $ T $ be a locally finite leafless tree, and let $\mathcal{S}_T^{\SL_2(D)} $ be the set of isomorphism classes of topologically simple algebraic groups acting on $ T $ that are furthermore isomorphic to $ \SL_2(D)/Z $ for some central division algebra $ D $. Then $\mathcal{S}_T^{\SL_2(D)} $ is closed in $\mathcal{S}_T $.
\end{theorem}

Hence, for the reasons explained before Corollary~\ref{Cor:complete result for special degrees} and according to the tables in \cite{Tits77}*{4.2 and 4.3}, we obtain the following strengthening of Corollary~\ref{Cor:complete result for special degrees}.
\begin{corollary}\label{Cor:complete result for special degrees strengthening}
Let $ p $ be a prime number, and let $ T $ be the $ (p^{n}+1) $-regular tree where $ n $ is not divisible by $ 3 $, or the $ (p^{3n}+1;p^{n}+1) $-semiregular tree. Then the set $\mathcal{S}_T^{\Alg} $ coincides with $ \mathcal{S}_T^{\algqs}\cup \mathcal{S}_{T}^{\SL_2(D)} $, so that it is closed in $\mathcal{S}_T $.
\end{corollary}

Again, just as for the quasi-split case, we are actually able to describe explicitly the topological space $ \mathcal{S}_{T}^{\SL_2(D)} $ and all the convergences in this space. In the following theorem, $ \overline{D} $ denotes the residue field of a finite dimensional central division algebra $ D $ over $ K $. We also make the same abuse of notations than in the previous theorems (in particular, the same letter $Z$ denotes the centre of various groups).
\begin{theorem}\label{Thm:explicit form for SL_2(D)}
Let $ T $ be the $ (p^{n}+1) $-regular tree.
\begin{enumerate}
\item The topological space $ \mathcal{S}_{T}^{\SL_2(D)} $ is homeomorphic to $ \hat{\mathbf{N}}\times \lbrace 1,\dots,\lceil \frac{n+1}{2}\rceil \rbrace $. The first Cantor-Bendixson derivative of $ \mathcal{S}_{T}^{\SL_2(D)} $ is the set $$ \lbrace \SL_2(D)/Z~\vert~\overline{D}\cong \mathbf{F}_{p^n} \text{ and } D \text{ is of characteristic } p\rbrace \subset \mathcal{S}_T, $$
which contains $ \lceil \frac{n+1}{2}\rceil $ elements.
\item For $ i\in \N $, let $ D_i $ $ ( $respectively $ D) $ be a finite dimensional central division algebra over $ K_i $ $ ( $respectively $ K) $ having residue field of cardinality $ p^n $. Let $ d_i $ $ ( $respectively $ d) $ be the degree of $ D_i $ $ ( $respectively $ D) $, so that $ \vert \overline{K}_i\vert^{d_i} = p^n = \vert \overline{K}\vert^{d} $, where $ \overline{K}_i $ $ ( $respectively $ \overline{K} )$ denotes the residue field of $ K_i $ $ ( $respectively $ K) $. Let $ r_i $ $ ( $respectively $ r) $ be the Hasse invariant of $ D_i $ $ ( $respectively $ D) $, as in Definition~\ref{Def:Hasse invariant}. If $ (\SL_2(D_i)/Z)_{i\in \N} $ converges to $ \SL_2(D)/Z $ in the Chabauty space \textbf{$ \Sub $}$ (\Aut (T)) $, then for all $ i $ large enough, $ r_i = \pm r $ and $ d_i = d $, so that $ \vert \overline{K}_i\vert = \vert \overline{K}\vert $ as well.
\end{enumerate}
\end{theorem} 

We conclude this introduction by mentioning the recent work of M.\ de la Salle and R.\ Tessera \cite{dlST15}, who used independently closely related ideas in their study of the space of Bruhat--Tits buildings of type $ \tilde{A}_n $ (with $ n > 2 $) endowed with the Gromov--Hausdorff topology.

\subsection*{Acknowledgements}
I gratefully thank an anonymous commenter to a question on the website MathOverflow (see \cite{St16}) for explaining how to extract from \cite{EGA4} the form of Hensel's Lemma we needed. I reproduced his comments for the proof of Theorem~\ref{Thm:genhensel}. I am also grateful to M.\ de la Salle and R.\ Tessera for their comments on this work, which led me to include the $ \SL_2(D) $ case. Furthermore, I warmly thank P.-E.\ Caprace and N.\ Radu for introducing me to this interesting topic, and for their enthusiasm about this work. The former also gave the slick argument to deduce Chabauty convergence of the whole group from Chabauty convergence of vertex stabilisers, while the latter also made the initial breakthrough by computing Chabauty limits in the $ \SL_2 $ case. I am also indebted to an anonymous referee, whose comments prompted many improvements in my exposition.

\section{Definitions of the algebraic groups under consideration}\label{Sec:Def of alg. grps}
For the rest of the paper, $ K $ will denote a local field (all our local fields are assumed to be non-archimedean), and $ D $ will denote a finite dimensional central simple division algebra over $ K $. Let us spell out our notational conventions for the objects associated with $ K $ (respectively $ D $): the ring of integers is denoted $ \mathcal{O}_{K} $ (respectively $ \mathcal{O}_{D} $), its maximal ideal by $ \mathfrak{m}_{K} $ (respectively $ \mathfrak{m}_{D} $), a uniformiser by $ \pi_K $ (respectively $ \pi_D $) and $ \overline{K} $ (respectively $ \overline{D} $) denotes the residue field. The valuation of $ K $ (respectively $ D $), and also its unique extension to any finite extension of $ K $, is denoted by $ \omega $. We use the notation $ \mathbf{Q}_{p^{n}} $ for the unique (up to isomorphism) unramified extension of $ \mathbf{Q}_{p} $ of degree $ n $.

Also, in order to avoid the repetition of long lists of adjectives, in this section, by an algebraic group, we mean an absolutely simple, simply connected algebraic group over an arbitrary field $ k $ (in this paper, we only work in the case when $ k $ is a local field, but nevertheless, we prefer to state Lemma~\ref{Lem:descriptionofqu-splitgroups} over an arbitrary base field).

\subsection{Quasi-split groups of relative rank \texorpdfstring{$ 1 $}{1}}\label{SSec:Q S groups}
As mentioned in the introduction, the Bruhat--Tits building of an algebraic group $ G $ over a local field is a tree if and only if $ G $ is of relative rank $ 1 $. Instead of giving the general definition of quasi-split algebraic groups, and then specialising to those that are of relative rank $ 1 $, we take a practical approach and give an explicit description of those groups, the result being that they are all of the form $ \SL_{2} $ or $ \SU_3 $ (and this is the case over any field). We begin by recalling the definition of $ \SU_3 $. It is customary to choose a presentation of $ \SU_3 $ using the transposition along the anti-diagonal, that we denote $ ^{S}(.) $ so that explicitly, if $ g $ is a $ 3 $-by-$ 3 $ matrix , $ (^{S}g)_{-j,-i}=g_{ij} $, for $ i,j\in \lbrace -1,0,1\rbrace $.

\begin{definition}\label{Def:SU(3,f)}
Let $ k $ be a field, let $ l $ be a separable quadratic extension of $ k $, and let $ \sigma $ be the nontrivial element of $ \Aut (l/k) $, whose action by conjugation on $ l $ is denoted $ x \mapsto \bar{x} $. We define
\begin{equation*}
\SU_3^{l/k}(k) = \lbrace g\in \SL_{3}(l) ~\vert~ ^{S}\bar{g}g = \Id \rbrace
\end{equation*}
We denote $ \SU_3^{l/k} $ (or simply $ \SU_3 $ when the pair of field $ (k,l) $ is arbitrary or understood from the context) the corresponding algebraic group over $ k $. Note that the equations $ \det (g) -1 $ and $ ^{S}\bar{g}g - \Id $ (together with the embedding $ l\hookrightarrow M_2(k) $) realise $ \SU_3^{l/k} $ as a closed subspace of the affine space $ \mathbf{A}_{k}^{n} $, where $ n = 4\times 3^{2} $. Using this, it is readily seen that $ \SU_3 $ is an algebraic group over $ k $.
\end{definition}
\begin{remark}\label{Rem:thehermitianform}
The group $ \SU_3 $ defined above is the special unitary group with respect to the following hermitian form of $ l^{3} $:
\begin{equation*}
((x_{-1},x_{0},x_{1}),(y_{-1},y_{0},y_{1}))\mapsto \overline{x}_{-1}y_{1}+\overline{x}_{0}y_{0}+\overline{x}_{1}y_{-1}
\end{equation*}
The advantage of taking this peculiar hermitian form is that the associated involution preserves the group of upper triangular matrices. As Lemma~\ref{Lem:descriptionofqu-splitgroups} shows, up to isomorphism, there is only one ``type" of non-split, quasi-split algebraic group of relative rank $ 1 $ (and this is the case over any base field). Hence, choosing the above hermitian form is in fact not restrictive.
\end{remark}

We can now describe quasi-split algebraic groups of relative rank $ 1 $ (recall that by the convention of this section, all our algebraic groups are absolutely simple, simply connected, algebraic groups over a field $ k $).
\begin{lemma}\label{Lem:descriptionofqu-splitgroups}
Let $ k $ be a field and let $ G $ be a quasi-split algebraic group of relative rank $ 1 $ over $ k $. Then $ G $ is one of the following groups:
\begin{enumerate}
\item $ \SL_{2} $ over $ k $.
\item $ \SU_3^{l/k} $, where $ l $ is as in Definition~\ref{Def:SU(3,f)}.
\end{enumerate}
\end{lemma}
\begin{proof}
If $ G $ is quasi-split, then by definition, its anisotropic kernel is trivial. Hence, by \cite{Tits66}*{2.7.1, Theorem~2}, $ G $ is entirely determined (up to $ k $-isomorphism) by its Dynkin diagram together with the $ \ast $-action on it (or in other words, $ G $ is determined by its index). Also note that the number of orbit under this $ \ast $-action is the relative rank, so that according to \cite{Tits66}*{Table II}, the only possibilities for the index are 
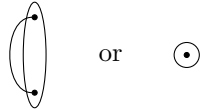
\begin{center}
\begin{tikzpicture}
 \node (1) [circle, draw=black, fill = black, scale = 0.2] at (1,0) {};
 \node (2) [circle, draw=black, fill = black, scale = 0.2] at (1,-1) {};
\draw \boundellipse{1,-0.5}{0.15}{0.7};
\draw [bend right=90,-] (1) to (2);
 \node (text) at (2,-0.5) {or};
 \node (3) [circle, draw=black, fill = black, scale = 0.2] at (3,-0.5) {}; 
 \node (4) [circle, draw=black, scale = 1] at (3,-0.5) {}; 
\end{tikzpicture}
\end{center}
The first index is the index of the quasi-split group $ \SU_3^{l/k} $, where $ l $ is any separable quadratic extension of $ k $, while the second index is the index of the split group $ \SL_2 $.
\end{proof}
%

\subsection{Definition of the algebraic group \texorpdfstring{$\SL_2(D)$}{SL2(D)}}
As outlined in the introduction, treating the case of the group $ \SL_2(D) $ (where $ D $ is a finite dimensional central division algebra) is very close to treating the case of $ \SL_2(K) $, so that we decided to include this case as well. Let us recall the definition of the group $ \SL_2(D) $.
\begin{definition}\label{Def:SL_2(D)}
Let $ D $ be a finite dimensional central division algebra over $ K $, and consider $ D^2 $ as a right $ D $-vector space. We define the group $ \SL_2(D) = \lbrace u\in \End_{D}(D^2)~\vert~\Nrd (u) = 1 \rbrace $, where $ \Nrd (u) $ stands for the reduced norm of $ u $ (we recall the definition of the reduced norm in Definition~\ref{Def:reduced norm}).
\end{definition}

Let us stress again that the case of main interest is the case of quasi-split groups, and that $ \SL_2(D) $ is quasi-split if and only if $ D=K $. We advice the reader to consider only this case in a first reading, and to encourage this attitude, the facts needed when $ D\neq K $ are relegated to Appendix~\ref{App:division algebra theory} and Appendix~\ref{App:integral models over D}.

When $ D=K $, the group $ \SL_2(K) $ is the group of rational points of a closed subspace $ \SL_2 $ of the affine space $ \mathbf{A}_{K}^{4} $ defined by the polynomial equation $ \det \!-1 $. It is then straightforward to check that $ \SL_2 $ is indeed an algebraic group over $ K $.

For arbitrary $ D $, it is well-known that $ \SL_2(D) $ can be seen as the group of rational point of an algebraic group over $ K $. We recall in Appendix~\ref{App:division algebra theory} the standard facts about division algebras, and we also discuss in Appendix~\ref{App:integral models over D} the representation of $ \SL_2(D) $ as an algebraic group over $ K $.

\section{The Bruhat--Tits tree of a group}\label{Sec:Rappel sur l'arbre de B-T}
In their fundamental paper \cite{BrTi1}, F.\ Bruhat and J.\ Tits show how to construct an affine building given a group $ G $ with a valued root datum, see especially \cite{BrTi1}*{7.4.1 and 7.4.2}. In rank one, this affine building is a tree, and their construction only uses a collection $ \lbrace (P_x)_{x\in \R}, N \rbrace $ of subgroups of $ G $  together with a homomorphism $ \nu \colon N\to \Aff (\R ) $, which are carefully constructed using the valued root datum on $ G $. Since this construction is the central object of this paper, we begin by recalling it. We use the same notations than in loc. cit., except for $ \hat{P}_x $ that we denote $ P_x $ instead.

\begin{definition}[{\cite{BrTi1}*{7.4.1 and 7.4.2}}]\label{Def:buildingassociated to a data}
Let $ G $ be a group with a valued root datum of rank one, and let $ \lbrace (P_x)_{x\in \R }, N \rbrace $ be a collection of subgroups of $ G $ and $ \nu \colon N\to \Aff (\R ) $ be a homomorphism. Assume that they are obtained from the valued root datum as prescribed in \cite{BrTi1}*{§6 and §7}. Define an equivalence relation on $ G\times \R $ as follows: $ (g,x)\sim (h,y) $ if and only if there exists $ n\in N $ such that $ y=\nu (n)(x) $ and $ g^{-1}hn\in P_{x} $. The Bruhat--Tits tree of $ G $ is $ \mathcal{I} = G\times \R /\sim $. We write $ [(g,x)] $ for the equivalence class of $ (g,x) $ in $ \mathcal{I} $. The group $ G $ acts on $ \mathcal{I} $ by multiplication on the first component.
\end{definition}

In this paper, we take a practical approach bypassing the valued root datum. In each case that we need it, we construct the Bruhat--Tits tree by giving directly the groups $ (P_x)_{x\in \R}, N $ and the homomorphism $ \nu \colon N\to \Aff (\R ) $. We can fortunately easily ensure that the given groups (together with the homomorphism $ \nu $) are indeed obtained from a valued root datum as prescribed in \cite{BrTi1}*{§6 and §7} thanks to the explicit computations made in \cite{BrTi1}*{§10}.

\begin{remark}\label{Rem:metric on I and stabiliser}
For $ g\in G $, the map $ f_g\colon \R \to \mathcal{I}\colon x\mapsto g.[(\Id ,x)] $ is injective, by the discussion in \cite{BrTi1}, below Definition~7.4.2. An apartment of $ \mathcal{I} $ is a subset of the form $ f_g(\R ) $ for some $ g\in G $, and we can endow $ \mathcal{I} $ with a metric which gives the usual metric on $ \R $ when restricted to any apartment. The action of $ G $ on its Bruhat--Tits tree preserves such a metric. Furthermore, in view of \cite{BrTi1}*{Proposition~7.4.4}, $ P_x $ is in fact the stabiliser of $ [(\Id ,x)]\in \mathcal{I} $.
\end{remark}
\begin{remark}\label{Rem:equivalent def with for all n}
Note that in Definition~\ref{Def:buildingassociated to a data}, it is equivalent to say that $ (g,x)\sim (h,y) $ if and only if for all $ \tilde{n}\in N $ such that $ \nu (\tilde{n})(x)=y $, we have $ g^{-1}h\tilde{n}\in P_{x} $. Indeed, if there exists $ n\in N $ such that $ \nu (n)(x) = y $ and $ g^{-1}hn\in P_{x} $, let $ \tilde{n} $ be any element of $ N $ such that $ \nu (\tilde{n})(x) = y $. Then $ g^{-1}h\tilde{n} = g^{-1}hnn^{-1}\tilde{n} $. But $ n^{-1}\tilde{n} $ stabilises $ [(\Id ,x)] $, and hence belongs to $ P_x $ by Remark~\ref{Rem:metric on I and stabiliser}. Thus, $ g^{-1}hnn^{-1}\tilde{n} $ belongs to $ P_{x} $ as well, as wanted.
\end{remark}

We end this section by recalling two facts about Bruhat--Tits trees that will be needed later on.
\begin{lemma}\label{Lem:integralityofn}
Let $ g,h\in P_{0} $, and let $ x,y\in \R $. If $ (g,x) \sim (h,y) $, there exists $ n\in N\cap P_{0} $ such that $ \nu (n)(x) = y $
\end{lemma}
\begin{proof}
Recall that $ P_0 $ is the stabiliser of $ [(\Id ,0)]\in \mathcal{I} $ in $ G $ (see Remark~\ref{Rem:metric on I and stabiliser}). Since $ G $ acts by isometries on $ \mathcal{I} $, and since $ g,h\in P_{0} $, we have
\begin{align*}
\vert x\vert &= d_{\mathcal{I}}([(\Id,x)];[(\Id ,0)]) = d_{\mathcal{I}}([(g,x)];[(\Id ,0)])\\
\vert y\vert &= d_{\mathcal{I}}([(\Id,y)];[(\Id ,0)]) = d_{\mathcal{I}}([(h,y)];[(\Id ,0)])
\end{align*}
where $ d_{\mathcal{I}} $ denotes the distance in the metric space $ \mathcal{I} $ (see Remark~\ref{Rem:metric on I and stabiliser}). But if $ (g,x)\sim (h,y) $, we have in particular $ d_{\mathcal{I}}([(g,x)];[(\Id ,0)]) = d_{\mathcal{I}}([(h,y)];[(\Id ,0)]) $, and hence $ \vert x\vert = \vert y\vert $. Thus, the existence of $ n\in N\cap P_{0} $ such that $ \nu (n)(x) = y $ follows from the fact that $ \nu (N\cap P_0) $ is the (spherical) Weyl group of $ G $.
\end{proof}
\begin{remark}
Again, we only use this proposition when $ G $ is a group of the form $ \SL_2(D) $ or $ \SU_3^{L/K} $. In each case, we will see explicitly that there exists an element $ n $ in $ N\cap P_0 $ such that $ \nu (n)\colon \R\to \R\colon x\mapsto -x $.
\end{remark}

\begin{lemma}\label{Lem:B-T tree from a based point}
Let $ G $ be a group having a valued root datum of rank one and let $ \mathcal{I} $ be the corresponding Bruhat--Tits tree. Then $ \mathcal{I} = \lbrace [(g,x)]\in \mathcal{I}~\vert~g\in P_0 \rbrace $
\end{lemma}
\begin{proof}
Let $ [(g,x)]\in \mathcal{I} $. Since $ G $ acts strongly transitively on $ \mathcal{I} $ (\cite{BrTi1}*{Corollaire~7.4.9}), there exists $ h\in P_{0} $ such that $ h.[(g,x)] = [(\Id,y)] $, for some $ y\in \R $. Hence, $ [(g,x)] = [(h^{-1},y)] $ and $ h^{-1}\in P_0 $, as wanted.
\end{proof}

\section{Convergence of groups of type \texorpdfstring{$\SL_2(D)$}{SL2(D)}}\label{Sec:SL_2(D)}
We recall the reader that our notational conventions for local fields and their finite dimensional central division algebras have been spelled out at the beginning of Section~\ref{Sec:Def of alg. grps}.

\subsection{Construction of the Bruhat--Tits tree}\label{SSec:B-T Tree SL_2(D)}
The aim of this section is to give a streamlined definition of the Bruhat--Tits tree associated with $ \SL_{2}(D) $, together with the action on it. As outlined in the introduction, our definition of the Bruhat--Tits tree follows \cite{BrTi1}*{§7}.

In order to be as efficient as possible, we only describe concretely the objects needed, and give unmotivated definitions. Our description is easily obtained from the explicit description given in \cite{BrTi1}*{§10}, and we give in Appendix~\ref{App:A} more details about the connection with \cite{BrTi1}.

Recall from Definition~\ref{Def:buildingassociated to a data} that the Bruhat--Tits tree $ \mathcal{I} $ of $ \SL_2(D) $ should be isomorphic to $ \SL_2(D)\times \mathbf{R}/\sim $. For $ x\in \mathbf{R} $, we define a group $ P_{x}\leq \SL_2(D) $ which will eventually turn out to be the stabiliser of $ [(\Id,x)]\in \mathcal{I} $ (see Remark~\ref{Rem:metric on I and stabiliser}).
\begin{definition}\label{Def:valuation of a matrix}
Let $ D $ be any valued division algebra with valuation $ \omega $, and let $ g $ be a $ n\times n $ matrix with coefficients in $ D $. Given a $ n\times n $ matrix $ m $ with coefficients in $ \mathbf{R} $, we say that $ g $ has a valuation greater than $ m $ if $ \omega (g_{ij}) \geq m_{ij} $ (for all $ i,j \in \lbrace 1,\dots,n\rbrace $), and we denote it by $ \omega (g) \geq m $.
\end{definition}
\begin{definition}\label{Def:stabiliserforSL2}
For $ x\in \R $, we define
\begin{equation*}
P_{x} = \lbrace g\in \SL_{2}(D)~\vert~\omega (g)\geq \begin{psmallmatrix}
0 & -x \\ 
x & 0
\end{psmallmatrix}\rbrace
\end{equation*}
\end{definition}
\begin{definition}\label{Def:sbgrp N for SL2}
Consider the following subsets
\begin{enumerate}[$ \bullet $]
\item $ T = \lbrace \begin{psmallmatrix}
x & 0 \\ 
0 & x^{-1} 
\end{psmallmatrix}~\vert~ x\in D^{\times}\rbrace < \SL_2 (D) $
\item $ M =\lbrace \begin{psmallmatrix}
0 & -x \\ 
x^{-1} & 0 
\end{psmallmatrix}~\vert~ x\in D^{\times}\rbrace \subset \SL_2 (D) $
\end{enumerate}
and let $ N = T\sqcup M $.
\end{definition}

\begin{definition}\label{Def:affineactionforNSL2}
Let $ \nu\colon N\to \Aff (\R) $ be defined as follows: for $ m = \begin{psmallmatrix}
0 & -x \\ 
x^{-1} & 0
\end{psmallmatrix}\in M $, $ \nu (m) $ is the reflection through $ -\omega (x) $, while for $ t=\begin{psmallmatrix}
x & 0 \\ 
0 & x^{-1}
\end{psmallmatrix}\in T $, $ \nu (t) $ is the translation by $ -2\omega (x) $.
\end{definition}

Then the Bruhat--Tits tree $ \mathcal{I} $ of $ \SL_2(D) $ is the one obtained by applying Definition~\ref{Def:buildingassociated to a data} to the collection of subgroups $ \lbrace (P_x)_{x\in \R}, N \rbrace $ appearing in Definition~\ref{Def:stabiliserforSL2} and Definition~\ref{Def:sbgrp N for SL2}, together with the homomorphism $ \nu \colon N\to \Aff (\R) $ of Definition~\ref{Def:affineactionforNSL2}. We discuss in Appendix~\ref{App:A} why our groups $ P_x $ and $ N $ coincide with the groups $ \hat{P}_{x} $ and $ N $ appearing in the definition of the Bruhat--Tits building in \cite{BrTi1}*{7.4.1 and 7.4.2}. We also check in Appendix~\ref{App:A} that $ \nu \colon N\to \Aff(\R) $ coincides with \cite{BrTi1}. Hence, the given data is indeed obtained from a valued root datum of rank one on $ G $, so that the above construction does indeed give rise to the Bruhat--Tits tree of $ \SL_2(D) $.
\begin{remark}\label{Rem:dependence on D of the tree}
Note that the construction of the Bruhat--Tits tree of $ \SL_2(D) $ depends on $ D $. When needed, we keep track of this dependence by adding the subscript $ D $ to the objects involved. This gives rise to the notations $ (P_x)_D $, $ T_D $, $ M_D $, $ N_D $, $ \nu_D $ and $ \mathcal{I}_D $.
\end{remark}
\begin{remark}\label{Rem:regularity of the B-T tree SL2}
The Bruhat--Tits tree of $ \SL_2(D) $ is actually the regular tree of degree $ \vert \overline{D}\vert+1 $. Indeed, this follows from the fact that our definition of $ \mathcal{I} $ agrees with the one given in \cite{BrTi1}*{7.4.1 and 7.4.2}, and from the tables in \cite{Tits77}*{4.2 and 4.3}.
\end{remark}

\subsection{Local model of the Bruhat--Tits tree}
We now aim to give a local description of balls of the Bruhat--Tits tree, together with the group action on it. Recall that the ball of radius $ 1 $ around $ [(\Id ,0)]\in \mathcal{I} $ (together with the action of $ P_{0} $ on it), is in some sense encoded in $ P_{0} $ considered over the residue field, i.e. over $ \mathcal{O}_{K}/\mathfrak{m}_{K} $ (or more precisely in the reductive quotient of this group, see \cite{BrTi2}*{Théorème~4.6.33} for a precise meaning). It is then natural to think that more generally, the ball of radius $ r $ around $ [(\Id ,0)]\in \mathcal{I} $ (together with the action of $ P_{0} $ on it) is encoded in $ P_{0} $ considered over the ring $ \mathcal{O}_{K}/\mathfrak{m}_{K}^{r} $. We give all the definitions in this section, and we then prove in the next section that those definitions behave as expected (see Theorem~\ref{Thm:localdescriptionoftheball}).

Our techniques only allow to describe balls or radius $ rd $, where $ d $ is the degree of $ D $ over $ K $ (note that $ d=1 $ when $ D=K $). We just mimic the definition of the Bruhat--Tits tree, except that the coefficients of all groups under consideration are now taken in the ring $ \mathcal{O}_{D}/\mathfrak{m}_{D}^{rd} $. All groups defined in this section are adorned by the superscript $ 0,rd $ to reflect the fact that they are local version around $ 0 $ of radius $ rd $.

Let us introduce some notations for the following definition of local stabilisers. In the rest of the paper, $ d $ denotes the degree of $ D $ over its centre $ K $. The valuation $ \omega $ on $ D $ induces a well-defined map on $ \mathcal{O}_{D}/\mathfrak{m}_{D}^{rd} $, that we still denote $ \omega $. Finally, recall that $ \pi_D $ denotes a uniformiser of $ D $.

\begin{definition}\label{Def:localdataSL}
Let $ r\in \N $ and $ x\in [-\omega (\pi_{D}^{rd}),\omega (\pi_{D}^{rd})] $. We set
$$ P_{x}^{0,rd} = \lbrace g\in \SL_{2}(\mathcal{O}_{D}/\mathfrak{m}_{D}^{rd})~\vert~ \omega (g)\geq \begin{psmallmatrix}
0 & -x \\
x & 0
\end{psmallmatrix} \rbrace. $$
\end{definition}
\begin{remark}
See Definition~\ref{Def:SL_2(O_D/m^r)} for the definition of $ \SL_{2}(\mathcal{O}_{D}/\mathfrak{m}_{D}^{rd}) $. When $ D=K $, we obtain the group $ \SL_2(\mathcal{O}_{K}/\mathfrak{m}_{K}^{r}) $ in its usual meaning, i.e. the group of $ 2\times 2 $ matrices with coefficients in $ \mathcal{O}_{K}/\mathfrak{m}_{K}^{r} $ having determinant $ 1 $. We remark that Definition~\ref{Def:localdataSL} parallels Definition~\ref{Def:stabiliserforSL2}.
\end{remark}

We also need the local version of the subgroup $ N $.
\begin{definition}
We define
\begin{enumerate}[$ \bullet $]
\item $ H^{0,rd} = \lbrace \begin{psmallmatrix}
x & 0 \\
0 & x^{-1}
\end{psmallmatrix}\in \SL_{2}(\mathcal{O}_{D}/\mathfrak{m}_{D}^{rd})~\vert~ \omega (x) = 0 \rbrace $
\item $ M^{0,rd} = \lbrace \begin{psmallmatrix}
0 & -x \\
x^{-1} & 0
\end{psmallmatrix}\in \SL_{2}(\mathcal{O}_{D}/\mathfrak{m}_{D}^{rd})~\vert~ \omega (x) = 0 \rbrace $
\end{enumerate}
And we set $ N^{0,rd} = H^{0,rd}\sqcup M^{0,rd} $
\end{definition}

We can also easily define an action of $ N^{0,rd} $ by affine isometries on $ \R $.
\begin{definition}
We let $ H^{0,rd} $ act trivially on $ \R $, and we let all elements of $ M^{0,rd} $ act as a reflection through $ 0\in \R $. This gives an affine action of $ N^{0,rd} $ on $ \R $, and we denote again the resulting map $ N^{0,rd}\to \Aff (\R) $ by $ \nu $.
\end{definition}

We are now able to give a definition of the ball of radius $ rd $ around $ [(\Id ,0)]\in \mathcal{I} $ which only depends on the ring $ \mathcal{O}_D/\mathfrak{m}_D^{rd} $, and not on the whole division algebra $ D $.
\begin{definition}\label{Def:local tree of radius r SL2}
Let $ r\in \N $. We define an $ rd $-local equivalence on $ P_{0}^{0,rd}\times [-\omega (\pi_D^{rd}), \omega (\pi_D^{rd})] $ as follows. For $ g,h\in P_{0}^{0,rd} $ and $ x,y\in [-\omega (\pi_D^{rd}), \omega (\pi_D^{rd})] $
\begin{equation*}
(g,x)\sim_{0,rd}(h,y) \Leftrightarrow \textrm{ there exists } n\in N^{0,rd} \textrm{ such that } \nu (n)(x)=y \textrm{ and } g^{-1}hn \in P_{x}^{0,rd}
\end{equation*}
The resulting space $ \mathcal{I}^{0,rd} = P_{0}^{0,rd}\times [-\omega (\pi_D^{rd}), \omega (\pi_D^{rd})]/\sim_{0,rd} $ is called the local Bruhat--Tits tree of radius $ rd $ around $ 0 $, and $ [(g,x)]^{0,rd} $ stands for the equivalence class of $ (g,x) $ in $ \mathcal{I}^{0,rd} $. The group $ P_{0}^{0,rd} $ acts on $ \mathcal{I}^{0,rd} $ by multiplication on the first component.
\end{definition}

\begin{remark}\label{Rem:dependence on D of the local tree}
Note that the construction of the local Bruhat--Tits tree of $ \SL_2(D) $ depends on $ D $. When needed, we keep track of this dependence by adding the subscript $ D $ to the objects involved. This gives rise to the notations $ (P_x^{0,rd})_D $, $ H^{0,rd}_D $, $ M^{0,rd}_D $, $ N^{0,rd}_D $ and $ \mathcal{I}^{0,rd}_D $.
\end{remark}
\begin{remark}\label{Rem:equivalent def with for all n for local model}
Note that as for Definition~\ref{Def:buildingassociated to a data}, it is equivalent to say that $ (g,x)\sim_{0,rd} (h,y) $ if and only if for all $ \tilde{n}\in N^{0,rd} $ such that $ \nu (\tilde{n})(x)=y $, we have $ g^{-1}h\tilde{n}\in P_{x}^{0,rd} $. Indeed, if there exists $ n\in N^{0,rd} $ such that $ \nu (n)(x) = y $ and $ g^{-1}hn\in P_{x}^{0,rd} $, let $ \tilde{n} $ be any element of $ N^{0,rd} $ such that $ \nu (\tilde{n})(x) = y $. We have $ g^{-1}h\tilde{n} = g^{-1}hnn^{-1}\tilde{n} $, and a case-by-case analysis shows that $ n^{-1}\tilde{n}\in P_{x}^{0,rd} $. Hence $ g^{-1}hnn^{-1}\tilde{n} $ belongs to $ P_{x}^{0,rd} $ as well, as wanted.
\end{remark}

\subsection{Integral model}
We have just defined the space $ \mathcal{I}^{0,rd} $, (recall that throughout this section, $ d $ is the degree of $ D $). In order to show that it encodes the ball of radius $ rd $ together with the action of $ P_0 $ on it (as will be done in Theorem~\ref{Thm:localdescriptionoftheball}), we need to prove that the projection $ \mathcal{O}_{D}\to \mathcal{O}_{D}/\mathfrak{m}_{D}^{rd} $ induces a \emph{surjective} homomorphism $ P_{0}\to P_{0}^{0,rd} $.

We solve this problem by defining a smooth $ \mathcal{O}_{K} $-scheme $ \underline{\SL}_{2,D} $ such that $ \underline{\SL}_{2,D}(\mathcal{O}_{K}) \cong P_{0} $ and $ \underline{\SL}_{2,D}(\mathcal{O}_{K}/\mathfrak{m}_{K}^{r})\cong P_{0}^{0,rd} $. Then the desired surjectivity follows by an application of Hensel's lemma for smooth schemes (that we recall in Theorem~\ref{Thm:genhensel}).

This smooth $ \mathcal{O}_{K} $-scheme is in fact the Bruhat--Tits integral model $ \hat{\mathfrak{G}}_{\varphi} $ associated with a standard valuation $ \varphi $ (see \cite{BrTi2}*{4.6.26}), and in this case, it is just the straightforward model one would consider.

\begin{definition}\label{Def:integral model for SL_2}
When $ D=K $, the integral model $ \underline{\SL}_{2,D} $ is the group $ \SL_{2} $ considered over $ \mathcal{O}_{K} $. Concretely, this is the $ \mathcal{O}_{K} $-scheme which is the spectrum of the $ \mathcal{O}_{K} $-algebra $ \mathcal{O}_{K}[\underline{\SL}_{2}] = \mathcal{O}_{K}[X_{11},X_{12},X_{21},X_{22}]/(X_{11}X_{22}-X_{12}X_{21}-1) $. In the case of a central division algebra of degree $ d>1 $ over $ K $, the integral model $ \underline{\SL}_{2,D} $ over $ \mathcal{O}_{K} $ is defined in the appendix (see Definition~\ref{Def:integral model for SL_2(D)}).
\end{definition}

\begin{theorem}\label{Thm:smoothness of SL2}
$ \underline{\SL}_{2,D} $ is a smooth $ \mathcal{O}_{K} $-scheme.
\end{theorem}
\begin{proof}
When $ D=K $, smoothness of $ \underline{\SL}_{2,D} $ over $ \mathcal{O}_{K} $ (and in fact of the algebraic group $ \SL_n $ over any ring) is easily checked using the infinitesimal lifting criterion (see \cite{stacks-project}*{Tag 02H6}). The case of an arbitrary $ D $ is relegated to the appendix (see Theorem~\ref{Thm:smmoothness for SL_2(D)}).
\end{proof}
We now spell out what the group $ \underline{\SL}_{2,D}(\mathcal{O}_{K}/\mathfrak{m}_{K}^{r}) $ is, along with the homomorphism $ p_{rd}\colon P_{0}\to P_{0}^{0,rd} $.

\begin{lemma}\label{Lem:description of the restriction P_0 to P_0^0,rSL2}
$ \underline{\SL}_{2,D}(\mathcal{O}_{K}) \cong P_{0} $ and $ \underline{\SL}_{2,D}(\mathcal{O}_{K}/\mathfrak{m}_{K}^{r}) \cong P_{0}^{0,rd} $. Following the identifications
\begin{center}
\begin{tikzpicture}[->]
 \node (1) at (0,0) {$ \underline{\SL}_{2,D}(\mathcal{O}_{K}) $};
 \node (2) at (2.4,0) {$ \SL_{2}(\mathcal{O}_{D}) = P_{0} $};
 \node (5) at (1,0) {$ \cong $};
 \node (3) at (-0.4,-1.2) {$ \underline{\SL}_{2,D}(\mathcal{O}_{K}/\mathfrak{m}_{K}^{r}) $};
 \node (4) at (3, -1.2) {$ \SL_{2}(\mathcal{O}_{D}/\mathfrak{m}_{D}^{rd}) = P_{0}^{0,rd} $};
 \node (6) at (1,-1.2) {$ \cong $};
  
 \draw[->] (0,-0.3) to (0,-0.9);
 \draw[->] (1.9,-0.3) to (1.9,-0.9);
\end{tikzpicture}
\end{center}
the homomorphism $ p_{rd}\colon P_{0}\to P_{0}^{0,rd} $ is the one induced by the projection of the coefficients $ \mathcal{O}_{D}\to \mathcal{O}_{D}/\mathfrak{m}_{D}^{rd} $.
\end{lemma}
\begin{proof}
When $ D=K $, by definition, $ \underline{\SL}_2(\mathcal{O}_{K}) = \Mor_{\mathcal{O}_{K}}(\mathcal{O}_{K}[\underline{\SL}_{2}],\mathcal{O}_{K}) $, which is clearly isomorphic to $ \SL_{2}(\mathcal{O}_{K}) $. Furthermore, $ \underline{\SL}_{2,D}(\mathcal{O}_{K}/\mathfrak{m}_{K}^{r}) = \Mor_{\mathcal{O}_{K}}(\mathcal{O}_{K}[\underline{\SL}_{2,D}],\mathcal{O}_{K}/\mathfrak{m}_{K}^{r}) $, which is clearly isomorphic to $ \SL_{2}(\mathcal{O}_{K}/\mathfrak{m}_{K}^{r}) $, as wanted. The general case is treated in the appendix (see Lemma~\ref{Lem:rational points for SL_2(D)}).
\end{proof}

The fact that $ \underline{\SL}_{2,D} $ is a smooth scheme over $ \mathcal{O}_{K} $ allows us to deduce the surjectivity of $ P_{0}\to P_{0}^{0,rd} $. For this, we use a well-known generalised version of Hensel's lemma for smooth schemes, that we now recall.
\begin{theorem}[Hensel's lemma for smooth schemes]\label{Thm:genhensel}
Let $ X $ be a smooth $ \mathcal{O}_{K} $-scheme, and let $ r\in \N $. The map $ X(\mathcal{O}_{K})\to X(\mathcal{O}_{K}/\mathfrak{m}_{K}^{r}) $ is surjective.
\end{theorem}
\begin{proof}
For $ r=1 $, this is \cite{EGA4}*{Théorème~18.5.17}. For $ r>1 $, note that as remarked below \cite{EGA4}*{Définition~18.5.5}, $ (S,S_{0}) $ is a Henselian couple if and only if $ (S_{red},(S_{0})_{red}) $ is so. We deduce that $ (\Spec \mathcal{O}_{K},\Spec \mathcal{O}_{K}/\mathfrak{m}_{K}^{r}) $ is a Henselian couple. 
Thus the proof of Théorème 18.5.17 applies verbatim to our situation, upon making one change: replace the reference to $ 18.5.11(b) $ to a reference to $ 18.5.4(b) $ (taking $ S=\Spec \mathcal{O}_{K} $ and $ S_{0}=\Spec \mathcal{O}_{K}/\mathfrak{m}_{K}^{r} $ in the notations of $ 18.5.4 $).
\end{proof}

\begin{corollary}\label{Cor:surjectivityforSL2}
The homomorphism $ p_{rd}\colon P_{0}\to P_{0}^{0,rd} $ is surjective, for all $ r\in \N $.
\end{corollary}
\begin{proof}
This is a direct consequence of the commutative square involving $ P_{0}\to P_{0}^{0,rd} $ given in Lemma~\ref{Lem:description of the restriction P_0 to P_0^0,rSL2}, together with the fact that the integral model is smooth by Theorem~\ref{Thm:smoothness of SL2}, so that Theorem~\ref{Thm:genhensel} applies to the left hand side of the diagram.
\end{proof}

Along with the surjectivity of the restriction map $ p_{rd}\colon P_{0}\to P_{0}^{0,rd} $, one of the key result in our local description of the ball of radius $ rd $ is that $ p_{rd} $ is also somehow injective enough. This result can be seen as a natural generalisation of \cite{BrTi2}*{Corollaire~4.6.8}. 
\begin{lemma}\label{Lem:injectivityofpr}
Let $ r\in \N $ and let $ x\in [-\omega (\pi_{D}^{rd}),\omega (\pi_{D}^{rd})] $. Then $ p_{rd}^{-1}(P_{x}^{0,rd}) \subset P_{x} $.
\end{lemma}
\begin{proof}
Belonging to $ p_{rd}^{-1}(P_x^{0,rd}) $ implies that the valuation of the off diagonal entries are big enough. Hence, the result follows directly from Definition~\ref{Def:stabiliserforSL2}.
\end{proof}

We finally arrive at our main result: the ball of radius $ rd $ together with the action of $ \SL_{2}(\mathcal{O}_{D}) $ is encoded in $ P_{0}^{0,rd} $. We first need an adequate description of the ball of radius $ rd $ around $0$ in $ \mathcal{I} $.

\begin{lemma}\label{Lem:B_0(r) is really the ball of radius r}
Renormalise the distance on $ \R $ so that $ d_{\R}(0;\omega(\pi_D)) = 1 $, and put the metric $ d_{\mathcal{I}} $ on $ \mathcal{I} $ arising from the distance $ d_{\R} $ $ ( $see Remark~\ref{Rem:metric on I and stabiliser}$ ) $. Let $ B_0(rd) = \lbrace p\in \mathcal{I}~\vert~d_{\mathcal{I}}([(\Id ,0)];p)\leq rd\rbrace $ be the ball of radius $ rd $ around $ 0 $ in $ \mathcal{I} $. Let $ \tilde{B}_0(rd) = \lbrace [(g,x)]\in \mathcal{I} ~\vert~ g\in P_{0}, x\in [-\omega (\pi_D^{rd}), \omega (\pi_D^{rd})]\subset \R \rbrace $. Then $ B_0(rd) = \tilde{B}_0(rd) $.
\end{lemma}
\begin{proof}
If $ [(g,x)]\in \tilde{B}_0(rd) $, then $ d_\R(0;x)\leq rd $ by our normalisation of the distance on $ \R $. So we get $ rd\geq d_\R(0;x) = d_{\mathcal{I}}([(\Id ,0)];[(\Id ,x)]) = d_{\mathcal{I}}([(\Id ,0)];[(g ,x)]) $, where the last equality follows from the fact that $ G $ acts by isometries on $ \mathcal{I} $. Conversely, assume $ d_{\mathcal{I}}([(\Id ,0)];[(g,x)])\leq rd $. By Lemma~\ref{Lem:B-T tree from a based point}, there exist $ h\in P_0 $ and $ y\in \R $ such that $ [(g,x)] = [(h,y)] $. But $ d_{\mathcal{I}}([(\Id ,0)];[(g,x)]) = d_{\mathcal{I}}([(\Id ,0)];[(h ,y)]) = d_{\mathcal{I}}([(\Id ,0)];[(\Id ,y)]) = d_\R(0;y) $. Hence $ [(h,y)]\in \tilde{B}_0(rd) $, as wanted.
\end{proof}
\begin{remark}
The distance $ d_{\mathcal{I}} $ that we introduced in Lemma~\ref{Lem:B_0(r) is really the ball of radius r} is also the combinatorial distance on the tree. Indeed, looking at when $ P_y $ is inside $ P_x $ for $ x,y\in \R $, we see that $ [(\Id ,x)] $ is a vertex of $ \mathcal{I} $ if and only if $ x\in \omega(\pi_D)\Z $ (note that this argument uses the fact that a simple algebraic group acts on a tree without edge inversion).
\end{remark}

\begin{theorem}\label{Thm:localdescriptionoftheball}
Let $ r\in \N $. The map $ B_{0}(rd)\to \mathcal{I}^{0,rd}\colon [(g,x)]\mapsto [(p_{rd}(g),x)]^{0,rd} $ is a $ (p_{rd}\colon P_{0}\to P_{0}^{0,rd}) $-equivariant bijection.
\end{theorem}
\begin{proof}
The map is well-defined by Lemma~\ref{Lem:integralityofn}.
\begin{itemize}
\item Injectivity: let $ [(g,x)], [(h,y)] \in B_0(rd) $ be such that they have the same image in $ \mathcal{I}^{0,rd} $. By Remark~\ref{Rem:equivalent def with for all n for local model}, it means that for all $ \tilde{n}\in N^{0,rd} $ such that $ \nu (\tilde{n})(x) = y $, $ p_{rd}(g)^{-1}p_{rd}(h)\tilde{n}\in P_{x}^{0,rd} $. So, we can assume that $ \tilde{n} $ is either equal to $ \Id $, or is of the form $ \begin{psmallmatrix} 
0 & 1 \\
-1 & 0 
\end{psmallmatrix} $. Hence, there exists $ n\in N $ such that $ p_{rd}(n) = \tilde{n} $. But $ \nu (n)(x) = y $, and $ g^{-1}hn\in p_{rd}^{-1}(P_{x}^{0,rd}) \subset P_{x} $ by Lemma~\ref{Lem:injectivityofpr}. Hence, $ [(g,x)] = [(h,y)] $, as wanted.
\item Surjectivity: follows directly from the surjectivity of $ p_{rd}\colon P_{0}\to P_{0}^{0,rd} $ (Corollary~\ref{Cor:surjectivityforSL2}).
\item Equivariance: $ h.[(g,x)] = [(hg,x)] \mapsto [(p_{rd}(hg),x)]^{0,rd} = p_{rd}(h).[(p_{rd}(g),x)]^{0,rd} $. \qedhere
\end{itemize}
\end{proof}

\begin{remark}
Theorem~\ref{Thm:localdescriptionoftheball} shows that the ball of radius $rd$ in the Bruhat--Tits tree of $\SL_2(D)$ only depends on $ \mathcal{O}_D/\mathfrak{m}_D^{rd} $. This result alone could be obtained more easily using the description of Bruhat--Tits buildings as (admissible) lattices, see \cite{AN02}. Using precisely this strategy, the fact that the ball of radius $rd$ in the Bruhat--Tits tree of $\SL_2(D)$ only depends on $ \mathcal{O}_D/\mathfrak{m}_D^{rd} $ is obtained by M.\ de la Salle and R.\ Tessera in \cite{dlST15}*{Corollary~2.2}. However, for our purpose, we need to control the local action, and in particular we crucially rely on Corollary~\ref{Cor:surjectivityforSL2} to study Chabauty convergences (see the proof of Theorem~\ref{Thm:continuity from D to Chab(Aut(T))}). Hence, we believe that an approach relying on lattices would not spare the need for a smooth integral model.
\end{remark}

\subsection{Arithmetic convergence}
\begin{definition}\label{Def:the space D}
Let $ \mathcal{K} $ be the set of local fields up to isomorphism. Let $ \mathcal{D} $ be the set of finite dimensional division algebras $ D $ over their centre $ Z(D) $ and such that $ Z(D) $ is a local field. We also consider this set up to isomorphism. We set $ \mathcal{K}_{p^n} = \lbrace K\in \mathcal{K}~\vert~\vert \overline{K}\vert =p^n \rbrace $ and $ \mathcal{D}_{p^n} = \lbrace D\in \mathcal{D}~\vert~\vert \overline{D}\vert =p^n \rbrace $.
\end{definition}

Note that $ \mathcal{K} $ can naturally be seen as a subset of $ \mathcal{D} $, so that $ \mathcal{K}_{p^n} \subset \mathcal{D}_{p^n} $. Following an idea dating back to Krasner (see \cite{D84} for references, this idea is also used in e.g. \cite{Ka86}), we define a metric on the space $ \mathcal{D} $.
\begin{definition}
Let $ D_1, D_2 \in \mathcal{D} $. We say that $ D_1 $ is $ r $-close to $ D_2 $ if and only if there exists an isomorphism $ \mathcal{O}_{D_{1}}/\mathfrak{m}_{D_1}^{r}\cong \mathcal{O}_{D_{2}}/\mathfrak{m}_{D_2}^{r} $.
\end{definition}

\begin{remark}\label{Rem:transitivityofdistance}
Note that being $ r $-close is an equivalence relation, and that if $ r\geq l $ and $ D_1 $ is $ r $-close to $ D_2 $, then $ D_1$ is $ l $-close to $ D_2 $.
\end{remark}

Observe that this notion of closeness induces a non-archimedean metric on $ \mathcal{D} $. Let
\begin{equation*}
d\colon \mathcal{D}\times \mathcal{D}\to \mathbf{R}_{\geq 0}\colon d(D_{1};D_{2})= \inf \lbrace \frac{1}{2^{r}} ~\vert~ D_{1} \textrm{ is } r \textrm{-close to }D_{2} \rbrace
\end{equation*}

\begin{lemma}\label{Lem:Non-archim. metric space}
$ d(\cdot~{;}~\cdot ) $ is a non-archimedean metric on $ \mathcal{D} $. 
\end{lemma}
\begin{proof}
If $ d(D_1;D_2) = 0 $, then $ \mathcal{O}_{D_1} $ and $ \mathcal{O}_{D_2} $ are isomorphic. 
Hence, their field of fraction are isomorphic, so that $ D_1=D_2 $ in $ \mathcal{D} $, as wanted. The fact that this distance is non-archimedean is a consequence of Remark~\ref{Rem:transitivityofdistance}.
\end{proof}

A crucial fact about the space $ \mathcal{D}_{p^{n}} $ (for a fixed prime power $ p^{n} $, as in Definition~\ref{Def:the space D}) is that it is a compact space. This is one of the key observation to prove that $\mathcal{S}_T^{\SL_2(D)} $ is closed in $\mathcal{S}_T $. In fact, it is even possible to give an explicit description of the metric space $ \mathcal{D}_{p^{n}} $. The corner stone in this description is Theorem~\ref{Thm:fundamental approximation lemma} which is certainly well known to experts (this is for example used implicitly in \cite{Ka86}). While working on this paper, we learnt that it had also been obtained and used independently in \cite{dlST15}*{Lemma~1.3}. Given its importance, we decide nevertheless to include our own proof.

\begin{theorem}\label{Thm:fundamental approximation lemma}
Let $ K $ be a totally ramified extension of degree $ k $ of $ \mathbf{Q}_{p^{n}} $. The distance between $ K $ and $ \mathbf{F}_{p^{n}}(\!(X)\!) $ is $ \frac{1}{2^{k}} $. More explicitly, let $ \lbrace a_{x}\rbrace_{x\in \mathbf{F}_{p^{n}}}\subset \mathcal{O}_K\cap \mathbf{Q}_{p^{n}} $ be a set of representative of $ \overline{K} $. Then the bijection 
\begin{align*}
\varphi_{\pi_{K}}\colon \mathcal{O}_{K}&\to \mathbf{F}_{p^{n}}[\![X]\!] \\
\sum \limits_{i=0}^{\infty}a_{x_{i}}\pi_{K}^{i}&\mapsto \sum \limits_{i=0}^{\infty}x_{i}X^{i}
\end{align*}
$ ( $which depends on a choice of uniformiser of $ K) $ induces an isomorphism of rings
\begin{align*}
\overline{\varphi}_{\pi_{K}}\colon \mathcal{O}_{K}/\mathfrak{m}_{K}^{k}\to \mathbf{F}_{p^{n}}[\![X]\!]/(X^{k})
\end{align*}
 \end{theorem}
\begin{proof}
Let $ \lbrace a_{x}\rbrace_{x\in \mathbf{F}_{p^{n}}}\subset \mathcal{O}_K $ be a set of representative of $ \overline{K} $. Since $ \mathbf{Q}_{p^{n}}\leq K $ is totally ramified, we can and do choose the $ a_x $'s so that they all lie in $ \mathbf{Q}_{p^{n}} $. Now, we have $ a_{x}+a_{y} - a_{x+y} \in (p) $ and $ a_{x}a_{y} - a_{xy} \in (p) $. Furthermore, since $ K $ is totally ramified, $ (p) = \mathfrak{m}_{K}^{k} $. Hence, this implies that the map $ \varphi_{\pi_K} $ (which is always a bijection, by the general theory of local fields) is a homomorphism modulo $ \mathfrak{m}_{K}^{k} $ and $ (X^{k}) $.

To conclude that $ K $ and $ \mathbf{F}_{p^{n}}(\!(X)\!) $ are at distance $ \frac{1}{2^{k}} $, it suffices to observe that $ \mathcal{O}_{K}/\mathfrak{m}_{K}^{k+1} $ is not isomorphic to $ \mathbf{F}_{p^{n}}[\![X]\!]/(X^{k+1}) $. But this is clear, since $ p\notin \mathfrak{m}_{K}^{k+1} $, hence $ \sum \limits_{i=1}^{p}1 \neq 0 $ in $ \mathcal{O}_{K}/\mathfrak{m}_{K}^{k+1} $.
\end{proof}

We need to transpose the situation of Theorem~\ref{Thm:fundamental approximation lemma} to division algebras.
\begin{lemma}\label{Lem:division algebras are accumulation points}
Let $ D_1\in \mathcal{D} $ be the cyclic algebra $ (E_1/K_1,\sigma^{r_1},\pi_{K_1}) $ $ ( $respectively $ D_2\in \mathcal{D} $ be the cyclic algebra $ (E_2/K_2,\sigma^{r_2},\pi_{K_2})) $. Assume that $ D_1 $ and $ D_2 $ are of the same degree $ d $ over their respective centre, that $ r_1 = r_2=r\in (\Z /d\Z)^{\times} $, and that $ K_1 $ is $ e $-close to $ K_2 $. Then $ D_1 $ is $ ed $-close to $ D_2 $
\end{lemma}
\begin{proof}
Note that the isomorphism type of $ D_i $ does not depend on the choice of the uniformiser $ \pi_{K_i} $, so that we can and do assume that the given isomorphism $ \mathcal{O}_{K_1}/\mathfrak{m}_{K_1}^{e}\cong \mathcal{O}_{K_2}/\mathfrak{m}_{K_2}^{e} $ maps $ \pi_{K_1} $ to $ \pi_{K_2} $. The (non-commutative) ring $ \mathcal{O}_{D_i}/\mathfrak{m}_{D_i}^{ed} $ is actually isomorphic to the cyclic algebra $ ((\mathcal{O}_{E_i}/\mathfrak{m}_{E_i}^{e})/(\mathcal{O}_{K_i}/\mathfrak{m}_{K_i}^{e}),\sigma^r, \pi_{K_i}) $, so that the result follows.
\end{proof}

It is then quite straightforward to work out the homeomorphism type of $ \mathcal{D}_{p^{n}} $.
As in the introduction, let $ \hat{\N} $ denote the one point compactification of $ \N $.
\begin{proposition}\label{Prop:explicit description of D}
Let $ p $ be a prime number. Then $ \mathcal{D}_{p^{n}} $ is homeomorphic to $ \hat{\N}\times \lbrace 1,2,\dots ,n\rbrace $, and $ \mathcal{K}_{p^n}\subset \mathcal{D}_{p^{n}} $ is a clopen subset homeomorphic to $ \hat{\N} $.
\end{proposition}

\begin{proof} \setcounter{claim}{0}
$  $
\begin{claim}\label{Claim:1'}
Let $ K $ be a local field. If $ \vert \overline{K}\vert = p^{n} $, then $ K $ is a totally ramified extension of $ \mathbf{Q}_{p^{n}} $, or it is isomorphic to $ \mathbf{F}_{p^{n}}(\!(X)\!) $.
\end{claim}

\begin{claimproof}
By the classification of local fields, $ K $ is either a finite extension of $ \mathbf{Q}_{p} $, or isomorphic to $ \mathbf{F}_{p^{n}}(\!(X)\!) $ for some prime power $ p^n $. Since $ \overline{\mathbf{F}_{p^{n}}(\!(X)\!)} = \mathbf{F}_{p^{n}} $, the latter case is clear. For the first case, $ \overline{K} = \mathbf{F}_{p^{n}} $ if and only if the maximal unramified subextension of $ K $ is $ \mathbf{Q}_{p^{n}} $.
\end{claimproof}

\medskip

\begin{claim}\label{Claim:2'}
Let $ K_k $ and $ K_l $ be totally ramified extension of $ \mathbf{Q}_{p^{n}} $ such that $ [K_k:\mathbf{Q}_{p^{n}}] = k < [K_l:\mathbf{Q}_{p^{n}}] = l $. Then $ d(K_k;K_l) = \frac{1}{2^{k}} $.
\end{claim}

\begin{claimproof}
We observed in Lemma~\ref{Lem:Non-archim. metric space} that $ \mathcal{D} $ is a non-archimedean metric space, and hence every triangle is isosceles. Thus, the distance between $ K_k $ and $ K_l $ is either $ \frac{1}{2^{k}} $ or $ \frac{1}{2^{l}} $ (taking in each case $ \mathbf{F}_{p^n}(\!(T)\!) $ as a comparison point, and using Theorem~\ref{Thm:fundamental approximation lemma}). But in the latter case, since being $ l $-close is an equivalence relation, we would conclude that $ K_k $ is $ l $-close to $ \mathbf{F}_{p^{n}}(\!(X)\!) $, which would contradict Theorem~\ref{Thm:fundamental approximation lemma}.
\end{claimproof}

\medskip

\begin{claim}\label{Claim:3'}
There are only finitely many totally ramified extension of degree $ \leq k $ of a local field of characteristic $ 0 $.
\end{claim}

\begin{claimproof}
This is just a well-known corollary of the so called Krasner's Lemma. A proof of Claim~\ref{Claim:3'} can be found in \cite{Lang94}*{Chapter~II, §5, Proposition~14}.
\end{claimproof}

\medskip

\begin{claim}\label{Claim:4'}
Let $ D\in \mathcal{D}_{p^{n}} $. If $ D $ is of characteristic $ 0 $, it is isolated in $ \mathcal{D}_{p^{n}} $.
\end{claim}

\begin{claimproof}
$ D $ is isomorphic to the cyclic algebra $ (E/K,\sigma^{r},\pi_K) $ (see Definition~\ref{Def:cyclic algebra}), where $ [E:K]=d $ divides $ n $, $ r\in (\mathbf{Z}/d\mathbf{Z})^{\times} $ and $ \vert \overline{K}\vert = p^{\frac{n}{d}} $. Let $ D_1 = (E_1/K_1,\sigma^{r_1},\pi_{K_1}) $ (respectively $ D_2 = (E_2/K_2,\sigma^{r_2},\pi_{K_2}) $) be of degree $ d_1 $ (respectively $ d_2 $). Using the explicit description of cyclic algebras, it is easily seen that if $ D_1 $ is $ 2 $-close to $ D_2 $, then $ d_1 = d_2 $, $ \vert \overline{K_1}\vert = \vert \overline{K_2}\vert $ and $ r_1=r_2 $. 
Furthermore, if $ D_1 $ is $ ed_1 $-close to $ D_2 $ for some $ e\in \N $, then $ K_1 $ is $ e $-close to $ K_2 $, since $ \mathcal{O}_{K_i}/(\pi_{K_i}^{e}) $ is the centre of $ \mathcal{O}_{D_i}/(\pi_{D_i}^{ed_i}) $. Hence, the result follows from Claim~\ref{Claim:2'} and Claim~\ref{Claim:3'}.
\end{claimproof}

\medskip

\begin{claim}\label{Claim:5'} 
$ \mathcal{D}_{p^{n}} $ is a countable space.
\end{claim}

\begin{claimproof}
By Claim~\ref{Claim:3'} and the classification of division algebras over local fields, there are only countably many division algebras of characteristic $ 0 $ in $ \mathcal{D}_{p^{n}} $. Furthermore, the number of division algebras of characteristic $ p $ in $ \mathcal{D}_{p^{n}} $ is finite.
\end{claimproof}

\medskip

We are now able to deduce the homeomorphism type of $ \mathcal{D}_{p^{n}} $: division algebras of characteristic~$ 0 $ are isolated by Claim~\ref{Claim:4'}, and every division algebra of positive characteristic is an accumulation point in $ \mathcal{D}_{p^n} $ by Theorem~\ref{Thm:fundamental approximation lemma} and Lemma~\ref{Lem:division algebras are accumulation points}. Hence, by \cite{MS20}*{Théorème~1}, $ \mathcal{D}_{p^n} $ is homeomorphic to $ x $ disjoint copies of $ \hat{\N} $, where $ x $ is the number of division algebras of positive characteristic in $ \mathcal{D}_{p^n} $, i.e. $ x = \sum_{d\vert n}\vert (\mathbf{Z}/d\mathbf{Z})^{\times}\vert = n $. The subspace $ \mathcal{K}_{p^n}\subset \mathcal{D}_{p^n} $ consists of division algebras of degree $ 1 $, and hence accounts for one copy of $ \hat{\N} $.
\end{proof}

\subsection{Continuity from division algebras to subgroups of \texorpdfstring{$ \Aut (T) $}{Aut(T)}}
In this section, we start to vary the division algebra $ D $, and look at the variation it produces on the Bruhat--Tits tree of $ \SL_2(D) $. Recall that we introduced a notation to keep track of the dependence on $ D $ of many of the definitions we made in this section (see Remark~\ref{Rem:dependence on D of the tree} and Remark~\ref{Rem:dependence on D of the local tree}).

\begin{proposition}\label{Prop:aritmimpliesgeomproximityforSL_2(D)}
Let $ D_1 $ and $ D_2 $ be two elements in $ \mathcal{D} $, with respective degree $ d_1 $ and $ d_2 $. Assume that $ D_1 $ is $ rd_1 $-close to $ D_2 $, with $ rd_1 \geq 2 $. Then $ d_1 = d_2 = d $, $ (P_{0}^{0,rd})_{D_1}\cong (P_{0}^{0,rd})_{D_2} $ and $ \mathcal{I}_{D_1}^{0,rd}$ is equivariantly in bijection with $ \mathcal{I}_{D_2}^{0,rd} $.
\end{proposition}
\begin{proof}
The isomorphism $ \mathcal{O}_{D_1}/\mathfrak{m}_{D_1}^{rd}\cong \mathcal{O}_{D_2}/\mathfrak{m}_{D_2}^{rd} $ induces a group isomorphism $ \varphi \colon (P_{0}^{0,rd})_{D_1} = \SL_{2}(\mathcal{O}_{D_1}/\mathfrak{m}_{D_1}^{rd})\cong \SL_{2}(\mathcal{O}_{D_2}/\mathfrak{m}_{D_2}^{rd}) = (P_{0}^{0,rd})_{D_2} $.
Define a linear map $ f\colon \mathbf{R}\to \mathbf{R}\colon x\mapsto x\frac{\omega (\pi_{D_2})}{\omega (\pi_{D_1})} $. It is clear that for all $ x\in [-\omega (\pi_{D_1}^{rd}), \omega (\pi_{D_1}^{rd})] $, $ \varphi $ restricts to an isomorphism $ (P_{x}^{0,rd})_{D_{1}}\cong (P_{f(x)}^{0,rd})_{D_{2}} $. Furthermore, 
\begin{align*}
\varphi (T^{0,rd})_{D_1}&= (T^{0,rd})_{D_2}\\
\varphi (M^{0,rd})_{D_1}&= (M^{0,rd})_{D_2}
\end{align*}
and for all $ n\in N^{0,rd} $, $ f(n.x) = \varphi (n).f(x) $. Hence, the map $ \mathcal{I}_{D_1}^{0,rd}\to \mathcal{I}_{D_2}^{0,rd}\colon [(g,x)]^{0,rd}\mapsto [(\varphi (g), f(x))]^{0,rd} $ is a $ \varphi $-equivariant bijection.
\end{proof}

We can finally go back to our original problem, which is to study convergence of algebraic groups in the Chabauty space of $ \Aut (T) $. We first discuss the homomorphism $ \SL_2(D) \to \Aut (\mathcal{I}_D) $.
\begin{proposition}\label{Prop:Embedding G(K) in Aut(T)}
Let $ \mathcal{I} = \mathcal{I}_D $ be the Bruhat--Tits tree of $ \SL_2(D) $. The homomorphism given by the action of $ \SL_2(D) $ on its Bruhat--Tits building~$ \hat{}~\colon \SL_2(D)\to \Aut (\mathcal{I}) $ is continuous with closed image, and the kernel is equal to the centre of $ \SL_2(D) $. 
\end{proposition}
\begin{proof}
In each case, the group $ P_x $ is really the stabiliser of $ [(\Id ,x)]\in \mathcal{I} $ (see Remark~\ref{Rem:metric on I and stabiliser}). Since a basic identity neighbourhood in $ \Aut (\mathcal{I}) $ is given by intersecting finitely many vertices stabilisers, the continuity follows. The fact that the image is closed follows from the general argument in \cite{BM96}*{Lemma~5.3}. Finally, the kernel can also be seen directly from the explicit description of $ P_{x} $. Indeed, if $ g $ is in the intersection $ \bigcap \limits_{x\in \R} P_x $, then $ g $ is diagonal. But also, the conjugation action of $ g $ on root groups needs to be trivial, so that $ g $ is in the centre of $ \SL_2(D) $. Conversely, the centre of $ \SL_2(D) $ clearly acts trivially on $ \mathcal{I} $, which concludes the proof.
\end{proof}

The convergence is then a more or less direct consequence of Theorem~\ref{Thm:localdescriptionoftheball}.
\begin{theorem}\label{Thm:continuity from D to Chab(Aut(T))}
Let $ (D_i)_{i\in \mathbf{N}} $ be a sequence in $ \mathcal{D} $ which converges to $ D $, and let $G_i = \SL_2(D_i) $ $ ( $respectively $ G = \SL_2(D)) $. For $ N $ big enough and for all $ i\geq N $, there exist isomorphisms $ \mathcal{I}_{D_i} \cong \mathcal{I}_D $ such that the induced embeddings~ $ \hat{G_i}\hookrightarrow \Aut (\mathcal{I}_D) $ make $ (\hat{G_i})_{i\geq N} $ converge to $ \hat{G} $ in the Chabauty topology of $ \Aut (\mathcal{I}_D) $.
\end{theorem}
\begin{remark}
The convergence depends on a choice of specific isomorphisms $ \mathcal{I}_{D_i}\cong \mathcal{I}_D $, or in other words it depends on choosing how $ \hat{G}_{i} $ sits in $ \Aut (\mathcal{I}_D) $. This dependence is not problematic since for two isomorphic closed subgroups $ H,H' $ of $ \Aut (\mathcal{I}_D) $ both acting $ 2 $-transitively on $ \partial \mathcal{I}_D $, there exists $ g $ in the fixator of $ e_0 $ such that $ gHg^{-1} = H' $, where $ e_0 $ is any edge of $ \mathcal{I}_D $ (see \cite{Rad15}*{Proposition~A.1}, and recall also that $ H $ acts transitively on the edges of $ \mathcal{I}_D $). Hence, for other choices of embeddings, the sequence converges to a conjugate of $ \hat{G} $ in $ \Aut (\mathcal{I}_D) $. Recall also that we introduced the space $ \mathcal{S}_T $ in the introduction precisely to avoid this dependence.
\end{remark}
The main step of the proof is to establish that the sequence of stabilisers $ ((\hat{P}_{0})_{D_i})_{i\geq N} $ converges to the stabiliser $ (\hat{P}_{0})_{D} $ in $ \Aut ( \mathcal{I}_D) $. From there, we can conclude that $ (\hat{G}_i)_{i\geq N} $ converges to $ \hat{G} $ from general theory.

\begin{proof}
The Bruhat--Tits tree $ \mathcal{I}_{D_i} $ is the regular tree of degree $ p^n +1 $ if and only if $ D_i $ belongs to $ \mathcal{D}_{p^n} $. Hence there exists $ N $ such that for all $ i\geq N $, $  \mathcal{I}_{D_i}\cong  \mathcal{I}_D $.

Passing to a subsequence, we can assume that $ D_i $ is ($ di $)-close to $ D $, where $ d $ is the degree of $ D $ over its centre. Hence, for $ i\geq 2 $, $ D_i $ is also of degree $ d $ over its centre. We now define an explicit isomorphism $ f_i\colon \mathcal{I}_{D_i}\to \mathcal{I}_{D} $ (for $i\geq 2$) as follows: let $ \mathcal{I}_{D_i}^{0,di}\cong \mathcal{I}_{D}^{0,di} $ be the isomorphism given by Proposition~\ref{Prop:aritmimpliesgeomproximityforSL_2(D)}. By Theorem~\ref{Thm:localdescriptionoftheball}, this gives an isomorphism on balls of radius $ di $: $ \mathcal{I}_{D_i}\supset B_0(di)\cong B_0(di)\subset \mathcal{I}_{D} $ (recall that by Lemma~\ref{Lem:B_0(r) is really the ball of radius r}, $ B_0(di) $ is really the ball of radius $ di $ on the tree $ \mathcal{I}_{D} $). As $ \mathcal{I}_{D_i} $ is a regular tree of the same degree than $ \mathcal{I}_{D} $, we can extend this isomorphism of balls to an isomorphism $ f_{i}\colon \mathcal{I}_{D_i}\to \mathcal{I}_{D} $ (this extension is of course not unique, but we choose one such). By means of $ f_i $, we get an embedding $ \hat{G}_i\hookrightarrow \Aut ( \mathcal{I}_D) $.

We claim that $ ((\hat{P}_{0})_{D_i})_{i\in \N} $ converges to $ (\hat{P}_{0})_{D} $. According to \cite{CR16}*{Lemma~2.1}, there are two things to verify.
\begin{enumerate}
\item Let $ (\hat{h}_i) $ be a sequence such that $ \hat{h}_i\in (\hat{P}_{0})_{D_i} $, and assume that $ \hat{h}_{i} $ converges to $ \hat{h} $ in $ \Aut ( \mathcal{I}_D) $. We have to show that $ \hat{h}\in (\hat{P}_{0})_{D} $. For all $ i $, let $ h_i\in (P_0)_{D_i} $ be an inverse image of $ \hat{h}_i $ under $ \hat{}~\colon G_i\to \Aut ( \mathcal{I}_D) $. Let $ \bar{h}_{i} = p_{di}(h_i)\in (P_{0}^{0,di})_{D_i} $. Let $ \varphi_{di}\colon (P_{0}^{0,di})_{D_i}\cong (P_{0}^{0,di})_{D} $ be the isomorphism given in Proposition~\ref{Prop:aritmimpliesgeomproximityforSL_2(D)}. By Corollary~\ref{Cor:surjectivityforSL2}, there exists $ \tilde{h}_{i}\in (P_{0})_{D} $ which is an inverse image of $ \varphi_{di}(\bar{h}_i) $ under $ p_{di}\colon (P_{0})_{D}\to (P_{0}^{0,di})_{D} $. Now, because all the identifications were equivariant, the action of $ \tilde{h}_{i} $ on the ball of radius $ di $ around $ 0 $ is the same than the action of $ \hat{h}_i $ on this ball. Hence, $ (\hat{\tilde{h}}_{i}) $ converges to $ \hat{h} $ as well. But $ (\hat{P}_{0})_{D} $ is a closed subgroup of $ \Aut ( \mathcal{I}_D) $ (by Proposition~\ref{Prop:Embedding G(K) in Aut(T)}), hence $ \hat{h}\in (\hat{P}_{0})_{D} $, as wanted.
\item Conversely, given an element $ \hat{h}\in (\hat{P}_{0})_{D} $, we have to find a sequence $ (\hat{h}_i) $ of elements in $ (\hat{P}_{0})_{D_i} $ such that $ (\hat{h}_{i}) $ converges to $ \hat{h} $ in $ \Aut ( \mathcal{I}_D) $. It suffices to follow the path of identifications in reverse: let $ h $ be an inverse image of $ \hat{h} $ under $ \hat{}~\colon G\to \Aut ( \mathcal{I}_D) $. Let $ \bar{h}_{i} = p_{di}(h)\in (P_{0}^{0,di})_{D} $, and let $ \varphi_{di}\colon (P_{0}^{0,di})_{D}\cong (P_{0}^{0,di})_{D_i} $ be the isomorphism given in Proposition~\ref{Prop:aritmimpliesgeomproximityforSL_2(D)}. For all $ i $, let $ h_i $ be an inverse image of $ \varphi_{di} (\bar{h}_i) $ under $ p_{di}~\colon (P_{0})_{D_i}\to (P_{0}^{0,di})_{D_i} $, which exists by Corollary~\ref{Cor:surjectivityforSL2}. Now, because all the identifications were equivariant, the action of $ h_{i} $ on the ball of radius $ di $ around $ 0 $ is the same than the action of $ h $ on this ball. Hence, $ (\hat{h}_{i}) $ converges to $ \hat{h} $, as wanted.
\end{enumerate}
Finally, from the convergence of $ ((\hat{P}_{0})_{D_i})_{i\geq N} $ to $ (\hat{P}_{0})_{D} $, we can formally deduce the convergence of $ (\hat{G}_i)_{i\geq N} $ to $ \hat{G} $. Indeed, $ (\hat{G}_i)_{i\geq N} $ subconverges to a topologically simple group $ H $, by \cite{CR16}*{Theorem~1.2}. But since $ ((\hat{P}_{0})_{D_i})_{i\geq N} $ converges to $ (\hat{P}_{0})_{D} $, $ H $ has an open compact subgroup isomorphic to $ (\hat{P}_{0})_{D} $. Hence, by \cite{CS15}*{Corollary~1.3}, $ H $ is algebraic. And hence, by \cite{Pink98}*{Corollary~0.3}, $ H \cong G $. Since by the same argument, any subsequence of $ (\hat{G}_i)_{i\geq N} $ subconverges to $ \hat{G} $, we conclude that $ (\hat{G}_i)_{i\geq N} $ converges to $ \hat{G} $. 
\end{proof}

We then deduce the proof of the main theorem announced in the introduction for groups of type $ \SL_2(D) $. To shorten the notations, we set $ G_D = \SL_2(D) $ in the following proof.
\begin{proof}[Proof of Theorem~\ref{Thm:explicit form for SL_2(D)}]
Let $ T $ be a regular tree and let $ \mathcal{D}_{T} = \lbrace D\in \mathcal{D}~\vert~ $the Bruhat--Tits tree of $ G_{D} $ is isomorphic to $ T\rbrace $. By Remark~\ref{Rem:regularity of the B-T tree SL2} and Proposition~\ref{Prop:explicit description of D}, $ \mathcal{D}_{T} $ is a compact space. Now, by Theorem~\ref{Thm:continuity from D to Chab(Aut(T))}, the map $ \mathcal{D}_T\to\mathcal{S}_T\colon D\mapsto \hat{G}_{D} $ is continuous. Let $ D_1 $ and $ D_2 $ be central division algebras over $ K_1 $ and $ K_2 $ respectively, with respective degree $ d_1,d_2 $ and Hasse invariant $ r_1,r_2 $ (as defined in Definition~\ref{Def:Hasse invariant}). We claim that $ \hat{G}_{D_1} = \hat{G}_{D_2} $ if and only if $ K_1\cong K_2 $, $ d_1=d_2 $ and $ r_1 = \pm r_2 $. Indeed, if $ \hat{G}_{D_1} $ is abstractly isomorphic to $ \hat{G}_{D_2} $, then by \cite{BoTi73}*{Corollaire~8.13}, the corresponding adjoint algebraic groups $ \Ad \mathbf{G}_1 $ and $ \Ad \mathbf{G}_2 $ are algebraically isomorphic over an isomorphism of fields $ K_1\cong K_2 $. Now, according to \cite{KMRT98}*{Remark 26.11}, this is only possible if $ D_1\cong D_2 $ or $ D_1\cong D_2^{\opp} $, which is equivalent to the conditions we gave.

To summarise, let $ \mathcal{D}_T/\sim_{\opp} $ be the space $ \mathcal{D}_T $ modulo the equivalence relation $ D_1\sim_{\opp}D_2 $ if and only if $ D_1\cong D_2 $ or $ D_1\cong D_2^{\opp} $. We proved that $ \mathcal{D}_T/\sim_{\opp}\to\mathcal{S}_T\colon D\mapsto \hat{G}_{D} $ is an injective continuous map whose source is a compact space, hence it is a homeomorphism onto its image. Now, the explicit description given in Theorem~\ref{Thm:explicit form for SL_2(D)} follows from Remark~\ref{Rem:regularity of the B-T tree SL2} and Proposition~\ref{Prop:explicit description of D}.

To be able to conclude that for $ T $ the $ (p^n+1) $-regular tree, $ \mathcal{S}_T^{\SL_2(D)} $ is homeomorphic to $ \hat{\N}\times \lbrace 1,\dots,\lceil \frac{n+1}{2}\rceil\rbrace $, one has to count the number of division algebras in $ \mathcal{D}_T/\sim_{\opp} $ of characteristic $ p $. But there is only one such division algebra in $ \mathcal{D}_T/\sim_{\opp} $ of degree $ 1 $ over its centre, one such division algebra in $ \mathcal{D}_T/\sim_{\opp} $ of degree $ 2 $ over its centre if $ 2 $ divides $ n $, and for all $ 3\leq d $ dividing $ n $, there are $ \frac{\varphi (d)}{2} $ such division algebras in $ \mathcal{D}_T/\sim_{\opp} $ of degree $ d $ over their centre (where $ \varphi $ denotes Euler's totient function). Hence, if $ n $ is even (respectively odd), we have $ 2 + \sum_{d\vert n, d\geq 3}\frac{\varphi (d)}{2} $ (respectively $ 1 + \sum_{d\vert n, d\geq 3}\frac{\varphi (d)}{2} $) division algebras of characteristic $ p $ in $ \mathcal{D}_T/\sim_{\opp} $. Using that $ \sum_{d\vert n}\varphi (d) = n $, we readily get the conclusion.
\end{proof}

\section{Convergence of groups of type \texorpdfstring{$ \SU_3^{L/K} $}{SU3L/K}, \texorpdfstring{$ L $}{L} unramified}\label{Sec:SU3L/K unramified}
We keep our notations for local fields (see Section~\ref{Sec:Def of alg. grps}). Furthermore, throughout this section, $ L $ is an unramified quadratic extension of the base local field $ K $. Note that such an extension is automatically separable. Also note that $ \pi_K $ is equally well a uniformiser of $ L $. We carry out the same program than in Section~\ref{Sec:SL_2(D)}, replacing all occurrences of $ \SL_2(D) $ by $ \SU_3^{L/K} $. The comments made all along Section~\ref{Sec:SL_2(D)} also apply here, but we do not repeat them to not lengthen too much the paper.

\subsection{Construction of the Bruhat--Tits tree}
In the following definition of point stabilisers, we again use the notation introduced in Definition~\ref{Def:valuation of a matrix}.
\begin{definition}\label{Def:pointstabiliserforSU unramified}
For $ x\in \R $, we define $ P_{x} = \lbrace g\in \SU_{3}^{L/K}(K)~\vert~ \omega (g)\geq \begin{psmallmatrix}
0 & -\frac{x}{2} & -x \\ 
\frac{x}{2} & 0 & -\frac{x}{2} \\
x & \frac{x}{2}& 0
\end{psmallmatrix}\rbrace $
\end{definition}

\begin{definition}\label{Def:sbgrp N SU unramified}
Consider the following subsets
\begin{enumerate}[$ \bullet $]
\item $ T = \lbrace \begin{psmallmatrix}
x & 0 & 0 \\ 
0 & x^{-1}\bar{x} & 0 \\
0 & 0 & \bar{x}^{-1}
\end{psmallmatrix}~\vert~x\in L^{\times}\rbrace < \SU_3^{L/K} (K) $
\item $ M =\lbrace \begin{psmallmatrix}
0 & 0 & x \\ 
0 & -x^{-1}\bar{x} & 0 \\
\bar{x}^{-1} & 0 & 0
\end{psmallmatrix}~\vert~x\in L^{\times}\rbrace \subset \SU_3^{L/K} (K) $
\end{enumerate}
and let $ N = T\sqcup M $.
\end{definition}

\begin{definition}\label{Def:affineactionforN unramified}
Let $ \nu\colon N\to \Aff (\R) $ be defined as follows: for $ m = \begin{psmallmatrix}
0 & 0 & x \\ 
0 & -x^{-1}\bar{x} & 0 \\
\bar{x}^{-1} & 0 & 0
\end{psmallmatrix}\in M $, $ \nu (m) $ is the reflection through $ -\omega (x) $, while for $ t=\begin{psmallmatrix}
x & 0 & 0 \\ 
0 & x^{-1}\bar{x} & 0 \\
0 & 0 & \bar{x}^{-1}
\end{psmallmatrix}\in T $, $ \nu (t) $ is the translation by $ -2\omega (x) $.
\end{definition}

Then the Bruhat--Tits tree $ \mathcal{I} $ of $ \SU_3^{L/K} $ (recall that in this section, $ L $ is unramified) is the one obtained by applying Definition~\ref{Def:buildingassociated to a data} to the collection of subgroups $ \lbrace (P_x)_{x\in \R}, N \rbrace $ appearing in Definition~\ref{Def:pointstabiliserforSU unramified} and Definition~\ref{Def:sbgrp N SU unramified}, together with the homomorphism $ \nu \colon N\to \Aff (\R) $ of Definition~\ref{Def:affineactionforN unramified}. We show in Appendix~\ref{App:A} that our definitions agree with \cite{BrTi1}*{7.4.1~and~7.4.2}, so that the given data is indeed obtained from a valued root datum of rank one on $ G $

\begin{remark}\label{Rem:dependence on D of the tree ur}
Note that the construction of the Bruhat--Tits tree of $ \SU_3^{L/K} $ depends on the pair $ (K,L) $. When needed, we keep track of this dependence by adding the subscript $ (K,L) $ to the objects involved. This gives rise to the notations $ (P_x)_{(K,L)} $, $ T_{(K,L)} $, $ M_{(K,L)} $, $ N_{(K,L)} $, $ \nu_{(K,L)} $ and $ \mathcal{I}_{(K,L)} $.
\end{remark}
\begin{remark}\label{Rem:regularity of the B-T tree SU unramified}
The Bruhat--Tits tree of $ \SU_3^{L/K} $ is actually the $ (\vert \overline{K}\vert^{3}+1;\vert \overline{K}\vert+1) $-semiregular tree. Indeed, this follows from the fact that our definition of $ \mathcal{I} $ agrees with the one given in \cite{BrTi1}*{7.4.1 and 7.4.2}, and from the tables in \cite{Tits77}*{4.2 and 4.3}.
\end{remark}

\subsection{Local model of the Bruhat--Tits tree}
We now proceed to define a local model for the Bruhat--Tits tree of $ SU_3^{L/K} $ when $ L $ is unramified. The same remarks as in the $ SL_2(D) $ case apply, so that we go quickly through the definitions. Note that the valuation $ \omega $ (respectively the Galois conjugation $ x\mapsto\bar{x} $) on $ L $ induces a well-defined map on $ \mathcal{O}_L/\mathfrak{m}_L^r $ that we still denote $ \omega $ (respectively $ x\mapsto \bar{x} $).
\begin{definition}\label{Def:localdataSU(3)ur}
Let $ r\in \N $ and $ x\in [-\omega (\pi_{L}^{r}),\omega (\pi_{L}^{r})] $. We set $$ P_{x}^{0,r} = \lbrace g\in \SL_{3}(\mathcal{O}_{L}/\mathfrak{m}_{L}^{r})~\vert~ ^{S}\bar{g}g= \Id, \omega (g)\geq \begin{psmallmatrix}
0 & -\frac{x}{2} & -x \\
\frac{x}{2} & 0 & -\frac{x}{2} \\
x & \frac{x}{2} & 0
\end{psmallmatrix} \rbrace. $$
\end{definition}

\begin{definition}
\begin{enumerate}[$ \bullet $]
We define
\item $ H^{0,r} = \lbrace \begin{psmallmatrix}
x & 0 & 0 \\
0 & x^{-1}\bar{x} & 0 \\
0 &0 & \bar{x}^{-1}
\end{psmallmatrix}\in \SL_{3}(\mathcal{O}_{L}/\mathfrak{m}_{L}^{r})~\vert~ \omega (x) = 0 \rbrace $
\item $ M^{0,r} = \lbrace \begin{psmallmatrix}
0 & 0 & x \\
0 & -x^{-1}\bar{x} & 0 \\
\bar{x}^{-1} &0 & 0
\end{psmallmatrix}\in \SL_{3}(\mathcal{O}_{L}/\mathfrak{m}_{L}^{r})~\vert~ \omega (x) = 0 \rbrace $
\end{enumerate}
And we set $ N^{0,r} = H^{0,r}\sqcup M^{0,r} $.
\end{definition}

\begin{definition}
We let $ H^{0,r} $ act trivially on $ \R $, and we let all elements of $ M^{0,r} $ act as a reflection through $ 0\in \R $. This gives an affine action of $ N^{0,r} $ on $ \R $, and we denote again the resulting map $ N^{0,r}\to \Aff (\R) $ by $ \nu $.
\end{definition}

We are now able to give a definition of the ball of radius $ r $ around $ [(\Id ,0)]\in \mathcal{I} $ which only depends on the ring $ \mathcal{O}_L/\mathfrak{m}_L^{r} $, and not on the whole field $ L $.
\begin{definition}\label{Def:local tree of radius r unramified}
Let $ r\in \N $. We define an $ r $-local equivalence on $ P_{0}^{0,r}\times [-\omega (\pi_L^{r}), \omega (\pi_L^{r})] $ as follows. For $ g,h\in P_{0}^{0,r} $ and $ x,y\in [-\omega (\pi_L^{r}), \omega (\pi_L^{r})] $
\begin{equation*}
(g,x)\sim_{0,r}(h,y) \Leftrightarrow \textrm{ there exists } n\in N^{0,r} \textrm{ such that } \nu (n)(x)=y \textrm{ and } g^{-1}hn \in P_{x}^{0,r}
\end{equation*}
The resulting space $ \mathcal{I}^{0,r} = P_{0}^{0,r}\times [-\omega (\pi_L^{r}), \omega (\pi_L^{r})]/\sim_{0,r} $ is called the local Bruhat--Tits tree of radius $ r $ around $ 0 $, and $ [(g,x)]^{0,r} $ stands for the equivalence class of $ (g,x) $ in $ \mathcal{I}^{0,r} $. The group $ P_{0}^{0,r} $ acts on $ \mathcal{I}^{0,r} $ by multiplication on the first component.
\end{definition}

\begin{remark}\label{Rem:dependence on D of the local tree ur}
Note that the construction of the local Bruhat--Tits tree of $ \SU_3^{L/K} $ depends on the pair $ (K,L) $ (which is assumed to be unramified in this section). When needed, we keep track of this dependence by adding the subscript $ (K,L) $ to the objects involved. This gives rise to the notations $ (P_x^{0,r})_{(K,L)} $, $ H^{0,r}_{(K,L)} $, $ M^{0,r}_{(K,L)} $, $ N^{0,r}_{(K,L)} $ and $ \mathcal{I}^{0,r}_{(K,L)} $.
\end{remark}
\begin{remark}\label{Rem:equivalent def with for all n for local model unramified}
Also, Remark~\ref{Rem:equivalent def with for all n for local model} holds equally well in this case, with exactly the same proof, upon replacing all $ d $'s by $ 1 $.
\end{remark}

\subsection{Integral model}
We now proceed to define an integral model, and compare its rational points with our local model for the Bruhat--Tits tree of $ SU_3^{L/K} $ when $ L $ is unramified. Again, the same remarks as in the $ SL_2(D) $ case apply, so that we go quickly through the definitions.

\begin{definition}\label{Def:integral model for ur SU}
Let $ \underline{\SU}_{3}^{L/K} $ be the group $ \SU_{3} $ considered over $ \mathcal{O}_{K} $. We omit the superscript $ L/K $ when it is understood from the context.
\end{definition}

Concretely, $ \underline{\SU}_{3} $ is the $ \mathcal{O}_{K} $-scheme associated with the $ \mathcal{O}_{K} $-algebra $ \mathcal{O}_{K}[\underline{\SU}_{3}] = \mathcal{O}_{K}[X_{ij}^{kl}]/I $ ($ i,j\in \lbrace 1,2,3\rbrace $, $ k,l\in \lbrace 1,2\rbrace $), where $ I $ is the ideal generated by the following equations 
\begin{flalign*}
& \text{For all }i,j\in \lbrace 1,2,3\rbrace,~
\begin{cases}
X_{ij}^{12} = -\beta X_{ij}^{21}\\
X_{ij}^{22} = X_{ij}^{11}+\alpha X_{ij}^{21}
\end{cases} \\
& \sum\limits_{\sigma \in \Sym (3)}[ (-1)^{sgn(\sigma)}\prod\limits_{i=1}^3X_{i\sigma (i)} ]- 1\\
& \begin{psmallmatrix}
\overline{X}_{33} & \overline{X}_{23} & \overline{X}_{13} \\
\overline{X}_{32} & \overline{X}_{22} & \overline{X}_{12} \\
\overline{X}_{31} & \overline{X}_{21} & \overline{X}_{11} 
\end{psmallmatrix}\begin{psmallmatrix}
X_{11} & X_{12} & X_{13} \\
X_{21} & X_{22} & X_{23} \\
X_{31} & X_{32} & X_{33} 
\end{psmallmatrix} - \begin{psmallmatrix}
1 & 0 & 0 \\
0 &1 & 0 \\
0 & 0 & 1
\end{psmallmatrix}
\end{flalign*}
Here $ \alpha $ and $ \beta $ are elements of $ \mathcal{O}_K $ such that $ L\cong K[T]/(T^2-\alpha T+\beta) $, so that the first equations encode the ring embedding $ \mathcal{O}_{L}\hookrightarrow M_{2}(\mathcal{O}_{K}) $. In the other equations, $ X_{ij} $ stands for the $ 2\times 2 $ matrix $ \begin{psmallmatrix}
X_{ij}^{11} & X_{ij}^{12} \\
X_{ij}^{21} & X_{ij}^{22} 
\end{psmallmatrix} $. Also, for a $ 2\times 2 $ matrix $ M = \begin{psmallmatrix}
M^{11} & M^{12} \\
M^{21} & M^{22} 
\end{psmallmatrix} $, we denote $ \overline{M} = \begin{psmallmatrix}
M^{22} & -M^{12} \\
-M^{21} & M^{11} 
\end{psmallmatrix} $ (this operation reflects the conjugation on $ \mathcal{O}_{L} $). Finally note that a $ 1 $ (respectively a $ 0 $) in the above equations denotes the $ 2\times 2 $ identity matrix (respectively the $ 2\times 2 $ zero matrix), i.e. it corresponds to the $ 1\in L $ (respectively $ 0\in L $).

\begin{theorem}\label{Thm:smoothness of SU3 unramified}
$ \underline{\SU}_{3}^{L/K} $ is a smooth $ \mathcal{O}_{K} $-scheme.
\end{theorem}
\begin{proof}
We prove that the base change of $ \underline{\SU}_{3}^{L/K} $ to $ \mathcal{O}_L $ is isomorphic to $ \SL_3 $ (as an algebraic group over $ \mathcal{O}_L $). Since $ \SL_3 $ is smooth over any base ring, the result is then a consequence of faithfully flat descent.

Let $ ~\bar{}~\colon \mathcal{O}_L\to \mathcal{O}_L $ be the Galois conjugation. Note that since $ \mathcal{O}_L $ is unramified over $ \mathcal{O}_K $, the map $ \psi \colon \mathcal{O}_L\otimes_{\mathcal{O}_K}\mathcal{O}_L\to \mathcal{O}_L\times \mathcal{O}_L\colon x\otimes y\mapsto (xy,x\bar{y}) $ is an isomorphism of $ \mathcal{O}_L $-algebras. 
Indeed, this follows from the fact that $ L $ is unramified, so that $ \mathcal{O}_K\to \mathcal{O}_L $ is etale, and since local fields are henselian, the corresponding morphism of schemes is a Galois covering (see \cite{BLR90}*{6.2, Example~B} for more details). But in this simple case, one can also just check by hand that $ \psi $ is an isomorphism. This implies that for any $ \mathcal{O}_L $-algebra $ R $, the map $ \varphi \colon R\otimes_{\mathcal{O}_K}\mathcal{O}_L\to R\times R\colon r\otimes x\mapsto (rx,r\bar{x}) $ is an isomorphism of $ \mathcal{O}_L $-algebras. 

Furthermore, let $ \tau \colon R\times R\to  R\times R\colon (r,r')\mapsto (r',r) $ be the natural involution on $ R\times R $. Then for any $ r\otimes x\in R\otimes_{\mathcal{O}_K}\mathcal{O}_L $, $ \varphi (r\otimes \bar{x}) = \tau (\varphi (r\otimes x)) $. It follows that for any $ \mathcal{O}_L $-algebra $ R $, we have
\begin{align*}
(\underline{\SU}_{3}^{L/K})_{\mathcal{O}_L}(R) &= \lbrace g\otimes x\in \SL_3(R\otimes_{\mathcal{O}_K}\mathcal{O}_L)~\vert ~^{S}(g\otimes \bar{x})(g\otimes x) = \Id \rbrace\\
&\cong \lbrace (g,h)\in \SL_3(R\times R)~\vert ~(^{S}hg,~^{S}gh) = (\Id ,\Id) \rbrace\\
&\cong \SL_3(R)
\end{align*}
Since the above isomorphisms are natural in $R$, this indeed shows that $ (\underline{\SU}_{3}^{L/K})_{ \mathcal{O}_L}\cong \SL_3 $ as group schemes over $ \mathcal{O}_L $, which concludes the proof.
\end{proof}
We now compare the rational points of the integral model with our local model.
\begin{lemma}\label{Lem:description of the restriction P_0 to P_0^0,r unramified}
$ \underline{\SU}_3(\mathcal{O}_{K}) \cong P_{0} $ and $ \underline{\SU}_3(\mathcal{O}_{K}/\mathfrak{m}_{K}^{r}) \cong P_{0}^{0,r} $. Following the identifications
\begin{center}
\begin{tikzpicture}[->]
 \node (1) at (0,0) {$ \underline{\SU}_3(\mathcal{O}_{K}) $};
 \node (2) at (2.3,0) {$ P_{0}\leq \SL_{3}(\mathcal{O}_{L}) $};
 \node (5) at (1,0) {$ \cong $};
 \node (3) at (-0.3,-1.2) {$ \underline{\SU}_3(\mathcal{O}_{K}/\mathfrak{m}_{K}^{r}) $};
 \node (4) at (2.8, -1.2) {$ P_{0}^{0,r}\leq \SL_3(\mathcal{O}_{L}/\mathfrak{m}_{L}^{ r}) $};
 \node (6) at (1,-1.2) {$ \cong $};
 
 \draw[->] (0,-0.3) to (0,-0.9);
 \draw[->] (1.4,-0.3) to (1.4,-0.9);
\end{tikzpicture}
\end{center}
the homomorphism $ p_{r}\colon P_{0}\to P_{0}^{0,r} $ is the one induced by the projection of the coefficients $ \mathcal{O}_{L}\to \mathcal{O}_{L}/\mathfrak{m}_{L}^{r} $.
\end{lemma}
\begin{proof}
Let $ t $ be any element in $ \mathcal{O}_L^{\times}\setminus \mathcal{O}_K $. Using the fact that $ \mathcal{O}_{L}\cong \mathcal{O}_{K}\oplus t.\mathcal{O}_{K} $, one can check that $ \Mor_{\mathcal{O}_{K}}(\mathcal{O}_{K}[\underline{\SU}_{3}],\mathcal{O}_{K}) \cong \lbrace g\in \SL_{3}(\mathcal{O}_{L})~ \vert ~ ^{S}\bar{g}g=\Id \rbrace $, as wanted.

Furthermore, since $ L $ is unramified, $ \mathcal{O}_{L}/\mathfrak{m}_{L}^{r}\cong \mathcal{O}_{K}/\mathfrak{m}_{K}^{r}\oplus t.\mathcal{O}_{K}/\mathfrak{m}_{K}^{r} $, and one can check that $ \Mor_{\mathcal{O}_{K}}(\mathcal{O}_{K}[\underline{\SU}_{3}],\mathcal{O}_{K}/\mathfrak{m}_{K}^{r}) \cong \lbrace g\in \SL_{3}(\mathcal{O}_{L}/\mathfrak{m}_{L}^{r})~ \vert ~ ^{S}\bar{g}g=\Id \rbrace $, as wanted.
\end{proof}

And we can then deduce the surjectivity of the map $ p_r $.
\begin{corollary}\label{Cor:surjectivityofpr unramified}
The map $ p_{r}\colon P_{0}\to P_{0}^{0,r} $ is surjective, for all $ r\in \N $.
\end{corollary}
\begin{proof}
This is a direct consequence of the commutative square involving $ P_{0}\to P_{0}^{0,r} $ given in Lemma~\ref{Lem:description of the restriction P_0 to P_0^0,r unramified}, together with the fact that the integral model is smooth by Theorem~\ref{Thm:smoothness of SU3 unramified}, so that Theorem~\ref{Thm:genhensel} applies to the left hand side of the diagram.
\end{proof}

We also need a kind of injectivity result:
\begin{lemma}\label{Lem:injectivityofpr unramified}
Let $ r\in \N $ and $ x\in [-\omega (\pi_{L}^{r}),\omega (\pi_{L}^{r})] $. Then $ p_{r}^{-1}(P_{x}^{0,r}) \subset P_{x} $.
\end{lemma}
\begin{proof}
Belonging to $ p_{r}^{-1}(P_x^{0,r}) $ implies that the valuation of the off diagonal entries are big enough. Hence, the result follows directly from Definition~\ref{Def:pointstabiliserforSU unramified}.
\end{proof}

We finally arrive at the result corresponding to Theorem~\ref{Thm:localdescriptionoftheball}: the ball of radius $ r $ together with the action of $ \SU_{3}^{L/K}(K) $ is encoded in $ P_{0}^{0,r} $. We first need an adequate description of the ball of radius $ r $ around $0$ in $ \mathcal{I} $.

\begin{lemma}\label{Lem:B_0(r) is really the ball of radius r unramified}
Renormalise the distance on $ \R $ so that $ d_{\R}(0;\omega(\pi_L)) = 1 $, and put the metric $ d_{\mathcal{I}} $ on $ \mathcal{I} $ arising from the distance $ d_{\R} $ $ ( $see Remark~\ref{Rem:metric on I and stabiliser}$ ) $. Let $ B_0(r) = \lbrace p\in \mathcal{I}~\vert~d_{\mathcal{I}}([(\Id ,0)];p)\leq r\rbrace $ be the ball of radius $ r $ around $ 0 $ in $ \mathcal{I} $. Let $ \tilde{B}_0(r) = \lbrace [(g,x)]\in \mathcal{I} ~\vert~ g\in P_{0}, x\in [-\omega (\pi_L^{r}), \omega (\pi_L^{r})]\subset \R \rbrace $. Then $ B_0(r) = \tilde{B}_0(r) $.
\end{lemma}
\begin{proof}
The proof is word for word the same than the proof of Lemma~\ref{Lem:B_0(r) is really the ball of radius r}, upon replacing all $d$'s by $1$'s.
\end{proof}
\begin{remark}
The distance $ d_{\mathcal{I}} $ that we introduced in Lemma~\ref{Lem:B_0(r) is really the ball of radius r unramified} is also the combinatorial distance on the tree. Indeed, looking at when $ P_y $ is inside $ P_x $ for $ x,y\in \R $, we see that $ [(\Id ,x)] $ is a vertex of $ \mathcal{I} $ if and only if $ x\in \omega(\pi_L)\Z $. Furthermore, $ x $ is a vertex of degree $ \vert \overline{K}\vert^3+1 $ if and only if $ x\in 2.\omega(\pi_L)\Z $.
\end{remark}

\begin{theorem}\label{Thm:localdescriptionoftheball ur}
Let $ r\in \N $. The map $ B_{0}(r)\to \mathcal{I}^{0,r}\colon [(g,x)]\mapsto [(p_{r}(g),x)]^{0,r} $ is a $ (p_{r}\colon P_{0}\to P_{0}^{0,r}) $-equivariant bijection.
\end{theorem}
\begin{proof}
The map is well-defined by Lemma~\ref{Lem:integralityofn}.
\begin{itemize}
\item Injectivity: let $ [(g,x)], [(h,y)] \in B_0(r) $ be such that they have the same image in $ \mathcal{I}^{0,r} $. By Remark~\ref{Rem:equivalent def with for all n for local model unramified}, it means that for all $ \tilde{n}\in N^{0,r} $ such that $ \nu (\tilde{n})(x) = y $, $ p_{r}(g)^{-1}p_{r}(h)\tilde{n}\in P_{x}^{0,r} $. So, we can assume that $ \tilde{n} $ is either equal to $ \Id $, or is of the form $ \begin{psmallmatrix} 
0&0 & 1 \\
0&-1 & 0 \\
1&0&0
\end{psmallmatrix} $. Hence, there exists $ n\in N $ such that $ p_{r}(n) = \tilde{n} $. But $ \nu (n)(x) = y $, and $ g^{-1}hn\in p_{r}^{-1}(P_{x}^{0,r}) \subset P_{x} $ by Lemma~\ref{Lem:injectivityofpr unramified}. Hence, $ [(g,x)] = [(h,y)] $, as wanted.
\item Surjectivity: follows directly from the surjectivity of $ p_{r}\colon P_{0}\to P_{0}^{0,r} $ (Corollary~\ref{Cor:surjectivityofpr unramified}).
\item Equivariance: $ h.[(g,x)] = [(hg,x)] \mapsto [(p_{r}(hg),x)]^{0,r} = p_{r}(h).[(p_{r}(g),x)]^{0,r} $. \qedhere
\end{itemize}
\end{proof}

\subsection{Arithmetic convergence}
\begin{definition}\label{Def:space L unramified}
Consider the set of unramified pairs $(K,L)$ where $K$ is a local field and $L$ is an unramified (separable) quadratic extension of $K$. We say that two unramified pairs $ (K_{1},L_{1}) $ and $ (K_{2},L_{2}) $ are isomorphic if there exists a conjugation equivariant isomorphism between $ L_1 $ and $ L_2 $, and we let $ \mathcal{L}^{\uram} $ be the set of unramified pairs, up to isomorphism. 
For each prime $ p $, let us also define $ \mathcal{L}^{\uram}_{p^{n}} = \lbrace (K,L) \in \mathcal{L}^{\uram}~\vert~ \vert \overline{K}\vert = p^{n}\rbrace $.
\end{definition}

As in Section~\ref{Sec:SL_2(D)}, we define a metric on the space $ \mathcal{L}^{\uram} $. For $L\in \mathcal{L}^{\uram} $ and $ r\in \N $, the Galois conjugation induces an automorphism of $ \mathcal{O}_{L}/\mathfrak{m}_{L}^{r} $ that we still call the conjugation.
\begin{definition}
Let $ (K_{1},L_{1}) $ and $ (K_{2},L_{2}) $ be in $ \mathcal{L}^{\uram} $. We say that $ (K_{1},L_{1}) $ is $ r $-close to $ (K_2,L_2) $ if and only if there exists a conjugation equivariant isomorphism $ \mathcal{O}_{L_{1}}/\mathfrak{m}_{L_1}^{r}\to \mathcal{O}_{L_{2}}/\mathfrak{m}_{L_2}^{r} $.
\end{definition}
\begin{remark}\label{Rem:induced proximity on K ur}
A conjugation equivariant isomorphism $ \mathcal{O}_{L_{1}}/\mathfrak{m}_{L_1}^{r}\to \mathcal{O}_{L_{2}}/\mathfrak{m}_{L_2}^{r} $ always induces an isomorphism $ \mathcal{O}_{K_{1}}/\mathfrak{m}_{K_1}^{r}\to \mathcal{O}_{K_{2}}/\mathfrak{m}_{L_2}^{r} $, since $ \mathcal{O}_{K_{i}}/\mathfrak{m}_{K_i}^{r} $ is the invariant subring of $ \mathcal{O}_{L_{i}}/\mathfrak{m}_{L_i}^{r} $.
\end{remark}
\begin{remark}\label{Rem:transitivityofdistance  unramified}
Note that being $ r $-close is an equivalence relation, and that if $ r\geq l $ and $ (K_1,L_1) $ is $ r $-close to $ (K_2,L_2) $, then $ (K_1,L_1) $ is $ l $-close to $ (K_2,L_2) $.
\end{remark}

Again, this notion of closeness induces a non-archimedean metric on $ \mathcal{L}^{\uram} $. Let
\begin{equation*}
d\colon \mathcal{L}^{\uram}\times \mathcal{L}^{\uram}\to \mathbf{R}_{\geq 0}\colon d((K_1,L_1);(K_2,L_2))= \inf \lbrace \frac{1}{2^{r}} ~\vert~ (K_1,L_1) \textrm{ is } r \textrm{-close to }(K_2,L_2) \rbrace
\end{equation*}

\begin{lemma}\label{Lem:Non-archim. metric space ur}
$ d(\cdot~;~\cdot ) $ is a non-archimedean metric on $ \mathcal{L}^{\uram} $. 
\end{lemma}
\begin{proof}
If $ d((K_1,L_1);(K_2,L_2)) = 0 $, then $ \mathcal{O}_{L_1} $ and $ \mathcal{O}_{L_2} $ are equivariantly isomorphic. 
Hence, the pairs of field of fraction are isomorphic in $ \mathcal{L}^{\uram} $, as wanted. The fact that this distance is non-archimedean is a consequence of Remark~\ref{Rem:transitivityofdistance unramified}.
\end{proof}

Actually, by the uniqueness of unramified quadratic extension over a given local field, $ \mathcal{L}^{\uram} $ is isometric to the space $ \mathcal{K} $ introduced in Definition~\ref{Def:the space D}.
\begin{proposition}\label{Prop:explicit description of L ur}
The map $ \mathcal{L}^{\uram}\to \mathcal{K}\colon (K,L)\to K $ is an isometry, which maps $ \mathcal{L}^{\uram}_{p^n} $ to $ \mathcal{K}_{p^n} $. Hence, $ \mathcal{L}^{\uram}_{p^n} $ is homeomorphic to $ \hat{\N} $, the accumulation point being $ (\mathbf{F}_{p^{n}}(\!(X)\!),\mathbf{F}_{p^{2n}}(\!(X)\!)) $.
\end{proposition}
\begin{proof}
The given map is indeed a bijection, by the uniqueness of unramified quadratic extensions. We prove that it is an isometry. Let $ (K_1,L_1),(K_2,L_2)\in \mathcal{L}^{\uram} $. By Remark~\ref{Rem:induced proximity on K ur}, if $ d((K_1,L_1);(K_2,L_2))\leq \varepsilon  $ (for some $ \varepsilon \in \R_{>0}) $), then $ d(K_1;K_2)\leq \varepsilon $. 

To prove the converse, assume that $ K_1 $ and $ K_2 $ are $ r $-close with $ r\geq 1 $ (note that there is nothing to prove for $ r=0 $). Then $ \overline{K_1}\cong \overline{K_2} \cong \mathbf{F}_{p^n} $. Let $ T^2-\alpha T+\beta $ be an irreducible separable quadratic polynomial in $ \mathbf{F}_{p^n}[T] $. For $i=1,2$, let $ \alpha_i $ (respectively $ \beta_i $) be a lift of $ \alpha $ (respectively $ \beta $) under $ \mathcal{O}_{K_i}/\mathfrak{m}_{K_i}^r\to \overline{K_i} $. Then $ \mathcal{O}_{L_i}/\mathfrak{m}_{L_i}^{r}\cong (\mathcal{O}_{K_i}/\mathfrak{m}_{K_i}^{r})[T]/(T^2-\alpha_iT+\beta_i) $ (for $ i = 1, 2 $). 

Since this is true for any lifts $\alpha_i$ and $\beta_i$, we can assume that $ \alpha_1 $ (respectively $\beta_1$) maps to $ \alpha_2 $ (respectively $ \beta_2 $) under the given isomorphism $ f\colon \mathcal{O}_{K_1}/\mathfrak{m}_{K_1}^r\cong \mathcal{O}_{K_2}/\mathfrak{m}_{K_2}^r $, so that $ f $ extends to a conjugation equivariant isomorphism of rings $ \mathcal{O}_{L_1}/\mathfrak{m}_{L_1}^{r}\cong \mathcal{O}_{L_2}/\mathfrak{m}_{L_2}^{r} $. We conclude that $ \mathcal{L}^{\uram} $ is homeomorphic to $ \hat{\N} $ in view of Proposition~\ref{Prop:explicit description of D}.
\end{proof}

\subsection{Continuity from unramified pairs to subgroups of \texorpdfstring{$ \Aut (T) $}{Aut(T)}}
In this section, we start to vary the unramified pair $ (K,L) $, and look at the variation it produces on the Bruhat--Tits tree of $ \SU_3^{L/K} $. Recall that we introduced a notation to keep track of the dependence on $ (K,L) $ of many of the definitions we made in this section (see Remark~\ref{Rem:dependence on D of the tree ur} and Remark~\ref{Rem:dependence on D of the local tree ur}).

\begin{proposition}\label{Prop:aritmimpliesgeomproximityforSU3 ur}
Let $ (K_1,L_1) $ and $ (K_2,L_2) $ be two elements in $ \mathcal{L}^{\uram} $. Assume that $ (K_1,L_1) $ is $ r $-close to $ (K_2,L_2) $, for some $ r \in \N $. Then, $ (P_{0}^{0,r})_{(K_1,L_1)}\cong (P_{0}^{0,r})_{(K_2,L_2)} $, and $ \mathcal{I}_{(K_1,L_1)}^{0,r}$ is equivariantly in bijection with $ \mathcal{I}_{(K_2,L_2)}^{0,r} $.
\end{proposition}
\begin{proof}
The isomorphism $ \mathcal{O}_{L_1}/\mathfrak{m}_{L_1}^{r}\cong \mathcal{O}_{L_2}/\mathfrak{m}_{L_2}^{r} $ induces a group isomorphism $ \varphi $
\begin{center}
\begin{tikzpicture}[->]
 \node (1) at (0,0) {$ \SL_{3}(\mathcal{O}_{L_1}/\mathfrak{m}_{L_1}^{r}) $};
 \node (2) at (3,0) {$ \SL_{3}(\mathcal{O}_{L_2}/\mathfrak{m}_{L_2}^{r}) $};
 \node (7) at (1.5,0) {$ \cong$};
 \node (5) at (-0.5,-0.4) {$ \vee $};
 \node (3) at (-0.2,-0.9) {$ (P_{0}^{0,r})_{(K_1,L_1)} $};
 \node (4) at (2.8, -0.9) {$ (P_{0}^{0,r})_{(K_2,L_2)} $};
 \node (6) at (2.5,-0.4) {$ \vee $};

\path [every node/.style={font=\sffamily\small}]
 (3) edge node [below]{$ \varphi $} (4);


\end{tikzpicture}
\end{center}
Define a linear map $ f\colon \mathbf{R}\to \mathbf{R}\colon x\mapsto x\frac{\omega (\pi_{L_2})}{\omega (\pi_{L_1})} $. It is clear that for all $ x\in [-\omega (\pi_{L_1}^{r}), \omega (\pi_{L_1}^{r})] $, $ \varphi $ restricts to an isomorphism $ (P_{x}^{0,r})_{(K_1,L_1)}\cong (P_{f(x)}^{0,r})_{(K_2,L_2)} $. Furthermore, 
\begin{align*}
\varphi (T^{0,r})_{(K_1,L_1)}&= (T^{0,r})_{(K_2,L_2)}\\
\varphi (M^{0,r})_{(K_1,L_1)}&= (M^{0,r})_{(K_2,L_2)}
\end{align*}
and for all $ n\in N^{0,r} $, $ f(n.x) = \varphi (n).f(x) $. Hence, the map $ \mathcal{I}_{(K_1,L_1)}^{0,r}\to \mathcal{I}_{(K_1,L_1)}^{0,r}\colon [(g,x)]^{0,r}\mapsto [(\varphi (g), f(x))]^{0,r} $ is a $ \varphi $-equivariant bijection.
\end{proof}

We again discuss the homomorphism $ \SU_3^{L/K}(K) \to \Aut (\mathcal{I}_{(K,L)}) $.
\begin{proposition}\label{Prop:Embedding G(K) in Aut(T) ur}
Let $ \mathcal{I} = \mathcal{I}_{(K,L)} $ be the Bruhat--Tits tree of $ \SU_3^{L/K}(K) $. The homomorphism~ $ \hat{}~\colon \SU_3^{L/K}(K)\to \Aut (\mathcal{I}) $ is continuous with closed image, and the kernel is equal to the centre of $ \SU_3^{L/K}(K) $. 
\end{proposition}
\begin{proof}
The proof is word for word the same as the proof of Proposition~\ref{Prop:Embedding G(K) in Aut(T)}, upon replacing $ \SL_2(D) $ by $ \SU_3^{L/K}(K) $.
\end{proof}

The convergence is then a more or less direct consequence of Theorem~\ref{Thm:localdescriptionoftheball ur}.
\begin{theorem}\label{Thm:continuity from Lur to Chab(Aut(T))}
Let $ ((K_i,L_i))_{i\in \mathbf{N}} $ be a sequence in $ \mathcal{L}^{\uram} $ which converges to $ (K,L) $, and let $G_i = \SU_3^{L_i/K_i}(K_i) $ $ ( $respectively $ G = \SU_3^{L/K}(K)) $. For $ N $ big enough and for all $ i\geq N $, there exist isomorphisms $ \mathcal{I}_{(K_i,L_i)} \cong \mathcal{I}_{(K,L)} $ such that the induced embeddings~ $ \hat{G_i}\hookrightarrow \Aut (\mathcal{I}_{(K,L)}) $ make $ (\hat{G_i})_{i\geq N} $ converge to $ \hat{G} $ in the Chabauty topology of $ \Aut (\mathcal{I}_{(K,L)}) $.
\end{theorem}
\begin{proof}
The Bruhat--Tits tree $ \mathcal{I}_{(K_i,L_i)} $ is the semiregular tree of bidegree ($ p^{3n} +1;p^n+1 $) if and only if $ (K_i,L_i) $ belongs to $ \mathcal{L}_{p^n}^{\uram} $. Hence there exists $ N $ such that for all $ i\geq N $, $  \mathcal{I}_{(K_i,L_i)}\cong  \mathcal{I}_{(K,L)} $.

Passing to a subsequence, we can assume that $ (K_i,L_i) $ is ($ i $)-close to $ (K,L) $. We now define an explicit isomorphism $ f_i\colon \mathcal{I}_{(K_i,L_i)}\to \mathcal{I}_{(K,L)} $ as follows: let $ \mathcal{I}_{(K_i,L_i)}^{0,i}\cong \mathcal{I}_{(K,L)}^{0,i} $ be the isomorphism given by Proposition~\ref{Prop:aritmimpliesgeomproximityforSU3 ur}. By Theorem~\ref{Thm:localdescriptionoftheball ur}, this gives an isomorphism on balls of radius $ i $: $ \mathcal{I}_{(K_i,L_i)}\supset B_0(i)\cong B_0(i)\subset \mathcal{I}_{(K,L)} $ (recall that by Lemma~\ref{Lem:B_0(r) is really the ball of radius r unramified}, $ B_0(i) $ is really the ball of radius $ i $ on the tree $ \mathcal{I}_{(K,L)} $). As $ \mathcal{I}_{(K_i,L_i)} $ is a semiregular tree of the same bidegree than $ \mathcal{I}_{(K,L)} $, we can extend this isomorphism of balls to an isomorphism $ f_{i}\colon \mathcal{I}_{(K_i,L_i)}\to \mathcal{I}_{(K,L)} $ (this extension is of course not unique, but we choose one such). By means of $ f_i $, we get an embedding $ \hat{G}_i\hookrightarrow \Aut ( \mathcal{I}_{(K,L)}) $.

Now the end of the proof is word for word the same as the corresponding end of the proof of Theorem~\ref{Thm:continuity from D to Chab(Aut(T))}, upon making the following changes: replace $ D_i $ with $(K_i,L_i)$, replace $D$ with $(K,L)$, replace $d$ with $1$, and replace all references to results in Section~\ref{Sec:SL_2(D)} by their corresponding results in Section~\ref{Sec:SU3L/K unramified}.
\end{proof}

We then deduce the proof of the main theorem announced in the introduction for groups of type $ \SU_3^{L/K} $ when $ L $ is unramified. To shorten the notations, we set $ G_{(K,L)} = \SU_3^{L/K}(K) $ in the following proof.
\begin{proof}[Proof of Theorem~\ref{Thm:explicit form of main theorem ur}]
Let $ T $ be a semiregular tree and let $ \mathcal{L}^{\uram}_{T} = \lbrace (K,L)\in \mathcal{L}^{\uram}~\vert~ $the Bruhat--Tits tree of $ G_{(K,L)} $ is isomorphic to $ T\rbrace $. By Remark~\ref{Rem:regularity of the B-T tree SU unramified} and Proposition~\ref{Prop:explicit description of L ur}, $ \mathcal{L}^{\uram}_{T} $ is a compact space. Now, by Theorem~\ref{Thm:continuity from Lur to Chab(Aut(T))}, the map $ \mathcal{L}^{\uram}_T\to\mathcal{S}_T\colon (K,L)\mapsto \hat{G}_{(K,L)} $ is continuous. Let $ (K_1,L_1) $ and $ (K_2,L_2) $ be unramified pairs in $ \mathcal{L}^{\uram} $. We claim that $ \hat{G}_{(K_1,L_1)} = \hat{G}_{(K_2,L_2)} $ if and only if $ (K_1,L_1) = (K_2,L_2) $. Indeed, if $ \hat{G}_{(K_{1},L_1)} $ is abstractly isomorphic to $ \hat{G}_{(K_2,L_2)} $, then by \cite{BoTi73}*{Corollaire~8.13}, the corresponding adjoint algebraic groups $ \Ad \mathbf{G}_1 $ and $ \Ad \mathbf{G}_2 $ are algebraically isomorphic over an isomorphism of fields $ K_1\cong K_2 $. Since $ \Ad \mathbf{G}_1 $ (respectively $ \Ad \mathbf{G}_2 $) is quasi-split, there exists a smallest extension splitting it (\cite{BrTi2}*{4.1.2}), namely $ L_1 $ (respectively $ L_2 $). Hence, $ (K_1,L_1)\cong (K_2,L_2) $, as wanted.

To summarise, $ \mathcal{L}^{\uram}_T\to\mathcal{S}_T\colon (K,L)\mapsto \hat{G}_{(K,L)} $ is an injective continuous map whose source is a compact space, hence it is a homeomorphism onto its image. Now, the explicit description given in Theorem~\ref{Thm:explicit form of main theorem ur} follows from Remark~\ref{Rem:regularity of the B-T tree SU unramified} and Proposition~\ref{Prop:explicit description of L ur}.
\end{proof}

\section{Convergence of groups of type \texorpdfstring{$ \SU_3^{L/K} $}{SU3L/K}, \texorpdfstring{$ L $}{L} ramified of odd residue characteristic}\label{Sec:SU3L/K ramified}
We keep our notations for local fields (see Section~\ref{Sec:Def of alg. grps}). Furthermore, throughout this section, $ L $ is a ramified quadratic extension of the base local field $ K $, and the residue characteristic is not $2$. Note that such an extension is automatically separable (because of the assumption on the residue characteristic) and we have $ L\cong K[T]/(T^2+\beta ) $ with $ \beta $ a uniformiser of $ K $. The results are very close to those of Section~\ref{Sec:SU3L/K unramified}.

\subsection{Construction of the Bruhat--Tits tree}
In the following definition of point stabilisers, we again use the notation introduced in Definition~\ref{Def:valuation of a matrix}.
\begin{definition}\label{Def:pointstabiliserforSU ramified}
For $ x\in \R $, we define $ P_{x} = \lbrace g\in \SU_{3}^{L/K}(K)~\vert~ \omega (g)\geq \begin{psmallmatrix}
0 & -\frac{x}{2} & -x \\ 
\frac{x}{2} & 0 & -\frac{x}{2} \\
x & \frac{x}{2}& 0
\end{psmallmatrix}\rbrace $
\end{definition}

\begin{definition}\label{Def:sbgrp N SU ramified}
Consider the following subsets
\begin{enumerate}[$ \bullet $]
\item $ T = \lbrace \begin{psmallmatrix}
x & 0 & 0 \\ 
0 & x^{-1}\bar{x} & 0 \\
0 & 0 & \bar{x}^{-1}
\end{psmallmatrix}~\vert~x\in L^{\times}\rbrace < \SU_3^{L/K} (K) $
\item $ M =\lbrace \begin{psmallmatrix}
0 & 0 & x \\ 
0 & -x^{-1}\bar{x} & 0 \\
\bar{x}^{-1} & 0 & 0
\end{psmallmatrix}~\vert~x\in L^{\times}\rbrace \subset \SU_3^{L/K} (K) $
\end{enumerate}
and let $ N = T\sqcup M $.
\end{definition}

\begin{definition}\label{Def:affineactionforN ramified}
Let $ \nu\colon N\to \Aff (\R) $ be defined as follows: for $ m = \begin{psmallmatrix}
0 & 0 & x \\ 
0 & -x^{-1}\bar{x} & 0 \\
\bar{x}^{-1} & 0 & 0
\end{psmallmatrix}\in M $, $ \nu (m) $ is the reflection through $ -\omega (x) $, while for $ t=\begin{psmallmatrix}
x & 0 & 0 \\ 
0 & x^{-1}\bar{x} & 0 \\
0 & 0 & \bar{x}^{-1}
\end{psmallmatrix}\in T $, $ \nu (t) $ is the translation by $ -2\omega (x) $.
\end{definition}

Then the Bruhat--Tits tree $ \mathcal{I} $ of $ \SU_3^{L/K} $ (recall that in this section, $ L $ is ramified and the residue characteristic is not $2$) is the one obtained by applying Definition~\ref{Def:buildingassociated to a data} to the collection of subgroups $ \lbrace (P_x)_{x\in \R}, N \rbrace $ appearing in Definition~\ref{Def:pointstabiliserforSU ramified} and Definition~\ref{Def:sbgrp N SU ramified}, together with the homomorphism $ \nu \colon N\to \Aff (\R) $ of Definition~\ref{Def:affineactionforN ramified}.  We show in Appendix~\ref{App:A} that our definitions agree with \cite{BrTi1}*{7.4.1~and~7.4.2}, so that the given data is indeed obtained from a valued root datum of rank one on $ G $.

\begin{remark}\label{Rem:dependence on D of the tree ram}
Note that the construction of the Bruhat--Tits tree of $ \SU_3^{L/K} $ depends on the pair $ (K,L) $. When needed, we keep track of this dependence by adding the subscript $ (K,L) $ to the objects involved. This gives rise to the notations $ (P_x)_{(K,L)} $, $ T_{(K,L)} $, $ M_{(K,L)} $, $ N_{(K,L)} $, $ \nu_{(K,L)} $ and $ \mathcal{I}_{(K,L)} $.
\end{remark}
\begin{remark}\label{Rem:regularity of the B-T tree SU ramified}
The Bruhat--Tits tree of $ \SU_3^{L/K} $ is actually the $ (\vert \overline{K}\vert+1)$-regular tree. Indeed, this follows from the fact that our definition of $ \mathcal{I} $ agrees with the one given in \cite{BrTi1}*{7.4.1 and 7.4.2}, and from the tables in \cite{Tits77}*{4.2 and 4.3}.
\end{remark}

\subsection{Local model of the Bruhat--Tits tree}
We now proceed to define a local model for the Bruhat--Tits tree of $ SU_3^{L/K} $ when $ L $ is ramified and the residue characteristic is not $ 2 $. The same remarks as in the $ SL_2(D) $ case apply, so that we go quickly through the definitions. Note that the valuation $ \omega $ (respectively the Galois conjugation $ x\mapsto\bar{x} $) on $ L $ induces a well-defined map on $ \mathcal{O}_L/\mathfrak{m}_L^r $ that we still denote $ \omega $ (respectively $ x\mapsto \bar{x} $).
\begin{definition}\label{Def:localdataSU(3)ram}
Let $ r\in \N $ and $ x\in [-\omega (\pi_{L}^{r}),\omega (\pi_{L}^{r})] $. We set $$ P_{x}^{0,r} = \lbrace g\in \SL_{3}(\mathcal{O}_{L}/\mathfrak{m}_{L}^{r})~\vert~ ^{S}\bar{g}g= \Id, \omega (g)\geq \begin{psmallmatrix}
0 & -\frac{x}{2} & -x \\
\frac{x}{2} & 0 & -\frac{x}{2} \\
x & \frac{x}{2} & 0
\end{psmallmatrix} \rbrace. $$
\end{definition}

\begin{definition}
\begin{enumerate}[$ \bullet $]
We define
\item $ H^{0,r} = \lbrace \begin{psmallmatrix}
x & 0 & 0 \\
0 & x^{-1}\bar{x} & 0 \\
0 &0 & \bar{x}^{-1}
\end{psmallmatrix}\in \SL_{3}(\mathcal{O}_{L}/\mathfrak{m}_{L}^{r})~\vert~ \omega (x) = 0 \rbrace $
\item $ M^{0,r} = \lbrace \begin{psmallmatrix}
0 & 0 & x \\
0 & -x^{-1}\bar{x} & 0 \\
\bar{x}^{-1} &0 & 0
\end{psmallmatrix}\in \SL_{3}(\mathcal{O}_{L}/\mathfrak{m}_{L}^{r})~\vert~ \omega (x) = 0 \rbrace $
\end{enumerate}
And we set $ N^{0,r} = H^{0,r}\sqcup M^{0,r} $.
\end{definition}

\begin{definition}
We let $ H^{0,r} $ act trivially on $ \R $, and we let all elements of $ M^{0,r} $ act as a reflection through $ 0\in \R $. This gives an affine action of $ N^{0,r} $ on $ \R $, and we denote again the resulting map $ N^{0,r}\to \Aff (\R) $ by $ \nu $.
\end{definition}

We are now able to give a definition of the ball of radius $ r $ around $ [(\Id ,0)]\in \mathcal{I} $ which only depends on the ring $ \mathcal{O}_L/\mathfrak{m}_L^{r} $, and not on the whole field $ L $.
\begin{definition}\label{Def:local tree of radius r ramified}
Let $ r\in \N $. We define an $ r $-local equivalence on $ P_{0}^{0,r}\times [-\omega (\pi_L^{r}), \omega (\pi_L^{r})] $ as follows. For $ g,h\in P_{0}^{0,r} $ and $ x,y\in [-\omega (\pi_L^{r}), \omega (\pi_L^{r})] $
\begin{equation*}
(g,x)\sim_{0,r}(h,y) \Leftrightarrow \textrm{ there exists } n\in N^{0,r} \textrm{ such that } \nu (n)(x)=y \textrm{ and } g^{-1}hn \in P_{x}^{0,r}
\end{equation*}
The resulting space $ \mathcal{I}^{0,r} = P_{0}^{0,r}\times [-\omega (\pi_L^{r}), \omega (\pi_L^{r})]/\sim_{0,r} $ is called the local Bruhat--Tits tree of radius $ r $ around $ 0 $, and $ [(g,x)]^{0,r} $ stands for the equivalence class of $ (g,x) $ in $ \mathcal{I}^{0,r} $. The group $ P_{0}^{0,r} $ acts on $ \mathcal{I}^{0,r} $ by multiplication on the first component.
\end{definition}

\begin{remark}\label{Rem:dependence on D of the local tree ram}
Note that the construction of the local Bruhat--Tits tree of $ \SU_3^{L/K} $ depends on the pair $ (K,L) $ (which is assumed to be ramified of odd residue characteristic in this section). When needed, we keep track of this dependence by adding the subscript $ (K,L) $ to the objects involved. This gives rise to the notations $ (P_x^{0,r})_{(K,L)} $, $ H^{0,r}_{(K,L)} $, $ M^{0,r}_{(K,L)} $, $ N^{0,r}_{(K,L)} $ and $ \mathcal{I}^{0,r}_{(K,L)} $.
\end{remark}
\begin{remark}\label{Rem:equivalent def with for all n for local model ramified}
Also, Remark~\ref{Rem:equivalent def with for all n for local model} holds equally well in this case, with exactly the same proof, upon replacing all $ d $'s by $ 1 $.
\end{remark}

\subsection{Integral model}
We now proceed to define an integral model, and compare its rational points with our local model for the Bruhat--Tits tree of $ SU_3^{L/K} $ when $ L $ is ramified and the residue characteristic is not $ 2 $. Again, the same remarks as in the $ SL_2(D) $ case apply, so that we go quickly through the definitions.

\begin{definition}\label{Def:integral model for ram SU}
Let $ \underline{\SU}_{3}^{L/K} $ be the group $ \SU_{3} $ considered over $ \mathcal{O}_{K} $. We omit the superscript $ L/K $ when it is understood from the context.
\end{definition}

Concretely, $ \underline{\SU}_{3} $ is the $ \mathcal{O}_{K} $-scheme associated with the $ \mathcal{O}_{K} $-algebra $ \mathcal{O}_{K}[\underline{\SU}_{3}] = \mathcal{O}_{K}[X_{ij}^{kl}]/I $ ($ i,j\in \lbrace 1,2,3\rbrace $, $ k,l\in \lbrace 1,2\rbrace $), where $ I $ is the ideal generated by the following equations 
\begin{flalign*}
& \text{For all }i,j\in \lbrace 1,2,3\rbrace,~
\begin{cases}
X_{ij}^{12} = -\beta X_{ij}^{21}\\
X_{ij}^{22} = X_{ij}^{11}
\end{cases} \\
& \sum\limits_{\sigma \in \Sym (3)}[ (-1)^{sgn(\sigma)}\prod\limits_{i=1}^3X_{i\sigma (i)} ]- 1\\
& \begin{psmallmatrix}
\overline{X}_{33} & \overline{X}_{23} & \overline{X}_{13} \\
\overline{X}_{32} & \overline{X}_{22} & \overline{X}_{12} \\
\overline{X}_{31} & \overline{X}_{21} & \overline{X}_{11} 
\end{psmallmatrix}\begin{psmallmatrix}
X_{11} & X_{12} & X_{13} \\
X_{21} & X_{22} & X_{23} \\
X_{31} & X_{32} & X_{33} 
\end{psmallmatrix} - \begin{psmallmatrix}
1 & 0 & 0 \\
0 &1 & 0 \\
0 & 0 & 1
\end{psmallmatrix}
\end{flalign*}
Here $ \beta $ is a uniformiser of $ \mathcal{O}_K $ such that $ L\cong K[T]/(T^2+\beta) $ ($ L $ is always of this form because it is a ramified extension and the residue characteristic is not 2), so that the first equations encode the ring embedding $ \mathcal{O}_{L}\hookrightarrow M_{2}(\mathcal{O}_{K}) $.  In the other equations, $ X_{ij} $ stands for the $ 2\times 2 $ matrix $ \begin{psmallmatrix}
X_{ij}^{11} & X_{ij}^{12} \\
X_{ij}^{21} & X_{ij}^{22} 
\end{psmallmatrix} $. Also, for a $ 2\times 2 $ matrix $ M = \begin{psmallmatrix}
M^{11} & M^{12} \\
M^{21} & M^{22} 
\end{psmallmatrix} $, we denote $ \overline{M} = \begin{psmallmatrix}
M^{22} & -M^{12} \\
-M^{21} & M^{11} 
\end{psmallmatrix} $ (this operation reflects the conjugation on $ \mathcal{O}_{L} $). Finally note that a $ 1 $ (respectively a $ 0 $) in the above equations denotes the $ 2\times 2 $ identity matrix (respectively the $ 2\times 2 $ zero matrix), i.e. it corresponds to the $ 1\in L $ (respectively $ 0\in L $).

\begin{theorem}\label{Thm:smoothness of SU3 ramified}
$ \underline{\SU}_{3}^{L/K} $ is a smooth $ \mathcal{O}_{K} $-scheme.
\end{theorem}
\begin{proof}
It suffices to prove that it is flat and that the fibres are smooth. The generic fibre is $ \SU_3^{L/K} $, and is a form of $ \SL_3 $, hence is smooth over $ K $. The closed fibre is the $ \overline{K} $-functor $ (\underline{\SU}_{3})_{\overline{K}} $ which associates to any $ \overline{K} $-algebra $ R $ the group 
\begin{equation*}
(\underline{\SU}_{3})_{\overline{K}}(R) = \lbrace g\otimes x\in \SL_3(R\otimes_{\overline{K}}\mathcal{O}_{L}/\mathfrak{m}_{L}^{2})~\vert ~^{S}(g\otimes \bar{x})(g\otimes x) = \Id \rbrace
\end{equation*}
Let $ \SO_3 $ be the special orthogonal group associated with the quadratic form $ (x_{-1},x_0,x_1)\mapsto x_{-1}x_1+x_0^{2} $, considered over $ \overline{K} $. More explicitly, for a $ \overline{K} $-algebra $ R $,
\begin{equation*}
(\SO_3)_{\overline{K}}(R) = \lbrace g\in \SL_3(R)~\vert~^{S}gg=\Id \rbrace
\end{equation*}
Recall that throughout this section, we are assuming that the characteristic of $ \overline{K} $ is not $ 2 $. It is then well known that $ \SO_3 $ is isomorphic to $ \PGL_2 $ over $ \overline{K} $, hence is smooth and connected of dimension $ 3 $. The homomorphism of $ \overline{K} $-algebra $ \mathcal{O}_{L}/\mathfrak{m}_{L}^{2}\to \overline{K} $ induces a homomorphism of algebraic groups $ f\colon (\underline{\SU}_{3})_{\overline{K}}\to (\SO_3)_{\overline{K}} $. The kernel of this map can be computed by hand, and we obtain that for any $ \overline{K} $-algebra $ R $,
\begin{align*}
\ker f (R) = \lbrace g\in \SL_{3}(R\otimes_{\overline{K}} \mathcal{O}_{L}/\mathfrak{m}_{L}^{2})~\vert~ g = \begin{psmallmatrix}
\begin{psmallmatrix}
1 & 0 \\
g_{11}^{21} & 1
\end{psmallmatrix} & \begin{psmallmatrix}
0 & 0 \\
g_{12}^{21} & 0
\end{psmallmatrix} & \begin{psmallmatrix}
0 & 0 \\
g_{13}^{21} & 0
\end{psmallmatrix} \\
\begin{psmallmatrix}
0 & 0 \\
g_{21}^{21} & 0
\end{psmallmatrix} & \begin{psmallmatrix}
1 & 0 \\
-2g_{11}^{21} & 1
\end{psmallmatrix} & \begin{psmallmatrix}
0 & 0 \\
g_{12}^{21} & 0
\end{psmallmatrix} \\
\begin{psmallmatrix}
0 & 0 \\
g_{31}^{21} & 0
\end{psmallmatrix} & \begin{psmallmatrix}
0 & 0 \\
g_{21}^{21} & 0
\end{psmallmatrix} & \begin{psmallmatrix}
1 & 0 \\
g_{11}^{21} & 1
\end{psmallmatrix}
\end{psmallmatrix} \rbrace
\end{align*}

This description makes it clear that $ \ker f $ is of dimension $ 5 $ and connected. Hence, using \cite{DG70}*{II, §5, Proposition~5.1} (note that it does not use smoothness), we conclude that $ \dim (\underline{\SU}_3)_{\overline{K}} = 8 $. But we can also easily compute that the Lie algebra of $ (\underline{\SU}_{3})_{\overline{K}} $ is
\begin{equation*}
(\mathfrak{su}_{3})_{\overline{K}} = \lbrace g\in M_{3}(\mathcal{O}_{L}/\mathfrak{m}_{L}^{2})~ \vert ~^{S}\bar{g}+g = 0, \trace (g)=0 \rbrace
\end{equation*}
This is readily seen to be of dimension $ 8 $ (recall that the residue characteristic is not $ 2 $), and hence, we conclude that $ (\underline{\SU}_3)_{\overline{K}} $ is smooth, as wanted. Also note that the homomorphism $ f\colon (\underline{\SU}_3)_{\overline{K}}\to (\SO_3)_{\overline{K}} $ is surjective onto a connected algebraic group, with connected kernel, hence $ (\underline{\SU}_3)_{\overline{K}} $ is also connected (and hence irreducible).

It remains to prove flatness, which is always a delicate matter. Since we already know that the fibres of $ \underline{\SU}_3 $ over $ \Spec \mathcal{O}_K $ are smooth, irreducible and of the same dimension, it suffices to give a closed embedding $ \mathbf{A}_{\mathcal{O}_K}^{1}\hookrightarrow \underline{\SU}_3 $ (see Lemma~\ref{Lem:Proving flatness of schemes over local rings}). Let $ \pi_L $ be a uniformiser of $ L $ such that $ \overline{\pi}_L = -\pi_L $. The morphism of schemes $ \mathbf{A}_{\mathcal{O}_K}^{1}\to \underline{\SU}_{3}\colon x\mapsto 
\begin{psmallmatrix} 1&0&x.\pi_L\\
0&1&0\\
0&0&1
\end{psmallmatrix} $ is such a closed embedding, which concludes the proof.
\end{proof}

\begin{lemma}\label{Lem:Proving flatness of schemes over local rings}
Let $ K $ be a local field, let $ X $ be an affine scheme of finite type over $ \mathcal{O}_K $ and let $ \mathbf{A}_{\mathcal{O}_K}^1 $ be the affine line over $ \mathcal{O}_K $. Assume that $ X_K $ and $ X_{\overline{K}} $ are smooth and irreducible of the same dimension. Further assume that there exists a closed embedding of $ \mathcal{O}_K $-schemes $ \mathbf{A}_{\mathcal{O}_K}^1\hookrightarrow X $. Then $ X $ is flat over $ \mathcal{O}_K $.
\end{lemma}
\begin{proof}
Let $ Y $ be the schematic adherence of $ X_K $ in $ X $. By definition, $ Y $ is flat with $ X_K = Y_K $, and it suffices to prove that $ Y_{\overline{K}} = X_{\overline{K}} $. 

Since $ \mathbf{A}_{\mathcal{O}_K}^1 $ is flat, $ (\mathbf{A}_{\mathcal{O}_K}^1)_{K} $ is dense in $ \mathbf{A}_{\mathcal{O}_K}^1 $, so that $ Y_{\overline{K}} $ is non-empty. Now, flatness of $ Y $ implies that the irreducible components of $ Y_{\overline{K}} $ have the same dimension than $ Y_K = X_K $ (see for example \cite{stacks-project}*{Tag 02NM}). Also, these irreducible components are closed subschemes of $ X_{\overline{K}} $, which is a smooth irreducible (hence integral) affine scheme. Hence, by Lemma~\ref{Lem:little algerbo-geometric lemma}, we conclude that $ Y_{\overline{K}} = X_{\overline{K}} $, as wanted.
\end{proof}
\begin{lemma}\label{Lem:little algerbo-geometric lemma}
Let $ k $ be a field, and let $ X $ be an integral finite type affine $ k $-scheme of dimension $ n $. Let $ Y $ be a closed subscheme of $ X $ and assume that $ Y $ has an irreducible component $ Z $ of dimension $ n $. Then $ Z=Y=X $ (as $ k $-schemes).
\end{lemma}
\begin{proof}
We follow the argument given in \cite{El13}: the composite $ k[X]\twoheadrightarrow k[Y]\twoheadrightarrow k[Z] $ has its kernel contained in the nilradical of $ k[X] $ (by Krull's principal ideal theorem), which shows that $ R\twoheadrightarrow k[Z] $ is injective as well (because being a domain, $ R $ has in particular a trivial nilradical).
\end{proof}
\begin{remark}\label{Rem:reductive quotient of the closed fibre}
In passing, note that the group $ (\underline{\SU}_{3}^{L/K})_{\overline{K}} $ is not a reductive group over $ \overline{K} $ (as predicted by \cite{BrTi2}*{4.6.31}). In fact, we have just showed in the above proof that its reductive quotient is naturally described as the orthogonal group in $ 3 $ variables. 
Again, this might be seen as a reason for the complication of the ramified, residue characteristic $ 2 $ case, since philosophically, it involves orthogonal group in characteristic $ 2 $. Also note that $ (\underline{\SU}_3)_{\mathcal{O}_{L}} $ is not isomorphic to $ \SL_3 $ over $ \mathcal{O}_{L} $, as opposed to what happens in the unramified case. Indeed, if it were, then the closed fibre $ (\underline{\SU}_3)_{\overline{K}} $ would be isomorphic to $ \SL_3 $ over $ \overline{K}\cong \mathcal{O}_{L}/\mathfrak{m}_{L} $, which is not true, as we have just seen in the above proof.
\end{remark}

We now compare the rational points of the integral model with our local model.
\begin{lemma}\label{Lem:description of the restriction P_0 to P_0^0,r ramified}
$ \underline{\SU}_3(\mathcal{O}_{K}) \cong P_{0} $ and $ \underline{\SU}_3(\mathcal{O}_{K}/\mathfrak{m}_{K}^{r}) \cong P_{0}^{0,2r} $. Following the identifications
\begin{center}
\begin{tikzpicture}[->]
 \node (1) at (0,0) {$ \underline{\SU}_3(\mathcal{O}_{K}) $};
 \node (2) at (2.3,0) {$ P_{0}\leq \SL_{3}(\mathcal{O}_{L}) $};
 \node (5) at (1,0) {$ \cong $};
 \node (3) at (-0.3,-1.2) {$ \underline{\SU}_3(\mathcal{O}_{K}/\mathfrak{m}_{K}^{r}) $};
 \node (4) at (2.9, -1.2) {$ P_{0}^{0,2r}\leq \SL_3(\mathcal{O}_{L}/\mathfrak{m}_{L}^{2r}) $};
 \node (6) at (1,-1.2) {$ \cong $};
 
 \draw[->] (0,-0.3) to (0,-0.9);
 \draw[->] (1.4,-0.3) to (1.4,-0.9);
\end{tikzpicture}
\end{center}
the homomorphism $ p_{2r}\colon P_{0}\to P_{0}^{0,2r} $ is the one induced by the projection of the coefficients $ \mathcal{O}_{L}\to \mathcal{O}_{L}/\mathfrak{m}_{L}^{2r} $.
\end{lemma}
\begin{proof}
Let $ \pi_L $ be a uniformiser of $ \mathcal{O}_L $ such that $ \pi_L^2 = -\beta $ is a uniformiser of $ \mathcal{O}_K $. Using the fact that $ \mathcal{O}_{L}\cong \mathcal{O}_{K}\oplus \pi_L.\mathcal{O}_{K} $, one can check that $ \Mor_{\mathcal{O}_{K}}(\mathcal{O}_{K}[\underline{\SU}_{3}],\mathcal{O}_{K}) \cong \lbrace g\in \SL_{3}(\mathcal{O}_{L})~ \vert ~ ^{S}\bar{g}g=\Id \rbrace $, as wanted.

Furthermore, since $ L $ is ramified, $ \mathcal{O}_{L}/\mathfrak{m}_{L}^{2r}\cong \mathcal{O}_{K}/\mathfrak{m}_{K}^{r}\oplus \pi_L.\mathcal{O}_{K}/\mathfrak{m}_{K}^{r} $, and one can check that $ \Mor_{\mathcal{O}_{K}}(\mathcal{O}_{K}[\underline{\SU}_{3}],\mathcal{O}_{K}/\mathfrak{m}_{K}^{r}) \cong \lbrace g\in \SL_{3}(\mathcal{O}_{L}/\mathfrak{m}_{L}^{2r})~ \vert ~ ^{S}\bar{g}g=\Id \rbrace $, as wanted.
\end{proof}

And we can then deduce the surjectivity of the map $ p_{2r} $.
\begin{corollary}\label{Cor:surjectivityofpr ramified even}
The map $ p_{2r}\colon P_{0}\to P_{0}^{0,2r} $ is surjective, for all $ r\in \N $.
\end{corollary}
\begin{proof}
This is a direct consequence of the commutative square involving $ P_{0}\to P_{0}^{0,2r} $ given in Lemma~\ref{Lem:description of the restriction P_0 to P_0^0,r ramified}, together with the fact that the integral model is smooth by Theorem~\ref{Thm:smoothness of SU3 ramified}, so that Theorem~\ref{Thm:genhensel} applies to the left hand side of the diagram.
\end{proof}
While this is not necessary for our work on Chabauty limits, we nevertheless give a proof that the map $ p_r $ is also surjective for odd $ r $. 
\begin{corollary}\label{Cor:surjectivityofpr ramified}
The map $ p_{r}\colon P_{0}\to P_{0}^{0,r} $ is surjective, for all $ r\in \N $.
\end{corollary}
\begin{proof}
By Corollary~\ref{Cor:surjectivityofpr ramified even}, it suffices to prove that $ P_{0}^{0,2(r+1)}\to P_{0}^{0,2r+1} $ is surjective for all $ r\in \N $. Note that the composite $ P_{0}^{0,2(r+1)} \xrightarrow{f_r} P_{0}^{0,2r+1} \xrightarrow{g_r} P_{0}^{0,2r}$ is surjective for all $ r\in \N $, and that the groups are finite. Hence, it suffices to prove that $ \vert \ker (f_r\circ g_r)\vert = \vert \ker (f_r)\vert.\vert \ker (g_r)\vert $. For a ring $ R $, let $ M_3(R) $ denote the $ 3 $-by-$ 3 $ matrices with coefficients in $ R $. Now, for $ r\geq 1 $, we have
\begin{align*}
\ker (f_r\circ g_r) \cong \lbrace g\in M_3(\mathcal{O}_L/\mathfrak{m}_L^2)~\vert~^S\bar{g}+g=0,\trace(g)=0\rbrace\\
\ker (f_r) \cong \lbrace g\in M_3(\mathcal{O}_L/\mathfrak{m}_L)~\vert~^Sg=g,\trace(g)=0\rbrace\\
\ker (g_r) \cong \lbrace g\in M_3(\mathcal{O}_L/\mathfrak{m}_L)~\vert~^Sg+g=0,\trace(g)=0\rbrace
\end{align*}
Since the residue characteristic is not $ 2 $, the result holds. Finally, the surjectivity of $ P_0^{0,2}\to P_0^{0,1} $ is a direct consequence of the splitting of the ring homomorphism $ \mathcal{O}_L/\mathfrak{m}_L^2\to \mathcal{O}_L/\mathfrak{m}_L $.
\end{proof}

We also need a kind of injectivity result:
\begin{lemma}\label{Lem:injectivityofpr ramified}
Let $ r\in \N $ and $ x\in [-\omega (\pi_{L}^{r}),\omega (\pi_{L}^{r})] $. Then $ p_{r}^{-1}(P_{x}^{0,r}) \subset P_{x} $.
\end{lemma}
\begin{proof}
Belonging to $ p_{r}^{-1}(P_x^{0,r}) $ implies that the valuation of the off diagonal entries are big enough. Hence, the result follows directly from Definition~\ref{Def:pointstabiliserforSU ramified}.
\end{proof}

We finally arrive at the result corresponding to Theorem~\ref{Thm:localdescriptionoftheball}: the ball of radius $ r $ together with the action of $ \SU_{3}^{L/K}(K) $ is encoded in $ P_{0}^{0,r} $. We first need an adequate description of the ball of radius $ r $ around $0$ in $ \mathcal{I} $.

\begin{lemma}\label{Lem:B_0(r) is really the ball of radius r ramified}
Renormalise the distance on $ \R $ so that $ d_{\R}(0;\omega(\pi_L)) = 1 $, and put the metric $ d_{\mathcal{I}} $ on $ \mathcal{I} $ arising from the distance $ d_{\R} $ $ ( $see Remark~\ref{Rem:metric on I and stabiliser}$ ) $. Let $ B_0(r) = \lbrace p\in \mathcal{I}~\vert~d_{\mathcal{I}}([(\Id ,0)];p)\leq r\rbrace $ be the ball of radius $ r $ around $ 0 $ in $ \mathcal{I} $. Let $ \tilde{B}_0(r) = \lbrace [(g,x)]\in \mathcal{I} ~\vert~ g\in P_{0}, x\in [-\omega (\pi_L^{r}), \omega (\pi_L^{r})]\subset \R \rbrace $. Then $ B_0(r) = \tilde{B}_0(r) $.
\end{lemma}
\begin{proof}
The proof is word for word the same than the proof of Lemma~\ref{Lem:B_0(r) is really the ball of radius r}, upon replacing all $d$'s by $1$'s.
\end{proof}
\begin{remark}
The distance $ d_{\mathcal{I}} $ that we introduced in Lemma~\ref{Lem:B_0(r) is really the ball of radius r ramified} is also the combinatorial distance on the tree. Indeed, looking at when $ P_y $ is inside $ P_x $ for $ x,y\in \R $, we see that $ [(\Id ,x)] $ is a vertex of $ \mathcal{I} $ if and only if $ x\in \omega(\pi_L)\Z $.
\end{remark}

\begin{theorem}\label{Thm:localdescriptionoftheball ramified}
Let $ r\in \N $. The map $ B_{0}(r)\to \mathcal{I}^{0,r}\colon [(g,x)]\mapsto [(p_{r}(g),x)]^{0,r} $ is a $ (p_{r}\colon P_{0}\to P_{0}^{0,r}) $-equivariant bijection.
\end{theorem}
\begin{proof}
The proof is word for word the same than the proof of Theorem~\ref{Thm:localdescriptionoftheball ur}, upon replacing all references to results in Section~\ref{Sec:SU3L/K unramified} with their analogues in this section.
\end{proof}

\subsection{Arithmetic convergence}
\begin{definition}\label{Def:space L ramified}
Consider the set of ramified pairs $(K,L)$ where $K$ is a local field of odd residue characteristic and $L$ is a ramified (separable) quadratic extension of $K$. We say that two ramified pairs $ (K_{1},L_{1}) $ and $ (K_{2},L_{2}) $ are isomorphic if there exists a conjugation equivariant isomorphism between $ L_1 $ and $ L_2 $, and we let $ \mathcal{L}^{\ram}_{\odd} $ be the set of ramified pairs in odd residue characteristic, up to isomorphism. 
For an odd prime power $ p^n $, we also define $ \mathcal{L}^{\ram}_{p^{n}} = \lbrace (K,L) \in \mathcal{L}^{\ram}_{\odd}~\vert~ \vert \overline{K}\vert = p^{n}\rbrace $.
\end{definition}

As in Section~\ref{Sec:SL_2(D)}, we define a metric on the space $ \mathcal{L}^{\ram}_{\odd} $. For $L\in \mathcal{L}^{\ram}_{\odd} $ and $ r\in \N $, the Galois conjugation induces an automorphism of $ \mathcal{O}_{L}/\mathfrak{m}_{L}^{r} $ that we still call the conjugation.
\begin{definition}
Let $ (K_{1},L_{1}) $ and $ (K_{2},L_{2}) $ be in $ \mathcal{L}^{\ram}_{\odd} $. We say that $ (K_{1},L_{1}) $ is $ r $-close to $ (K_2,L_2) $ if and only if there exists a conjugation equivariant isomorphism $ \mathcal{O}_{L_{1}}/\mathfrak{m}_{L_1}^{r}\to \mathcal{O}_{L_{2}}/\mathfrak{m}_{L_2}^{r} $.
\end{definition}
\begin{remark}\label{Rem:induced proximity on K ram}
A conjugation equivariant isomorphism $ \mathcal{O}_{L_{1}}/\mathfrak{m}_{L_1}^{r}\to \mathcal{O}_{L_{2}}/\mathfrak{m}_{L_2}^{r} $ always induces an isomorphism $ \mathcal{O}_{K_{1}}/\mathfrak{m}_{K_1}^{\lceil \frac{r}{2}\rceil}\to \mathcal{O}_{K_{2}}/\mathfrak{m}_{L_2}^{\lceil \frac{r}{2}\rceil} $, since $ \mathcal{O}_{K_{i}}/\mathfrak{m}_{K_i}^{\lceil \frac{r}{2}\rceil} $ is the invariant subring of $ \mathcal{O}_{L_{i}}/\mathfrak{m}_{L_i}^{r} $.
\end{remark}
\begin{remark}\label{Rem:one accumulation point ram}
For any uniformiser $ \beta \in \mathbf{F}_{p^{n}}(\!(X)\!) $, the pair $ (\mathbf{F}_{p^{n}}(\!(X)\!),\mathbf{F}_{p^{n}}(\!(X)\!)[T]/(T^{2}+\beta ))$ is isomorphic to the pair $ (\mathbf{F}_{p^{n}}(\!(X)\!),\mathbf{F}_{p^{n}}(\!(\sqrt{X})\!)) $ (because $ \mathbf{F}_{p^{n}}(\!(X)\!) $ has many automorphisms). Hence, despite the fact that there are two non-isomorphic quadratic ramified extensions of $ \mathbf{F}_{p^{n}}(\!(X)\!) $, there is only one ramified pair of positive characteristic in $ \mathcal{L}_{p^n}^{\ram} $.
\end{remark}
\begin{remark}\label{Rem:transitivityofdistance  ramified}
Note that being $ r $-close is an equivalence relation, and that if $ r\geq l $ and $ (K_1,L_1) $ is $ r $-close to $ (K_2,L_2) $, then $ (K_1,L_1) $ is $ l $-close to $ (K_2,L_2) $.
\end{remark}

Again, this notion of closeness induces a non-archimedean metric on $ \mathcal{L}^{\ram}_{\odd} $. Let
\begin{equation*}
d\colon \mathcal{L}^{\ram}_{\odd}\times \mathcal{L}^{\ram}_{\odd}\to \mathbf{R}_{\geq 0}\colon d((K_1,L_1);(K_2,L_2))= \inf \lbrace \frac{1}{2^{r}} ~\vert~ (K_1,L_1) \textrm{ is } r \textrm{-close to }(K_2,L_2) \rbrace
\end{equation*}

\begin{lemma}\label{Lem:Non-archim. metric space ram}
$ d(\cdot~;~\cdot ) $ is a non-archimedean metric on $ \mathcal{L}^{\ram}_{\odd} $. 
\end{lemma}
\begin{proof}
If $ d((K_1,L_1);(K_2,L_2)) = 0 $, then $ \mathcal{O}_{L_1} $ and $ \mathcal{O}_{L_2} $ are equivariantly isomorphic. 
Hence, the pairs of field of fraction are isomorphic in $ \mathcal{L}^{\ram}_{\odd} $, as wanted. The fact that this distance is non-archimedean is a consequence of Remark~\ref{Rem:transitivityofdistance  ramified}.
\end{proof}

We now go on to prove that $ \mathcal{L}_{p^n}^{\ram} $ is homeomorphic to $ \hat{\N} $. Again, the key ingredient in this identification is Theorem~\ref{Thm:fundamental approximation lemma}. We further need a variation for ramified quadratic extension in odd residue characteristic.
\begin{corollary}\label{Cor:variation on fundamental lemma ram}
Let $ K $ be a totally ramified extension of degree $ k $ of $ \mathbf{Q}_{p^{n}} $, where $ p $ is an odd prime, and let $ L $ be a ramified quadratic extension of $ K $. The distance between $ (K, L) $ and $ (\mathbf{F}_{p^{n}}(\!(X)\!),\mathbf{F}_{p^{n}}(\!(\sqrt{X})\!)) $ is $ \frac{1}{2^{2k}} $.
\end{corollary}
\begin{proof}
Let $ t $ be a uniformiser of $L$ such that $ t^2 = \pi_K $ is a uniformiser of $ K $ (such a $ t $ exists because the residue characteristic is odd). Since we have the following equivariant isomorphisms of rings
\begin{align*}
\mathcal{O}_{L}/\mathfrak{m}_{L}^{2k}&\cong \mathcal{O}_{K}/\mathfrak{m}_{K}^{k}\oplus t.\mathcal{O}_{K}/\mathfrak{m}_{K}^{k}\\
\mathbf{F}_{p^{n}}[\![\sqrt{X}]\!]/(\sqrt{X}^{2k})&\cong \mathbf{F}_{p^{n}}[\![X]\!]/(X^k)\oplus \sqrt{X}.\mathbf{F}_{p^{n}}[\![X]\!]/(X^k)
\end{align*} it is clear (in view of Theorem~\ref{Thm:fundamental approximation lemma}) that $ (K, L) $ is $ 2k $-close to $ (\mathbf{F}_{p^{n}}(\!(X)\!),\mathbf{F}_{p^{n}}(\!(\sqrt{X})\!)) $.

To conclude that the distance is $ \frac{1}{2^{2k}} $ it suffices to note that if $ (K, L) $ and $ (\mathbf{F}_{p^{n}}(\!(X)\!),\mathbf{F}_{p^{n}}(\!(\sqrt{X})\!)) $ were $ r $-close for $ r > 2k $, then $ K $ and $ \mathbf{F}_{p^{n}}(\!(X)\!) $ would be $ \lceil \frac{r}{2}\rceil $-close as well by Remark~\ref{Rem:induced proximity on K ram}, contradicting Theorem~\ref{Thm:fundamental approximation lemma}.
\end{proof}

We deduce the homeomorphism type of $ \mathcal{L}_{p^{n}}^{\ram} $.
\begin{proposition}\label{Prop:explicit description of L ram}
Let $ p $ be an odd prime number. Then $ \mathcal{L}_{p^{n}}^{\ram} $ is homeomorphic to $ \hat{\N} $.
\end{proposition}

\begin{proof} \setcounter{claim}{0}
$  $

\begin{claim}\label{Claim:1 ram}
Let $ (K,L)\in \mathcal{L}_{p^{n}}^{\ram} $. If $ K $ is of characteristic $ 0 $, $ (K,L) $ is isolated in $ \mathcal{L}_{p^{n}}^{\ram} $.
\end{claim}
\begin{claimproof}
If $ (K_1,L_1) $ is $ r $-close to $ (K_2,L_2) $, then $ K_1 $ is $ \lceil \frac{r}{2}\rceil $-close to $ K_2 $ by Remark~\ref{Rem:induced proximity on K ram}. Hence, the result follows from Claim~\ref{Claim:2'} and Claim~\ref{Claim:3'} in the proof of Proposition~\ref{Prop:explicit description of D}.
\end{claimproof}

\medskip

\begin{claim}\label{Claim:2 ram} 
$ \mathcal{L}_{p^{n}}^{\ram} $ is a countable space.
\end{claim}
\begin{claimproof}
By Claim~\ref{Claim:3'} in the proof of Proposition~\ref{Prop:explicit description of D}, there are only countably many pairs of characteristic $ 0 $ in $ \mathcal{L}_{p^{n}}^{\ram} $. Furthermore, there is only one ramified pair of characteristic $p$ in $ \mathcal{L}_{p^{n}}^{\ram} $ by Remark~\ref{Rem:one accumulation point ram}.
\end{claimproof}

\medskip

We are now able to deduce the homeomorphism type of $ \mathcal{L}_{p^{n}}^{\ram} $: ramified pairs of characteristic~$ 0 $ are isolated by Claim~\ref{Claim:1 ram}, and the ramified pair of positive characteristic is an accumulation point in $ \mathcal{L}_{p^{n}}^{\ram} $ by Corollary~\ref{Cor:variation on fundamental lemma ram}. Hence, by \cite{MS20}*{Théorème~1}, $ \mathcal{L}_{p^{n}}^{\ram} $ is homeomorphic to $ \hat{\N} $.
\end{proof}

\subsection{Continuity from pairs in \texorpdfstring{$ \mathcal{L}^{\ram}_{\odd} $}{L^ram_odd} to subgroups of \texorpdfstring{$ \Aut (T) $}{Aut(T)}}
In this section, we start to vary the ramified pair $ (K,L) $, and look at the variation it produces on the Bruhat--Tits tree of $ \SU_3^{L/K} $. Recall that we introduced a notation to keep track of the dependence on $ (K,L) $ of many of the definitions we made in this section (see Remark~\ref{Rem:dependence on D of the tree ram} and Remark~\ref{Rem:dependence on D of the local tree ram}).

\begin{proposition}\label{Prop:aritmimpliesgeomproximityforSU3 ram}
Let $ (K_1,L_1) $ and $ (K_2,L_2) $ be two elements in $ \mathcal{L}^{\ram}_{\odd} $. Assume that $ (K_1,L_1) $ is $ r $-close to $ (K_2,L_2) $, for some $ r \in \N $. Then, $ (P_{0}^{0,r})_{(K_1,L_1)}\cong (P_{0}^{0,r})_{(K_2,L_2)} $, and $ \mathcal{I}_{(K_1,L_1)}^{0,r}$ is equivariantly in bijection with $ \mathcal{I}_{(K_2,L_2)}^{0,r} $.
\end{proposition}
\begin{proof}
Parallel the proof of Proposition~\ref{Prop:aritmimpliesgeomproximityforSU3 ur}.
\end{proof}

\begin{proposition}\label{Prop:Embedding G(K) in Aut(T) ram}
Let $ \mathcal{I} = \mathcal{I}_{(K,L)} $ be the Bruhat--Tits tree of $ \SU_3^{L/K}(K) $. The homomorphism~ $ \hat{}~\colon \SU_3^{L/K}(K)\to \Aut (\mathcal{I}) $ is continuous with closed image, and the kernel is equal to the centre of $ \SU_3^{L/K}(K) $. 
\end{proposition}
\begin{proof}
The proof is word for word the same as the proof of Proposition~\ref{Prop:Embedding G(K) in Aut(T)}, upon replacing $ \SL_2(D) $ by $ \SU_3^{L/K}(K) $.
\end{proof}

\begin{theorem}\label{Thm:continuity from Lram to Chab(Aut(T))}
Let $ ((K_i,L_i))_{i\in \mathbf{N}} $ be a sequence in $ \mathcal{L}^{\ram}_{\odd} $ which converges to $ (K,L) $, and let $G_i = \SU_3^{L_i/K_i}(K_i) $ $ ( $respectively $ G = \SU_3^{L/K}(K)) $. For $ N $ big enough and for all $ i\geq N $, there exist isomorphisms $ \mathcal{I}_{(K_i,L_i)} \cong \mathcal{I}_{(K,L)} $ such that the induced embeddings~ $ \hat{G_i}\hookrightarrow \Aut (\mathcal{I}_{(K,L)}) $ make $ (\hat{G_i})_{i\geq N} $ converge to $ \hat{G} $ in the Chabauty topology of $ \Aut (\mathcal{I}_{(K,L)}) $.
\end{theorem}
\begin{proof}
The Bruhat--Tits tree $ \mathcal{I}_{(K_i,L_i)} $ is the regular tree of degree $ p^n+1 $ if and only if $ (K_i,L_i) $ belongs to $ \mathcal{L}_{p^n}^{\ram} $. Hence there exists $ N $ such that for all $ i\geq N $, $  \mathcal{I}_{(K_i,L_i)}\cong  \mathcal{I}_{(K,L)} $.

Then the end of the proof parallels the proof of Theorem~\ref{Thm:continuity from Lur to Chab(Aut(T))}.
\end{proof}

We then deduce the proof of the main theorem announced in the introduction for groups of type $ \SU_3^{L/K} $ when $ L $ is ramified and of odd residue characteristic. To shorten the notations, we set $ G_{(K,L)} = \SU_3^{L/K}(K) $ and $ G_{K} = \SL_2(K) $ in the following proof.
\begin{proof}[Proof of Theorem~\ref{Thm:explicit form of main theorem odd}]
Let $ T $ be the ($ p^n+1 $)-regular tree. Paralleling the proof of Theorem~\ref{Thm:explicit form of main theorem ur}, we see that the maps $ \mathcal{L}^{\ram}_{p^n}\to\mathcal{S}_T\colon (K,L)\mapsto \hat{G}_{(K,L)} $ and $ \mathcal{K}_{p^n}\to\mathcal{S}_T\colon K\mapsto \hat{G}_{K} $ are injective continuous maps whose source is a compact space, hence they are homeomorphisms onto their respective image. Now, the explicit description given in Theorem~\ref{Thm:explicit form of main theorem odd} follows from Remark~\ref{Rem:regularity of the B-T tree SU ramified} and Proposition~\ref{Prop:explicit description of L ram}. 
\end{proof}

\section{Convergence of groups of type \texorpdfstring{$ \SU_3^{L/K} $}{SU3L/K}, \texorpdfstring{$ L $}{L} ramified of residue characteristic \texorpdfstring{$ 2 $}{2}}\label{Sec:SU3L/K ramifiedeven}
We keep our notations for local fields (see Section~\ref{Sec:Def of alg. grps}). Furthermore, throughout this section, $ L $ is a separable quadratic ramified extension of the base local field $ K $, and the residue characteristic is $2$. We carry out the same program than in previous sections, but with more technicalities to overcome. The comments made all along Section~\ref{Sec:SL_2(D)} also apply here, and we do not repeat them to not lengthen too much the paper.

\subsection{Construction of the Bruhat--Tits tree}
The definition of $ P_{x} $ is less straightforward in this case. Following \cite{BrTi2}*{4.3.3}, we define a parameter that handles the complication.

\begin{lemma}\label{Lem:goodformforL}
Let $ L/K $ be as in this section. There exists $ t\in L $ and $ \alpha ,\beta \in K $ such that:
\begin{enumerate}
\item  $ L=K[t] $ and $ t^{2}-\alpha t+\beta = 0 $.
\item $ \beta $ is a uniformiser of $ K $.
\item $ \alpha = 0 $, or $ 0<\omega (\beta)\leq \omega (\alpha)\leq \omega (2) $.
\end{enumerate}
\end{lemma}
\begin{proof}
See \cite{BrTi2}*{Lemme~4.3.3, (ii)}. 
\end{proof}

Note that $ \alpha = 0 $ implies $ 2\neq 0 $ in $ K $, since $ L $ is assumed to be a separable extension.
\begin{definition}\label{Def:gamma}
Let $ L/K $ be as in this section, and let $ t, \alpha ,\beta $ be as in Lemma~\ref{Lem:goodformforL}. Let $ l = t\alpha^{-1} \in L $ if $ \alpha \neq 0 $, and $ l = \frac{1}{2} \in L $ if $ \alpha = 0 $, where $ \alpha $ is as in Lemma~\ref{Lem:goodformforL}. We then define $ \gamma = -\frac{1}{2}\omega (l) \in \R $.
\end{definition}

\begin{remark}
Note that since the residue characteristic is $ 2 $, $ \gamma > 0 $.
\end{remark}

In fact, the parameter $ \gamma $ associated with the extension $ L/K $ only depends on the normalisation of the valuation on $ K $.
\begin{proposition}
Let $ L/K $ be as in this section. Then the parameter $ \gamma $ introduced in Definition~\ref{Def:gamma} does not depend on the choices of $ t,\alpha $ and $ \beta $. We call $ \gamma $ the parameter associated with the extension $ L $ of $ K $.
\end{proposition}
\begin{proof}
This is a direct corollary of the work of Bruhat--Tits. Indeed, according to \cite{BrTi2}*{Proposition~4.3.3, (ii)}, the element $ l $ appearing in Definition~\ref{Def:gamma} has a maximal valuation amongst elements of $ L $ of trace $ 1 $.
\end{proof}

In the following definition of point stabilisers, we again use the notation introduced in Definition~\ref{Def:valuation of a matrix}.
\begin{definition}\label{Def:pointstabiliserforSU ramifiedeven}
Let $ \gamma $ be the parameter associated with the extension $ L $ of $ K $ as in Definition~\ref{Def:gamma}. For $ x\in \R $, we define 
\begin{equation*}
P_{x} = \lbrace g\in \SU_{3}^{L/K}(K)~\vert~ \omega (g)\geq \begin{psmallmatrix}
0 & -\frac{x}{2} - \gamma & -x \\ 
\frac{x}{2} + \gamma & 0 & -\frac{x}{2} + \gamma \\
x & \frac{x}{2} - \gamma & 0
\end{psmallmatrix}\rbrace
\end{equation*}
\end{definition}


\begin{definition}\label{Def:sbgrp N SU ramifiedeven}
Consider the following subsets
\begin{enumerate}[$ \bullet $]
\item $ T = \lbrace \begin{psmallmatrix}
x & 0 & 0 \\ 
0 & x^{-1}\bar{x} & 0 \\
0 & 0 & \bar{x}^{-1}
\end{psmallmatrix}~\vert~x\in L^{\times}\rbrace < \SU_3^{L/K} (K) $
\item $ M =\lbrace \begin{psmallmatrix}
0 & 0 & x \\ 
0 & -x^{-1}\bar{x} & 0 \\
\bar{x}^{-1} & 0 & 0
\end{psmallmatrix}~\vert~x\in L^{\times}\rbrace \subset \SU_3^{L/K} (K) $
\end{enumerate}
and let $ N = T\sqcup M $.
\end{definition}

\begin{definition}\label{Def:affineactionforN ramifiedeven}
Let $ \nu\colon N\to \Aff (\R) $ be defined as follows: for $ m = \begin{psmallmatrix}
0 & 0 & x \\ 
0 & -x^{-1}\bar{x} & 0 \\
\bar{x}^{-1} & 0 & 0
\end{psmallmatrix}\in M $, $ \nu (m) $ is the reflection through $ -\omega (x) $, while for $ t=\begin{psmallmatrix}
x & 0 & 0 \\ 
0 & x^{-1}\bar{x} & 0 \\
0 & 0 & \bar{x}^{-1}
\end{psmallmatrix}\in T $, $ \nu (t) $ is the translation by $ -2\omega (x) $.
\end{definition}

Then the Bruhat--Tits tree $ \mathcal{I} $ of $ \SU_3^{L/K} $ (recall that in this section, $ L $ is ramified of residue characterisitic $ 2 $) is the one obtained by applying Definition~\ref{Def:buildingassociated to a data} to the collection of subgroups $ \lbrace (P_x)_{x\in \R}, N \rbrace $ appearing in Definition~\ref{Def:pointstabiliserforSU ramifiedeven} and Definition~\ref{Def:sbgrp N SU ramifiedeven}, together with the homomorphism $ \nu \colon N\to \Aff (\R) $ of Definition~\ref{Def:affineactionforN ramifiedeven}.  We show in Appendix~\ref{App:A} that our definitions agree with \cite{BrTi1}*{7.4.1~and~7.4.2}, so that the given data is indeed obtained from a valued root datum of rank one on $ G $.

\begin{remark}\label{Rem:dependence on D of the tree rameven}
Note that the construction of the Bruhat--Tits tree of $ \SU_3^{L/K} $ depends on the pair $ (K,L) $. When needed, we keep track of this dependence by adding the subscript $ (K,L) $ to the objects involved. This gives rise to the notations $ \gamma_{(K,L)} $, $ (P_x)_{(K,L)} $, $ T_{(K,L)} $, $ M_{(K,L)} $, $ N_{(K,L)} $, $ \nu_{(K,L)} $ and $ \mathcal{I}_{(K,L)} $.
\end{remark}
\begin{remark}\label{Rem:regularity of the B-T tree SU rameven}
The Bruhat--Tits tree of $ \SU_3^{L/K} $ is actually the $ (\vert \overline{K}\vert+1) $-regular tree. Indeed, this follows from the fact that our definition of $ \mathcal{I} $ agrees with the one given in \cite{BrTi1}*{7.4.1 and 7.4.2}, and from the tables in \cite{Tits77}*{4.2 and 4.3}.
\end{remark}

\subsection{Local model of the Bruhat--Tits tree}
We now proceed to define a local model for the Bruhat--Tits tree of $ SU_3^{L/K} $ when $ L $ is ramified of residue characteristic $ 2 $. In fact, there are two kinds of local models: when the radius is small, the local action degenerates to an $ \SL_2 $-type action, whereas for large radii, the local action is (similar to) an $ \SU_3 $-type action. We introduce a new parameter which controls the meaning of small in this case.

\begin{definition}\label{Def:parameteri_0}
Set $ i_{0} = \min \lbrace r\in \N ~\vert~ \omega (\pi_{L}^{r})\geq \gamma \rbrace $. Equivalently, let $ \alpha $ be as in Lemma~\ref{Lem:goodformforL}. If $ \alpha = 0 $ (respectively if $ \alpha \neq 0 $), $ i_0 $ is such that $ \omega (\pi_{K}^{i_0}) = \omega (2) $ (respectively $ \omega (\pi_{K}^{i_0}) = \omega (\alpha ) $).
\end{definition}
Note that since $ \gamma $ only depends on the normalisation of the valuation on $ K $, $ i_0 $ is an intrinsic parameter of $ L/K $ (not even depending on the normalisation of the valuation).
\begin{lemma}\label{Lem:trace always of big valuation}
Let $ x $ be in $ \mathcal{O}_L $. Then $ x+\bar{x}\in \pi_K^{i_0}\mathcal{O}_K $.
\end{lemma}
\begin{proof}
Let $ t,\alpha ,\beta $ be as in Lemma~\ref{Lem:goodformforL}, so that $ \mathcal{O}_L=\mathcal{O}_K\oplus t.\mathcal{O}_K $. Let $ x_1,x_2 \in \mathcal{O}_K $ such that $ x=x_1+tx_2 $. Then $ x+\bar{x} = 2x_1+\alpha x_2 $, which belongs to $ \pi_K^{i_0}\mathcal{O}_K $ by Definition~\ref{Def:parameteri_0}.
\end{proof}

Hence, we can divide the trace of any element of $ \mathcal{O}_L $ by $ \pi_K^{i_0} $ and still get an element of $ \mathcal{O}_K $. Since this map plays a central role, we make a formal definition.
\begin{definition}\label{Def:reduced trace}
Let $ L/K $ be as in this section, and let $ i_0 $ be the associated parameter as in Definition~\ref{Def:parameteri_0}. Given a uniformiser $ \pi_K $, we define a homomorphism of $ \mathcal{O}_K $-module $ \frac{Tr}{\pi_K^{i_0}}\colon \mathcal{O}_L\to \mathcal{O}_K\colon x\mapsto \frac{x+\bar{x}}{\pi_K^{i_0}} $ that we call the reduced trace map. More generally, for any $ \mathcal{O}_K $-algebra $ R $, the map $ R\otimes_{\mathcal{O}_K}\mathcal{O}_L\to R\otimes_{\mathcal{O}_K}\mathcal{O}_K\colon r\otimes x\mapsto r\otimes \frac{Tr}{\pi_K^{i_0}}(x) $ is also called the reduced trace map and is also denoted by $ \frac{Tr}{\pi_K^{i_0}} $. In particular, taking $ R $ to be $ \mathcal{O}_K/\mathfrak{m}_K^r $, we get a reduced trace map $ \frac{Tr}{\pi_K^{i_0}}\colon \mathcal{O}_L/\mathfrak{m}_L^{2r}\to \mathcal{O}_K/\mathfrak{m}_K^r $.
\end{definition}

The reduced trace map depends on a choice of uniformiser in $ \mathcal{O}_K $. However, its action on $ \mathcal{O}_L/\mathfrak{m}_L^{2r} $ only depends on the uniformiser modulo $ \mathfrak{m}_K^{r+1} $, as the following result shows. Technically, this result is only going to be used much later, but we decided to place it here to illustrate concretely how this reduced trace map acts on quotient.
\begin{lemma}\label{Lem:reduced trace does not depend on lifting uniformiser}
Let $ L/K $ and $ i_0 $ be as in Definition~\ref{Def:reduced trace}. Let $ \pi_K $ and $ \tilde{\pi}_K $ be two uniformisers of $ K $ such that  $ \pi_K - \tilde{\pi}_K \in \mathfrak{m}_K^{r+1} $. Then for all $ s\leq r $ and for all $ x\in \mathcal{O}_L/\mathfrak{m}_L^{2s} $, $ \frac{Tr}{\pi_K^{i_0}}(x) = \frac{Tr}{\tilde{\pi}_K^{i_0}}(x) $.
\end{lemma}
\begin{proof}
Let $ x\in \mathcal{O}_L/\mathfrak{m}_L^{2s} $ and let $ x'\in \mathcal{O}_L $ be a lift of $ x $. Let $ \lambda $ and $ \tilde{\lambda} $ be elements of $ \mathcal{O}_K $ such that $ \pi_K^{i_0}\lambda = x'+\overline{x'} $ and $ \tilde{\pi}_K^{i_0}\tilde{\lambda} = x'+\overline{x'} $. We have to show that $ \lambda -\tilde{\lambda} $ is in $ \mathfrak{m}_K^s $.

By assumption, there exists $ \varepsilon \in \mathfrak{m}_K^{r+1} $ such that $ \tilde{\pi}_K = \pi_K+\varepsilon $. Hence $ \pi_K^{-1}.\tilde{\pi}_K = 1+\pi_K^{-1}\varepsilon $, so that there exists $ \delta \in \mathfrak{m}_K^r $ such that $ (1+\delta )\pi_K^{-1} = \tilde{\pi}_K^{-1} $. We thus get that $ \lambda = \pi_K^{-i_0}(x'+\overline{x'}) $, while $ \tilde{\lambda} = (1+\delta )^{i_0}\pi_K^{-i_0}(x'+\overline{x'}) $, so that $ \lambda -\tilde{\lambda} $ is in $ \mathfrak{m}_K^s $, as wanted.
\end{proof}

We now define local stabilisers. In case the radius is smaller than $ 2i_0 $, local stabilisers somehow degenerate to $ \SL_2 $-type stabilisers, which explains the surprising following definition. Ultimately, this fact is responsible for Chabauty convergence of groups of type $ \SU_3 $ to $ \SL_2 $.

\begin{definition}\label{Def:localdataforbadSU}
Let $ i_0 $ be the parameter associated to $ L/K $, let $ r\in \N $ and let $ x\in [-\omega (\pi_{L}^{2r}),\omega (\pi_{L}^{2r})] $.
\begin{enumerate}[$ \bullet $]
\item For $2r\leq 2i_0$, we define
$ P_{x}^{0,2r} = \lbrace g\in \SL_{2}(\mathcal{O}_{L}/\mathfrak{m}_{L}^{2r})~\vert~ \omega (g)\geq \begin{psmallmatrix}
0 & -x \\
x & 0
\end{psmallmatrix} \rbrace $
\item For $ 2r>2i_0 $, let $ t $ be a uniformiser of $ \mathcal{O}_{L} $ and let $ \beta = t\bar{t} $. We define
\begin{align*}
 P_{x}^{0,2r} = \lbrace  g \in \SL_3(\mathcal{O}_L/\mathfrak{m}_L^{2r})~\vert~^S\bar{g}\begin{psmallmatrix} 1&0 &0\\
0&\beta^{i_0}&0\\
0&0&1 \end{psmallmatrix}g=\begin{psmallmatrix} 1&0 &0\\
0&\beta^{i_0}&0\\
0&0&1 \end{psmallmatrix}, ~\omega (g)\geq \begin{psmallmatrix}
0 & -\frac{x}{2} &-x \\
\frac{x}{2} &0&-\frac{x}{2} \\
x & \frac{x}{2} &0
\end{psmallmatrix}\\
\frac{Tr}{\beta^{i_0}}(\bar{g}_{31}g_{11}) = -N(g_{21}), \frac{Tr}{\beta^{i_0}}(\bar{g}_{33}g_{13}) = -N(g_{23}) \rbrace.
\end{align*}
When needed, we emphasize the dependence on $ t $ by denoting this group $ P_x^{0,2r}(t) $.
\end{enumerate}
\end{definition}
\begin{remark}
It is not a priori clear that $  P_{x}^{0,2r} $ is a group when $ 2r>2i_0 $. This is a consequence of our work on integral model, see Corollary~\ref{Cor:local model are subgroups}.
\end{remark}
\begin{remark}
We only give local models for an even radius. It seems a priori possible to give local models for odd radius as well, but it would make this section even more technical than it already is, so that we decided to not dwell on those complications.
\end{remark}
\begin{remark}
If one defines $ \tilde{P}_x^{0,2r} $ by taking the second (more complicated) definition for all $ r\in \N $, then one can see (by carrying the kind of computations we perform later) that for $ 2r\leq 2i_0 $, $ P_x^{0,2r} $ is a kind of reductive quotient of $ \tilde{P}_x^{0,2r} $.
\end{remark}

We now show that $ P_x^{0,2r}(t) $ does not depend on $ t $.
\begin{lemma}\label{Lem:local stabiliser does not depend on choice of uniformiser}
Let $ 2r>2i_0 $, and let $ t,t' $ be uniformisers of $ \mathcal{O}_L $. Then for all $ x\in [-\omega (\pi_{L}^{2r}),\omega (\pi_{L}^{2r})] $, $ P_x^{0,2r}(t)\cong P_x^{0,2r}(t') $.
\end{lemma}
\begin{proof}
Those groups are actually conjugate. Indeed, the map 
\begin{equation*}
P_x^{0,2r}(t)\to P_x^{0,2r}(t')\colon g\mapsto \begin{psmallmatrix} 1&0 &0\\
0&t^{-i_0}(t')^{i_0}&0\\
0&0&1 \end{psmallmatrix}g\begin{psmallmatrix} 1&0 &0\\
0&t^{i_0}(t')^{-i_0}&0\\
0&0&1 \end{psmallmatrix}
\end{equation*}
is readily seen to be an isomorphism.
\end{proof}
\begin{remark}
In Definition~\ref{Def:localdataforbadSU} and in the case $ 2r>2i_0 $, one could bypass the choice of a uniformiser $ t\in \mathcal{O}_L $ by directly choosing a uniformiser $ \beta \in \mathcal{O}_K $ and proceed to make the same definition. However, it would then not be clear that the definition does not depend on the choice of $ \beta $, because there are uniformisers of $ \mathcal{O}_K $ that are not the norm of uniformisers of $ \mathcal{O}_L $ (see \cite{S79}*{Chapter~V, §3, Remark}).
\end{remark}

\begin{definition}
Let $ r\in \N $. For $ 2r\leq 2i_0 $, we define
\begin{enumerate}
\item $ H^{0,2r} = \lbrace \begin{psmallmatrix}
x & 0 \\
0 & x^{-1}
\end{psmallmatrix}\in \SL_{2}(\mathcal{O}_{L}/\mathfrak{m}_{L}^{2r})~\vert~ \omega (x) = 0 \rbrace $
\item $ M^{0,2r} = \lbrace \begin{psmallmatrix}
0 & -x \\
x^{-1} & 0
\end{psmallmatrix}\in \SL_{2}(\mathcal{O}_{L}/\mathfrak{m}_{L}^{2r})~\vert~ \omega (x) = 0 \rbrace $,
\end{enumerate} while for $ 2r>2i_0 $, we define
\begin{enumerate}
\item $ H^{0,2r} = \lbrace \begin{psmallmatrix}
x & 0 & 0 \\
0 & x^{-1}\bar{x} & 0 \\
0 &0 & \bar{x}^{-1}
\end{psmallmatrix}\in \SL_{3}(\mathcal{O}_{L}/\mathfrak{m}_{L}^{2r})~\vert~ \omega (x) = 0 \rbrace $
\item $ M^{0,2r} = \lbrace \begin{psmallmatrix}
0 & 0 & x \\
0 & -x^{-1}\bar{x} & 0 \\
\bar{x}^{-1} &0 & 0
\end{psmallmatrix}\in \SL_{3}(\mathcal{O}_{L}/\mathfrak{m}_{L}^{2r})~\vert~ \omega (x) = 0 \rbrace $
\end{enumerate}
And then for all $ r\in \N $, we set $ N^{0,2r} = H^{0,2r}\sqcup M^{0,2r} $.
\end{definition}
\begin{definition}
For all $ r\in \N $, we let $ H^{0,2r} $ act trivially on $ \R $, and we let all elements of $ M^{0,2r} $ act as a reflection through $ 0\in \R $. This gives an affine action of $ N^{0,2r} $ on $ \R $, and we denote again the resulting map $ N^{0,2r}\to \Aff (\R) $ by $ \nu $.
\end{definition}

We are now able to give a local definition of the ball of radius $ 2r $ around $ [(\Id ,0)]\in \mathcal{I} $. Since the reduced trace is involved in the definition when $ 2r>2i_0 $, this local definition depends on a bit more than $ \mathcal{O}_L/\mathfrak{m}_L^{2r} $. Using Lemma~\ref{Lem:controlling the reduced norm}, we prove in Proposition~\ref{Prop:aritmimpliesgeomproximityforSU3 rameven} that it only depends on $ \mathcal{O}_L/\mathfrak{m}_L^{2r+2i_0} $.
\begin{definition}\label{Def:local tree of radius r rameven}
Let $ r\in \N $. We define a $ 2r $-local equivalence on $ P_{0}^{0,2r}\times [-\omega (\pi_L^{2r}), \omega (\pi_L^{2r})] $ as follows. For $ g,h\in P_{0}^{0,2r} $ and $ x,y\in [-\omega (\pi_L^{2r}), \omega (\pi_L^{2r})] $
\begin{equation*}
(g,x)\sim_{0,2r}(h,y) \Leftrightarrow \textrm{ there exists } n\in N^{0,2r} \textrm{ such that } \nu (n)(x)=y \textrm{ and } g^{-1}hn \in P_{x}^{0,2r}
\end{equation*}
The resulting space $ \mathcal{I}^{0,2r} = P_{0}^{0,2r}\times [-\omega (\pi_L^{2r}), \omega (\pi_L^{2r})]/\sim_{0,2r} $ is called the local Bruhat--Tits tree of radius $ 2r $ around $ 0 $, and $ [(g,x)]^{0,2r} $ stands for the equivalence class of $ (g,x) $ in $ \mathcal{I}^{0,2r} $. The group $ P_{0}^{0,2r} $ acts on $ \mathcal{I}^{0,2r} $ by multiplication on the first component.
\end{definition}

\begin{remark}\label{Rem:dependence on D of the local tree rameven}
Note that the construction of the local Bruhat--Tits tree of $ \SU_3^{L/K} $ depends on the pair $ (K,L) $ (which is assumed to be ramified and of residue characteristic $ 2 $ in this section). When needed, we keep track of this dependence by adding the subscript $ (K,L) $ to the objects involved. This gives rise to the notations $ (i_0)_{(K,L)} $, $ (P_x^{0,2r})_{(K,L)} $, $ H^{0,2r}_{(K,L)} $, $ M^{0,2r}_{(K,L)} $, $ N^{0,2r}_{(K,L)} $ and $ \mathcal{I}^{0,2r}_{(K,L)} $.
\end{remark}

\begin{remark}\label{Rem:equivalent def with for all n for local model ram even}
Also, Remark~\ref{Rem:equivalent def with for all n for local model} holds equally well in this case, with exactly the same proof, upon replacing all $ d $'s by $ 2 $.
\end{remark}

\subsection{Integral model}
We now proceed to define an integral model, and compare its rational points with our local model for the Bruhat--Tits tree of $ SU_3^{L/K} $ when $ L $ is ramified of residue characteristic $ 2 $. The construction is more involved, because the naive integral model is not smooth in this case. In the following definition, we denote by $ \mathcal{R}^{\mathcal{O}_L}_{\mathcal{O_K}}\SL_3 $ the Weil restriction from $ \mathcal{O}_L $ to $ \mathcal{O}_K $ of $ \SL_3 $.
\begin{definition}\label{Def:integral model bad SU}
Let $ L/K $ be as in this section, let $ t $ be a uniformiser of $ L $ and let $ i_0 $ be as in Definition~\ref{Def:parameteri_0}. Consider the homomorphism of $K$-group schemes 
\begin{align*}
\varphi_{i_0}\colon \SU_3^{L/K}&\to [\mathcal{R}^{\mathcal{O}_L}_{\mathcal{O_K}}\SL_3]_K \\
g&\mapsto \begin{psmallmatrix} g_{11}&t^{i_0}g_{12} &g_{13}\\
t^{-i_0}g_{21}&g_{22}&t^{-i_0}g_{23}\\
g_{31}&t^{i_0}g_{32}&g_{33} \end{psmallmatrix}
\end{align*}
We define $ \underline{\SU}_3^{L/K} $ to be the schematic adherence of $ \varphi_{i_0} (\SU_3^{L/K}) $ in $ \mathcal{R}^{\mathcal{O}_L}_{\mathcal{O_K}}\SL_3 $.
\end{definition}
\begin{remark}
Since two uniformisers of $ L $ differ by an invertible element of $ \mathcal{O}_L $, it follows from the definition of the schematic adherence that $ \underline{\SU}_3^{L/K} $ does not depend on the choice of a uniformiser $ t\in \mathcal{O}_L $. One can also check that directly using our explicit description of $ \underline{\SU}_3^{L/K} $ in Theorem~\ref{Thm:smoothness of integral model rameven}. 
\end{remark}
\begin{remark}\label{Rem:underlineGisthesameinTits}
The concrete description given here was derived from the general construction made in \cite{BrTi3}, see especially section $ 3.9 $ and the Theorem in section $ 5 $ of loc. cit. In Theorem~\ref{Thm:smoothness of integral model rameven}, we give the defining equations of $ \underline{\SU}_3 $ inside $ \mathcal{R}^{\mathcal{O}_L}_{\mathcal{O_K}}\SL_3 $. This allows us to prove in a direct way that it is a smooth $ \mathcal{O}_{K} $-group scheme, without referring to loc. cit.
\end{remark}
\begin{lemma}\label{Lem:intergal model is a group}
$ \underline{\SU}_3^{L/K} $ is an $ \mathcal{O}_K $-group scheme.
\end{lemma}
\begin{proof}
Since $ \SL_3 $ is smooth over $ \mathcal{O}_L $, $ \mathcal{R}^{\mathcal{O}_L}_{\mathcal{O_K}}\SL_3 $ is smooth and hence flat over $ \mathcal{O}_K $. The result then follows from \cite{BrTi2}*{1.2.7}.
\end{proof}

We are actually able to determine explicitly the equations defining $ \underline{\SU}_3^{L/K} $. Recall the reduced trace $ \frac{Tr}{\pi_K^{i_0}}\colon \mathcal{O}_L\to \mathcal{O}_K $ introduced in Definition~\ref{Def:reduced trace}.
\begin{theorem}\label{Thm:smoothness of integral model rameven}
Keep the notations of Definition~\ref{Def:integral model bad SU} and let $ \beta = t\bar{t} $. For any $ \mathcal{O}_K $-algebra $ R $, 
\begin{align*}
\underline{\SU}_3^{L/K}(R) = \lbrace g\in &\mathcal{R}^{\mathcal{O}_L}_{\mathcal{O_K}}\SL_3(R)~\vert~^S\bar{g}\begin{psmallmatrix} 1&0 &0\\
0&\beta^{i_0}&0\\
0&0&1 \end{psmallmatrix}g=\begin{psmallmatrix} 1&0 &0\\
0&\beta^{i_0}&0\\
0&0&1 \end{psmallmatrix},\\
&\frac{Tr}{\beta^{i_0}}(\bar{g}_{31}g_{11}) = -N(g_{21}), \frac{Tr}{\beta^{i_0}}(\bar{g}_{33}g_{13}) = -N(g_{23})\rbrace
\end{align*} Furthermore, $ \underline{\SU}_3^{L/K} $ is smooth over $ \mathcal{O}_K $.
\end{theorem}
\begin{proof}\setcounter{claim}{0}
Recall that in loose terms, for $ Z $ an $ \mathcal{O}_K $-scheme and $ Y_K $ a closed subscheme of $ Z_K $, the schematic adherence of $ Y_K $ in $ Z $ is defined by all equations satisfied by $ Y_K $, where we ``put the valuation of the coefficients as low as possible'' (while still remaining in $ \mathcal{O}_K $). More precisely, let $ I(Y) $ be the ideal of $ K[Z] $ defining $ Y_K $ in $ Z_K $, and let $ j\colon \mathcal{O}_K[Z]\to K[Z] $ be the canonical map (which is injective if and only if $ Z $ is flat). Then the ideal defining the schematic adherence of $Y_K$ in $ Z $ is $ j^{-1}(I(Y)) $ (see \cite{BrTi2}*{1.2.6}).

Hence, an element $ g $ in $ \underline{\SU}_3^{L/K}(R) $ clearly satisfies $ ^S\bar{g}\begin{psmallmatrix} 1&0 &0\\
0&\beta^{i_0}&0\\
0&0&1 \end{psmallmatrix}g=\begin{psmallmatrix} 1&0 &0\\
0&\beta^{i_0}&0\\
0&0&1 \end{psmallmatrix} $. Moreover, since $ g $ satisfies this equation, we get in particular that $ \bar{g}_{33}g_{13}+\bar{g}_{23}\beta^{i_0}g_{23}+\bar{g}_{13}g_{33}=0 $, so that after division by $ \beta^{i_0} $, we get $ \frac{Tr}{\beta^{i_0}}(\bar{g}_{33}g_{13}) = -N(g_{23}) $. The remaining equation is obtained similarly.

Let $ X $ be the $\mathcal{O}_K$-closed subscheme of $ \mathcal{R}^{\mathcal{O}_L}_{\mathcal{O_K}}\SL_3 $ defined on the right hand side of the equality in the theorem. We have just showed that $ X $ contains $ \underline{\SU}_3^{L/K} $ as a closed subscheme, and we want to prove that they are actually equal. For this, let us study the fibres of $ X $ over $ \Spec \mathcal{O}_K $. Obviously, $ X_K $ is isomorphic to $ \SU_3^{L/K} $, and hence is smooth of dimension $ 8 $.

\begin{claim*}\label{Claim:smoothness}
$ X_{\overline{K}} $ is a smooth irreducible scheme of dimension $ 8 $.
\end{claim*}
\begin{claimproof}
To prove those claims, we can assume that the residue field $ \overline{K} $ is algebraically closed. Note that modulo $ \pi_K $, the Galois conjugation is trivial and $ t^2 = 0 $. Also note that $ \overline{K} $ is of characteristic $ 2 $, so that we do not need to use the minus sign for the moment. A direct application of the definitions shows that
\begin{align*}
 X_{\overline{K}}(\overline{K}) = \lbrace g\in \SL_3(\overline{K}[t])~&\vert~^Sg\begin{psmallmatrix} 1&0 &0\\
0&0&0\\
0&0&1 \end{psmallmatrix}g=\begin{psmallmatrix} 1&0 &0\\
0&0&0\\
0&0&1 \end{psmallmatrix},~\frac{Tr}{\beta^{i_0}}(g_{31}g_{11}) = (g_{21})^{2}, \frac{Tr}{\beta^{i_0}}(g_{33}g_{13}) = (g_{23})^{2} \rbrace, 
\end{align*}
Let $ g\in X_{\overline{K}}(\overline{K}) $. In particular, $ g $ satisfies the following equations:
\begin{align*}
g_{33}g_{11}+g_{13}g_{31}=1\\
g_{33}g_{12}+g_{13}g_{32}=0\\
g_{32}g_{11}+g_{12}g_{31}=0
\end{align*}

From this, we deduce that $ g_{12}=0 $ as follows:
\begin{align*}
g_{12} &= (g_{33}g_{11}+g_{13}g_{31})g_{12}=g_{11}g_{33}g_{12}+g_{13}g_{12}g_{31}\\
&=g_{11}g_{13}g_{32}+g_{13}g_{32}g_{11}=2(g_{11}g_{13}g_{32}) =0
\end{align*}

And using the same trick, we get $ g_{32}=0 $ as well. Now since $ g\in \SL_3(\overline{K}[t]) $, its determinant is $ 1 $. But since $ g_{12}=g_{32}=0 $, the equation $ \det g=1 $ degenerates to $ g_{11}g_{22}g_{33}+g_{13}g_{22}g_{31}=1 $. And using again that $ g_{33}g_{11}+g_{13}g_{31}=1 $, this implies that $ g_{22}=1 $.

We can summarise our results so far as follows:
\begin{align*}
 X_{\overline{K}}(\overline{K}) \cong \lbrace g\in \SL_3(\overline{K}[t])~&\vert~g_{22}=1,~g_{12} = g_{32} = 0,\\
&\frac{Tr}{\beta^{i_0}}(g_{31}g_{11}) = (g_{21})^{2},~\frac{Tr}{\beta^{i_0}}(g_{33}g_{13}) = (g_{23})^{2} \rbrace.
\end{align*}
In fact, the right hand side is a subgroup of $ \SL_3(\overline{K}[t]) $. To prove it, we write down two representative examples of the computations involved. First, for $ g,h \in \SL_3(\overline{K}[t]) $ and satisfying the conditions, we check that $ \frac{Tr}{\beta^{i_0}}((gh)_{33}(gh)_{13}) = ((gh)_{23})^{2} $. We have
\begin{align*}
(gh)_{33}(gh)_{13} &= (g_{31}h_{13}+g_{33}h_{33})(g_{11}h_{13}+g_{13}h_{33})\\
&=h_{13}^2g_{31}g_{11}+h_{33}h_{13}+h_{33}^2g_{33}g_{13}.
\end{align*} 

Hence, recalling that we are working in characteristic $ 2 $,
\begin{align*}
\frac{Tr}{\beta^{i_0}}((gh)_{33}(gh)_{13}) &= (h_{13})^2\frac{Tr}{\beta^{i_0}}(g_{31}g_{11})+\frac{Tr}{\beta^{i_0}}(h_{33}h_{13})+(h_{33})^2\frac{Tr}{\beta^{i_0}}(g_{33}g_{13})\\
&=(h_{13})^2(g_{21})^2+(h_{23})^2+(h_{33})^2(g_{23})^2\\
&=(g_{21}h_{13}+g_{22}h_{23}+g_{23}h_{33})^2=((gh)_{23})^2
\end{align*}

As a second example, for $ g\in \SL_3(\overline{K}[t]) $ such that $ g $ satisfies the given conditions, we check that $ \frac{Tr}{\beta^{i_0}}((g^{-1})_{33}(g^{-1})_{13}) = ((g^{-1})_{23})^{2} $. We have $ (g^{-1})_{31}(g^{-1})_{11} = g_{11}g_{13} $, while $ ((g^{-1})_{21})^{2} = (g_{11}g_{23}+g_{21}g_{13})^2 $. Hence we want to show that $ \frac{Tr}{\beta^{i_0}}(g_{11}g_{13}) = (g_{11}g_{23}+g_{21}g_{13})^2 $. This goes as follows:
\begin{align*}
\frac{Tr}{\beta^{i_0}}(g_{11}g_{13}) &= \frac{Tr}{\beta^{i_0}}(g_{11}g_{13}(g_{11}g_{33}+g_{13}g_{31}))\\
&=\frac{Tr}{\beta^{i_0}}(g_{11}^2g_{13}g_{33}+g_{11}g_{13}^2g_{31})\\
&=g_{11}^2\frac{Tr}{\beta^{i_0}}(g_{13}g_{33})+g_{13}^2\frac{Tr}{\beta^{i_0}}(g_{11}g_{31})\\
&=g_{11}^2g_{23}^2+g_{13}^2g_{21}^2 = (g_{11}g_{23}+g_{21}g_{13})^2,
\end{align*}
as wanted (recall that we are in characteristic $ 2 $). Hence, $ X_{\overline{K}}(\overline{K}) $ is indeed a group scheme over $ \overline{K} $.

Now, let $ \pi \colon \overline{K}[t]\to \overline{K} $ be the quotient modulo $ t $, and consider the following composition of homomorphism of $ \overline{K} $-group schemes:
\begin{align*}
 X_{\overline{K}}(\overline{K})\xrightarrow{f}~\SL_2(\overline{K}[t])&\xrightarrow{g} \SL_2(\overline{K})\\
 \begin{psmallmatrix} g_{11}& 0 &g_{13}\\
g_{21}&1&g_{23}\\
g_{31}& 0 &g_{33} \end{psmallmatrix}\mapsto \begin{psmallmatrix} g_{11}&g_{13}\\
g_{31}&g_{33} \end{psmallmatrix}&\mapsto \begin{psmallmatrix} \pi (g_{11})&\pi (g_{13})\\
\pi (g_{31})&\pi (g_{33}) \end{psmallmatrix}
\end{align*}
Note that $ \ker f = \lbrace  \begin{psmallmatrix} 1& 0 &0\\
g_{21}&1&g_{23}\\
0&0&1 \end{psmallmatrix}~\vert~g_{ij}\in \overline{K}[t],~(g_{21})^2=0=(g_{23})^2\rbrace $. Hence $ \ker f $ is not reduced, but the corresponding reduced $ \overline{K} $-scheme is irreducible of dimension $ 2 $ (recall that $ t^2 = 0 $). On the other hand, $ \ker g = \lbrace  \begin{psmallmatrix} 1+tg_{11} &tg_{12}\\
tg_{21}&1+tg_{22}\end{psmallmatrix}~\vert~g_{ij}\in \overline{K},~g_{11}+g_{22} = 0\rbrace $, which is irreducible of dimension $ 3 $. Finally, note that $ \SL_2(\overline{K}) $ is connected of dimension $ 3 $, and that $ f $ and $ g $ are surjective. Hence, we deduce that $  X_{\overline{K}}(\overline{K}) $ is connected. Furthermore, using twice \cite{DG70}*{II, §5, Proposition~5.1} (note that it does not use smoothness), $  X_{\overline{K}}(\overline{K}) $ is of dimension $ 8 $. To conclude, the tangent space at the identity of $ X_{\overline{K}}(\overline{K}) $ is $ \lbrace g\in M_3(\overline{K}[t])~\vert~g_{22}=0, g_{11}+g_{33}=0, g_{12}=g_{32}=0, \frac{Tr}{\beta^{i_0}}(g_{31})=0, \frac{Tr}{\beta^{i_0}}(g_{13})=0 \rbrace $, which is of dimension $ 8 $, as wanted.
\end{claimproof}

Since $ X_{\overline{K}} $ is irreducible and smooth of the same dimension than $ X_K $, we can prove that $ X $ is flat over $ \mathcal{O}_K $ simply by giving a closed embedding of $ \mathcal{O}_K $-schemes $ \mathbf{A}_{\mathcal{O}_K}^1\hookrightarrow X $ (see Lemma~\ref{Lem:Proving flatness of schemes over local rings}). But finding such an embedding is easy enough: the elements in $ \mathcal{O}_L $ having trace $ 0 $ form a free $ \mathcal{O}_K $-module of rank $ 1 $. Let $ \tau $ denote a generator of this $ \mathcal{O}_K $-module. Now, the map $ \mathbf{A}_{\mathcal{O}_K}^1\to X\colon x\mapsto \begin{psmallmatrix} 1&0 & x.\tau \\
0&1&0\\
0&0&1 \end{psmallmatrix} $ is the desired closed embedding.

We can finally conclude that $ \underline{SU}_3^{L/K} = X $. Indeed, both schemes are flat over $ \mathcal{O}_K $, and their generic fibres are equal. But a flat subscheme of $ \mathcal{R}^{\mathcal{O}_L}_{\mathcal{O_K}}\SL_3 $ is uniquely determined by its generic fibre by \cite{BrTi2}*{1.2.6} (which also uses the fact that $ \mathcal{R}^{\mathcal{O}_L}_{\mathcal{O_K}}\SL_3 $ is (flat and hence) torsion free). The last statement of the theorem follows from the fact that $ X $ is flat with smooth fibres.
\end{proof}

We now compare the rational points of the integral model with our local model.
\begin{lemma}\label{Lem:integral point is P_0}
Let $ \varphi_{i_0}\colon \SU_3^{L/K}\mapsto (\underline{\SU}_3^{L/K})_K $ be the homomorphism introduced in Definition~\ref{Def:integral model bad SU}, and let $ \varphi_{i_0}(K) $ be the induced isomorphism on $ K $-rational points. Then $ P_0 = \varphi_{i_0}^{-1}(K)(\underline{\SU}_3^{L/K}(\mathcal{O}_K)) $.
\end{lemma}
\begin{proof}
By definition, $ \varphi_{i_0}^{-1}(K)\colon (\underline{\SU}_3^{L/K})(K)\to \SU_3^{L/K}(K)\colon g\mapsto \begin{psmallmatrix} g_{11}&t^{-i_0}g_{12} &g_{13}\\
t^{i_0}g_{21}&g_{22}&t^{i_0}g_{23}\\
g_{31}&t^{-i_0}g_{32}&g_{33} \end{psmallmatrix} $ is an isomorphism. Since by Definition~\ref{Def:parameteri_0}, $ i_0 $ is the smallest integer such that $ \omega (t^{i_0})\geq \gamma $, the lemma follows from the fact that $ P_0 = \lbrace g\in \SU_3^{L/K}(K)~\vert~ \omega (g)\geq \begin{psmallmatrix} 0&-\gamma &0\\
\gamma &0&\gamma \\
0&-\gamma &0 \end{psmallmatrix}\rbrace $ (see Definition~\ref{Def:pointstabiliserforSU ramifiedeven}).
\end{proof}

\begin{lemma}\label{Lem:surjection on local model}
Let $ r\in \N $.
\begin{enumerate}[$ \bullet $]
\item For $ 2r\leq 2i_0 $, we define a map $ f_{2r}\colon \underline{\SU}_3(\mathcal{O}_K/\mathfrak{m}_K^{r})\to \SL_2(\mathcal{O}_L/\mathfrak{m}_L^{2r})\colon g\mapsto \begin{psmallmatrix} g_{11}&g_{13}\\
g_{31}&g_{33} \end{psmallmatrix} $. The map $ f_{2r} $ is a group homomorphism whose image is $ P_0^{0,2r} $.
\item For $ 2r>2i_0 $, we have an isomorphism $ f_{2r}\colon \underline{\SU}_3(\mathcal{O}_K/\mathfrak{m}_K^{r})\cong P_0^{0,2r} $.
\end{enumerate}
\end{lemma}
\begin{proof}
We begin with the case $ 2r\leq 2i_0 $. In this case, note that $ t^{2i_0} = 0 $ in $ \mathcal{O}_L/\mathfrak{m}_L^{2r} $, and that the conjugation action is trivial in $ \mathcal{O}_L/\mathfrak{m}_L^{2r} $. Furthermore, $ \beta^{i_0} = 0 $ and $ 2=0 $ in $ \mathcal{O}_L/\mathfrak{m}_L^{2r} $. Hence we can reproduce the computations made in the beginning of the proof of the claim in Theorem~\ref{Thm:smoothness of integral model rameven}. We thus get 
\begin{align*}
\underline{\SU}_3(\mathcal{O}_K/\mathfrak{m}_K^{r}) = \lbrace g\in \SL_3(\mathcal{O}_L/\mathfrak{m}_L^{2r})~&\vert~g_{22}=1,~g_{12} = g_{32} = 0\\
&\frac{Tr}{\beta^{i_0}}(g_{31}g_{11}) = (g_{21})^{2},~\frac{Tr}{\beta^{i_0}}(g_{33}g_{13}) = (g_{23})^{2} \rbrace.
\end{align*}
This already shows that $ f_{2r} $ is a group homomorphism. Now, to prove that $ f_{2r} $ is surjective, there just remains to prove that the map $ \mathcal{O}_L/\mathfrak{m}_L^{2r}\to \mathcal{O}_K/\mathfrak{m}_K^r\colon x\mapsto x^2 $ is surjective. Note that modulo $ \mathfrak{m}_L^{2r} $, $ t^2 = t\bar{t} = \beta $. Hence for $ x = x_1+tx_2 \in \mathcal{O}_L/\mathfrak{m}_L^{2r} $ with $ x_i\in \mathcal{O}_K/\mathfrak{m}_K^r $, we have $ x^2 = x_1^2+\beta x_2^2 $. Hence the result follows from the fact that $ \mathcal{O}_K/\mathfrak{m}_K^r = \sum_{i=0}^{r-1}\mathbf{F}_{2^n}\beta^i $, and that squaring is a bijection on the field $ \mathbf{F}_{2^n} $.

When $ 2r>2i_0 $, the assertion follows directly from Definition~\ref{Def:localdataforbadSU} and Theorem~\ref{Thm:smoothness of integral model rameven}.
\end{proof}
\begin{corollary}\label{Cor:local model are subgroups}
Let $ r\in \N $ and let $ x\in [-\omega (\pi_{L}^{2r}),\omega (\pi_{L}^{2r})] $. Then $ P_x^{0,2r} $ is a group.
\end{corollary}
\begin{proof}
Only the case $ 2r>2i_0 $ requires a proof. By Lemma~\ref{Lem:intergal model is a group} and Lemma~\ref{Lem:surjection on local model}, $ P_{0}^{0,2r} $ is a group. Note that $ A_x = \lbrace g\in \SL_3(\mathcal{O}_L/\mathfrak{m}_L^{2r})~\vert~\omega(g)\geq \begin{psmallmatrix}
0 & -\frac{x}{2} &-x \\
\frac{x}{2}&0&-\frac{x}{2}\\
x & \frac{x}{2} &0
\end{psmallmatrix} \rbrace $ is a subgroup of $ \SL_3(\mathcal{O}_L/\mathfrak{m}_L^{2r}) $. But $ P_x^{0,2r} = P_0^{0,2r}\cap A_x $. Hence the result follows.
\end{proof}
\begin{definition}\label{Def:description of the restriction P_0 to P_0^0,r rameven}
Let $ p_{2r}\colon P_{0}\to P_{0}^{0,2r} $ be the homomorphism such that the following square commutes
\begin{center}
\begin{tikzpicture}[->]
 \node (1) at (0,0) {$ \underline{\SU}_3(\mathcal{O}_{K}) $};
 \node (2) at (2.7,0) {$ P_{0}\leq \SL_{3}(\mathcal{O}_{L}) $};
 \node (5) at (1.2,0.2) {$ \xrightarrow{\varphi_{i_0}^{-1}} $};
 \node (3) at (-0.3,-1.2) {$ \underline{\SU}_3(\mathcal{O}_{K}/\mathfrak{m}_{K}^{r}) $};
 \node (4) at (2, -1.2) {$ P_{0}^{0,2r} $};
 \node (6) at (1.2,-1.1) {$ \xrightarrow{f_{2r}} $};
 \node (7) at (2.1,-0.6) {$ p_{2r} $};
 
 \draw[->] (0,-0.3) to (0,-0.9);
 \draw[->] (1.8,-0.3) to (1.8,-0.9);
\end{tikzpicture}
\end{center}
Let $ \pi_{2r}\colon \mathcal{O}_{L}\to \mathcal{O}_{L}/\mathfrak{m}_{L}^{2r} $ denote the reduction modulo $ \mathfrak{m}_{L}^{2r} $.
\begin{enumerate}[$ \bullet $]
\item If $ 2r\leq 2i_0 $, $ p_{2r}(\begin{psmallmatrix} g_{11}&t^{-i_0}g_{12} &g_{13}\\
t^{i_0}g_{21}&g_{22}&t^{i_0}g_{23}\\
g_{31}&t^{-i_0}g_{32}&g_{33} \end{psmallmatrix}) = \begin{psmallmatrix} \pi_{2r}(g_{11})&\pi_{2r}(g_{13})\\
\pi_{2r}(g_{31})&\pi_{2r}(g_{33}) \end{psmallmatrix}$
\item If $ 2r> 2i_0 $, $ p_{2r}(\begin{psmallmatrix} g_{11}&t^{-i_0}g_{12} &g_{13}\\
t^{i_0}g_{21}&g_{22}&t^{i_0}g_{23}\\
g_{31}&t^{-i_0}g_{32}&g_{33} \end{psmallmatrix}) = \begin{psmallmatrix} \pi_{2r}(g_{11})&\pi_{2r}(g_{12}) &\pi_{2r}(g_{13})\\
\pi_{2r}(g_{21})&\pi_{2r}(g_{22})&\pi_{2r}(g_{23})\\
\pi_{2r}(g_{31})&\pi_{2r}(g_{32})&\pi_{2r}(g_{33}) \end{psmallmatrix}$.
\end{enumerate}
\end{definition}

And we can then deduce the surjectivity of the map $ p_{2r} $.
\begin{corollary}\label{Cor:surjectivityofpr rameven}
The map $ p_{2r}\colon P_{0}\to P_{0}^{0,2r} $ is surjective, for all $ r\in \N $.
\end{corollary}
\begin{proof}
This is a direct consequence of the commutative square involving $ P_{0}\to P_{0}^{0,2r} $ given in Definition~\ref{Def:description of the restriction P_0 to P_0^0,r rameven}. Indeed, the integral model is smooth by Theorem~\ref{Thm:smoothness of integral model rameven}, so that Theorem~\ref{Thm:genhensel} applied to the left hand side of the diagram shows surjectivity there, and we proved in Lemma~\ref{Lem:surjection on local model} that $ f_{2r} $ is surjective as well.
\end{proof}

We also need a kind of injectivity result:
\begin{lemma}\label{Lem:injectivityofpr rameven}
Let $ r\in \N $ and $ x\in [-\omega (\pi_{L}^{2r}),\omega (\pi_{L}^{2r})] $. Then $ p_{2r}^{-1}(P_{x}^{0,2r}) \subset P_{x} $.
\end{lemma}
\begin{proof}
For $ 2r>2i_0 $, belonging to $ p_{2r}^{-1}(P_x^{0,2r}) $ implies that the valuation of the off diagonal entries are big enough. Hence, the result follows directly from Definition~\ref{Def:pointstabiliserforSU ramifiedeven}.

For $ 2r\leq 2i_0 $, we have to show that if $ g\in P_0 $ is such that $ \omega (g)\geq \begin{psmallmatrix} 0&-\gamma &-x\\
\gamma &0&\gamma \\
x&-\gamma &0 \end{psmallmatrix} $, then $ \omega (g)\geq \begin{psmallmatrix} 0& -\frac{x}{2}-\gamma &-x\\
 \frac{x}{2}+\gamma &0& -\frac{x}{2}+\gamma \\
x& \frac{x}{2}-\gamma &0 \end{psmallmatrix} $. Let us for example prove it for $ x\in [0,\omega (\pi_L^{2r})] $ (the other case being similar). Let $ t, \beta $ be as in Lemma~\ref{Lem:goodformforL}. Then there exists $ g_{ij}\in \mathcal{O}_L $ such that $ g = \begin{psmallmatrix} g_{11}&t^{-i_0}g_{12} &g_{13}\\
t^{i_0}g_{21}&g_{22}&t^{i_0}g_{23}\\
g_{31}&t^{-i_0}g_{32}&g_{33} \end{psmallmatrix} $, with $ \omega (g_{31})\geq x $. Using $ ^S\bar{g}g=\Id $ (respectively $ g^S\bar{g} = \Id $), we get in particular $ \bar{g}_{31}g_{11}+\bar{t}^{i_0}\bar{g}_{21}t^{i_0}g_{21}+\bar{g}_{11}g_{31} = 0 $ (respectively $ g_{31}\bar{g}_{33}+t^{i_0}g_{32}\bar{t}^{i_0}\bar{g}_{32}+g_{33}\bar{g}_{31} = 0 $). Hence, $ \frac{Tr}{\beta^{i_0}}(\bar{g}_{11}g_{31}) = -N(g_{21}) $ (respectively $ \frac{Tr}{\beta^{i_0}}(g_{33}\bar{g}_{31}) = -N(g_{32}) $), which implies that $ \omega (g_{21})\geq \frac{x}{2} $ (respectively $ \omega (g_{32})\geq \frac{x}{2} $), as wanted.
\end{proof}

We finally arrive at the result corresponding to Theorem~\ref{Thm:localdescriptionoftheball}: the ball of radius $ 2r $ together with the action of $ \SU_{3}^{L/K}(K) $ is encoded in $ P_{0}^{0,2r} $. We first need an adequate description of the ball of radius $ 2r $ around $0$ in $ \mathcal{I} $.

\begin{lemma}\label{Lem:B_0(r) is really the ball of radius r rameven}
Renormalise the distance on $ \R $ so that $ d_{\R}(0;\omega(\pi_L)) = 1 $, and put the metric $ d_{\mathcal{I}} $ on $ \mathcal{I} $ arising from the distance $ d_{\R} $ $ ( $see Remark~\ref{Rem:metric on I and stabiliser}$ ) $. Let $ B_0(2r) = \lbrace p\in \mathcal{I}~\vert~d_{\mathcal{I}}([(\Id ,0)];p)\leq 2r\rbrace $ be the ball of radius $ 2r $ around $ 0 $ in $ \mathcal{I} $. Let $ \tilde{B}_0(2r) = \lbrace [(g,x)]\in \mathcal{I} ~\vert~ g\in P_{0}, x\in [-\omega (\pi_L^{2r}), \omega (\pi_L^{2r})]\subset \R \rbrace $. Then $ B_0(2r) = \tilde{B}_0(2r) $.
\end{lemma}
\begin{proof}
The proof is word for word the same than the proof of Lemma~\ref{Lem:B_0(r) is really the ball of radius r}, upon replacing all $d$'s by $2$'s.
\end{proof}
\begin{remark}
The distance $ d_{\mathcal{I}} $ that we introduced in Lemma~\ref{Lem:B_0(r) is really the ball of radius r rameven} is also the combinatorial distance on the tree. Indeed, looking at when $ P_y $ is inside $ P_x $ for $ x,y\in \R $, we see that $ [(\Id ,x)] $ is a vertex of $ \mathcal{I} $ if and only if $ x\in \omega(\pi_L)\Z $.
\end{remark}

\begin{theorem}\label{Thm:localdescriptionoftheball rameven}
Let $ r\in \N $. The map $ B_{0}(2r)\to \mathcal{I}^{0,2r}\colon [(g,x)]\mapsto [(p_{2r}(g),x)]^{0,2r} $ is a $ (p_{2r}\colon P_{0}\to P_{0}^{0,2r}) $-equivariant bijection.
\end{theorem}
\begin{proof}
The map is well-defined by Lemma~\ref{Lem:integralityofn}.
\begin{itemize}
\item Injectivity: let $ [(g,x)], [(h,y)] \in B_0(2r) $ be such that they have the same image in $ \mathcal{I}^{0,2r} $. By Remark~\ref{Rem:equivalent def with for all n for local model ram even}, it means that for all $ \tilde{n}\in N^{0,2r} $ such that $ \nu (\tilde{n})(x) = y $, $ p_{2r}(g)^{-1}p_{2r}(h)\tilde{n}\in P_{x}^{0,2r} $. So, we can assume that $ \tilde{n} $ is either equal to $ \Id $, or is of the form $ \begin{psmallmatrix} 
0&0 & 1 \\
0&-1 & 0 \\
1&0&0
\end{psmallmatrix} $. Hence, there exists $ n\in N $ such that $ p_{2r}(n) = \tilde{n} $. But $ \nu (n)(x) = y $, and $ g^{-1}hn\in p_{2r}^{-1}(P_{x}^{0,2r}) \subset P_{x} $ by Lemma~\ref{Lem:injectivityofpr rameven}. Hence, $ [(g,x)] = [(h,y)] $, as wanted.
\item Surjectivity: follows directly from the surjectivity of $ p_{2r}\colon P_{0}\to P_{0}^{0,2r} $ (Corollary~\ref{Cor:surjectivityofpr rameven}).
\item Equivariance: $ h.[(g,x)] = [(hg,x)] \mapsto [(p_{2r}(hg),x)]^{0,2r} = p_{2r}(h).[(p_{2r}(g),x)]^{0,2r} $. \qedhere
\end{itemize}
\end{proof}

\subsection{Arithmetic convergence}\label{Sec:arithmetic convergence Lrameven}
In order to obtain a compact space of pairs of local fields, we need to break our convention that $ L $ is separable and also allow inseparable extensions. Note that all inseparable ramified quadratic extensions of $ \mathbf{F}_{2^n}(\!(X)\!) $ are isomorphic over an isomorphism of $ \mathbf{F}_{2^n}(\!(X)\!) $ (because $ \mathbf{F}_{2^n}(\!(X)\!) $ has many automorphisms), so that it is actually enough for our purposes to consider only one form of inseparable pair.
\begin{definition}\label{Def:admissible pairs rameven}
Consider the set of pairs of local fields $ (K,L) $ where $ K $ is of residue characteristic $ 2 $ and such that one of the following holds
\begin{enumerate}
\item $ K=\mathbf{F}_{2^n}(\!(X)\!) $ and $ L = \mathbf{F}_{2^n}(\!(\sqrt{X})\!) $ (endowed with the trivial conjugation action).
\item $ L $ is a separable quadratic ramified extension of $ K $.
\end{enumerate}
\end{definition}

\begin{definition}\label{Def:the space L rameven}
We say that two pairs $ (K_{1},L_{1}) $ and $ (K_{2},L_{2}) $ are isomorphic if there exists a conjugation equivariant isomorphism between $ L_1 $ and $ L_2 $. Let $ \mathcal{L}^{\ram}_{\even} $ be the set of pairs of local fields as in Definition~\ref{Def:admissible pairs rameven}, up to isomorphism. 
For each (non-zero) natural number $ n $, let us also define $ \mathcal{L}_{2^n}^{\ram} = \lbrace (K,L) \in \mathcal{L}^{\ram}_{\even}~\vert~ \vert \overline{K}\vert = 2^{n}\rbrace $.
\end{definition}

As in Section~\ref{Sec:SL_2(D)}, we define a metric on the space $ \mathcal{L}^{\ram}_{\even} $. For $L\in \mathcal{L}^{\ram}_{\even} $ and $ r\in \N $, the ``Galois'' conjugation (which is trivial for an inseparable pair) induces an automorphism of $ \mathcal{O}_{L}/\mathfrak{m}_{L}^{r} $ that we still call the conjugation.
\begin{definition}
Let $ (K_{1},L_{1}) $ and $ (K_{2},L_{2}) $ be in $ \mathcal{L}^{\ram}_{\even} $. We say that $ (K_{1},L_{1}) $ is $ r $-close to $ (K_2,L_2) $ if and only if there exists a conjugation equivariant isomorphism $ \mathcal{O}_{L_{1}}/\mathfrak{m}_{L_1}^{r}\to \mathcal{O}_{L_{2}}/\mathfrak{m}_{L_2}^{r} $.
\end{definition}
Again, this notion of closeness induces a non-archimedean metric on $ \mathcal{L}^{\ram}_{\even} $. Let
\begin{equation*}
d\colon \mathcal{L}^{\ram}_{\even}\times \mathcal{L}^{\ram}_{\even}\to \mathbf{R}_{\geq 0}\colon d((K_1,L_1);(K_2,L_2))= \inf \lbrace \frac{1}{2^{r}} ~\vert~ (K_1,L_1) \textrm{ is } r \textrm{-close to }(K_2,L_2) \rbrace
\end{equation*}

\begin{lemma}\label{Lem:Non-archim. metric space rameven}
$ d(\cdot~;~\cdot ) $ is a non-archimedean metric on $ \mathcal{L}^{\ram}_{\even} $. 
\end{lemma}
\begin{proof}
If $ d((K_1,L_1);(K_2,L_2)) = 0 $, then $ \mathcal{O}_{L_1} $ and $ \mathcal{O}_{L_2} $ are equivariantly isomorphic. 
Hence, the pairs of field of fraction are isomorphic in $ \mathcal{L}^{\ram}_{\even} $, as wanted. The fact that this distance is non-archimedean is a consequence of Remark~\ref{Rem:transitivityofdistance rameven}.
\end{proof}
\begin{remark}\label{Rem:transitivityofdistance rameven}
Note that being $ r $-close is an equivalence relation, and that if $ r\geq l $ and $ (K_1,L_1) $ is $ r $-close to $ (K_2,L_2) $, then $ (K_1,L_1) $ is $ l $-close to $ (K_2,L_2) $.
\end{remark}

\begin{lemma}\label{Lem:induced proximity on K rameven}
Let $ r\in \N $, and $ (K_1,L_1),~(K_2,L_2)\in \mathcal{L}^{\ram}_{\even} $. A conjugation equivariant isomorphism $ \mathcal{O}_{L_{1}}/\mathfrak{m}_{L_1}^{2r}\to \mathcal{O}_{L_{2}}/\mathfrak{m}_{L_2}^{2r} $ induces an isomorphism $ \mathcal{O}_{K_{1}}/\mathfrak{m}_{K_1}^{r}\to \mathcal{O}_{K_{2}}/\mathfrak{m}_{L_2}^{r} $.
\end{lemma}
\begin{proof}
Let $ (K,L)\in \mathcal{L}^{\ram}_{\even} $. The proof of the lemma follows if we can characterise $ \mathcal{O}_{K}/\mathfrak{m}_{K}^{r} $ inside $ \mathcal{O}_{L}/\mathfrak{m}_{L}^{2r} $ in an algebraic way. We claim that $ \mathcal{O}_{K}/\mathfrak{m}_{K}^{r} $ is the subring of $ \mathcal{O}_{L}/\mathfrak{m}_{L}^{2r} $ generated by the images of the norm and the trace map.

First assume that $L$ is a separable extension of $ K $, and let $ t,\alpha ,\beta $ be as in Lemma~\ref{Lem:goodformforL}, so that $ \mathcal{O}_{L}/\mathfrak{m}_{L}^{2r}\cong\mathcal{O}_{K}/\mathfrak{m}_{K}^{r}\oplus t.\mathcal{O}_{K}/\mathfrak{m}_{K}^{r}  $.  Let $ x = x_1+tx_2 \in \mathcal{O}_{K}/\mathfrak{m}_{K}^{r}\oplus t.\mathcal{O}_{K}/\mathfrak{m}_{K}^{r} $ and let $ i_0 $ be the parameter associated to $ L/K $ as in Definition~\ref{Def:parameteri_0}. Using $ x+\bar{x} = 2x_1 + \alpha x_2 $, we readily see that the image of the trace map generates $ \pi_K^{i_0}.\mathcal{O}_{K}/\mathfrak{m}_{K}^{r} $. Hence, we can work modulo $ (\pi_K^{i_0}) $. In particular, in view of Lemma~\ref{Lem:goodformforL}, we are in characteristic $ 2 $, and $ x\bar{x} = x_1^2+\beta x_2^2 $. Thus, the claim follows because squaring is surjective in the finite field $ \mathbf{F}_{2^n} $, so that for $ r\leq i_0 $, we have $ \mathcal{O}_{K}/\mathfrak{m}_{K}^{r} = \lbrace x_1^2+\beta x_2^2~\vert~x_i \in \mathcal{O}_{K}/\mathfrak{m}_{K}^{r} \rbrace $.

To conclude, note that if $ (K,L) $ is an inseparable pair, the ``norm'' map is just squaring while the ``trace" map is trivial, and that in this case, $ \mathcal{O}_{K}/\mathfrak{m}_{K}^{r} = (\mathcal{O}_{L}/\mathfrak{m}_{L}^{2r})^2 $ as well.
\end{proof}

We now go on to prove that $ \mathcal{L}_{2^n}^{\ram} $ is homeomorphic to $ \hat{\N}^2 $. Again, the key ingredient in this identification is Theorem~\ref{Thm:fundamental approximation lemma}. We further need a variation for ramified quadratic extension in residue characteristic $ 2 $. We begin by refining our knowledge about separable ramified extensions in characteristic $ 2 $.
\begin{lemma}\label{Lem:further results on mathcal L}
Let $ K = \mathbf{F}_{2^{n}}(\!(X)\!) $. Up to isomorphism, any separable pair of positive characteristic in $ \mathcal{L}_{2^n}^{\ram} $ is of the form $ (K, K[T]/(T^{2}-\alpha T+X)) $, for some non zero $ \alpha \in (X) $. Also, given $ i\in \N $, there are only finitely many extensions of $ K $ $ ( $up to isomorphism$ ) $ of the form $ K[T]/(T^{2}-\alpha T+X) $ where $ \alpha \in (X^{i})\setminus (X^{i+1}) $.
\end{lemma}
\begin{proof}
By Lemma~\ref{Lem:goodformforL}, any quadratic ramified extension of $ K $ is of the form $ K[T]/(T^{2}-\alpha T+\beta) $, where $ \beta \in (X)\setminus (X^{2}) $ and $ \alpha \in (X) $. Now, because $ \mathbf{F}_{2^{n}}(\!(X)\!) $ has many automorphisms, the pair $ (K, K[T]/(T^{2}-\alpha T+\beta)) $ is equivariantly isomorphic to a pair of the desired form. For the last statement, mimicking the proof of \cite{Lang94}*{Chapter~II, §5, Proposition~14}, the finiteness follows directly from the compactness of $ (X^{i})\setminus (X^{i+1}) $.
\end{proof}

We can now give the variations on Theorem~\ref{Thm:fundamental approximation lemma}:
\begin{corollary}\label{Cor:variation on fundamental lemma}
\begin{enumerate}[$ (1) $]
\item Let $ \mathbf{F}_{2^{n}}(\!(X)\!)[T]/(T^{2}-\alpha T+X) $ be a separable quadratic ramified extension of $ \mathbf{F}_{2^{n}}(\!(X)\!) $, with $ \alpha \in (X) $. Let $ K $ be a totally ramified extension of degree $ k $ of $ \mathbf{Q}_{2^{n}} $, and let $ \varphi_{\pi_K} \colon \mathcal{O}_{K}\to \mathbf{F}_{2^{n}}[\![X]\!] $ be the bijection defined in Theorem~\ref{Thm:fundamental approximation lemma}. Finally, let $ a = \varphi_{\pi_K}^{-1}(\alpha ) \in \mathcal{O}_{K} $. Then $ (K, K[T]/(T^{2}-aT+\pi_{K})) $ is $ 2k $-close to $ (\mathbf{F}_{2^{n}}(\!(X)\!),\mathbf{F}_{2^{n}}(\!(X)\!)[T]/(T^{2}-\alpha T+X) ) $.
\item $ (\mathbf{F}_{2^{n}}(\!(X)\!),\mathbf{F}_{2^{n}}(\!(X)\!)[T]/(T^{2}-X^{i}T+X)) $ is at distance $ \frac{1}{2^{2i}} $ from $ (\mathbf{F}_{2^{n}}(\!(X)\!), \mathbf{F}_{2^{n}}(\!(\sqrt{X})\!)) $. 
\end{enumerate}
\end{corollary}
\begin{proof}
\begin{enumerate}[(1)]
\item By Theorem~\ref{Thm:fundamental approximation lemma}, $ \mathcal{O}_{K}/\mathfrak{m}_{K}^{k}\cong \mathbf{F}_{2^{n}}[\![X]\!]/(X^{k}) $. Observing that for a ramified quadratic extension $ L = K[t] $ of $ K $ with $ t $ a uniformiser of $ L $, we have $ \mathcal{O}_{L}/\mathfrak{m}_{L}^{2r}\cong \mathcal{O}_{K}/\mathfrak{m}_{K}^{r}\oplus t.\mathcal{O}_{K}/\mathfrak{m}_{K}^{r} $, we directly obtain the conclusion. We could also easily conclude that the distance is $ \frac{1}{2^{2k}} $, but we do not need this information.

\item  To simplify notations, let $ L = \mathbf{F}_{2^{n}}(\!(X)\!)[T]/(T^{2}-X^{i}T+X) $. Observe that the conjugation action is trivial on $ \mathcal{O}_{L}/\mathfrak{m}_{L}^{2i} $, so that $ \mathcal{O}_{L}/\mathfrak{m}_{L}^{2i}\cong \mathbf{F}_{2^{n}}[\![X]\!]/(X^{i})\oplus \sqrt{X}.\mathbf{F}_{2^{n}}[\![X]\!]/(X^{i}) $, with trivial conjugation action. Hence, $ (\mathbf{F}_{2^{n}}(\!(X)\!),L) $ is $ 2i $-close from the inseparable pair $ (\mathbf{F}_{2^{n}}(\!(X)\!), \mathbf{F}_{2^{n}}(\!(\sqrt{X})\!))  $. Now, the conjugation action is non-trivial on $ \mathcal{O}_{L}/\mathfrak{m}_{L}^{2i+1} $, so that the distance is $ \frac{1}{2^{2i}} $.\qedhere
\end{enumerate} 
\end{proof}

We deduce the homeomorphism type of $ \mathcal{L}_{2^{n}}^{\ram} $.
\begin{proposition}\label{Prop:explicit description of L rameven}
The space $ \mathcal{L}_{2^{n}}^{\ram} $ is homeomorphic to $ \hat{\N}^2 $. Its first Cantor--Bendixson derivative $ (\mathcal{L}_{2^{n}}^{\ram})^{(1)} $ consists of pairs of positive characteristic, while its second Cantor--Bendixson derivative is the singleton consisting of the inseparable pair $ (\mathbf{F}_{2^{n}}(\!(X)\!), \mathbf{F}_{2^{n}}(\!(\sqrt{X})\!))  $.
\end{proposition}

\begin{proof} \setcounter{claim}{0}
$  $

\begin{claim}\label{Claim:1 rameven}
Let $ (K,L)\in \mathcal{L}_{2^{n}}^{\ram} $. If $ K $ is of characteristic $ 0 $, $ (K,L) $ is isolated in $ \mathcal{L}_{2^{n}}^{\ram} $.
\end{claim}
\begin{claimproof}
If $ (K_1,L_1) $ is $ r $-close to $ (K_2,L_2) $, then $ K_1 $ is $ \lfloor \frac{r}{2}\rfloor $-close to $ K_2 $ by Lemma~\ref{Lem:induced proximity on K rameven}. Hence, the result follows from Claim~\ref{Claim:2'} and Claim~\ref{Claim:3'} in the proof of Proposition~\ref{Prop:explicit description of D}.
\end{claimproof}

\medskip

\begin{claim}\label{Claim:2 rameven} 
$ \mathcal{L}_{2^{n}}^{\ram} $ is a countable space.
\end{claim}
\begin{claimproof}
By Claim~\ref{Claim:3'} in the proof of Proposition~\ref{Prop:explicit description of D}, there are only countably many pairs of characteristic $ 0 $ in $ \mathcal{L}_{2^{n}}^{\ram} $. Furthermore, there is only one inseparable pair by definition, and there are countably many separable pair of characteristic $2$ in $ \mathcal{L}_{2^{n}}^{\ram} $ by Lemma~\ref{Lem:further results on mathcal L}.
\end{claimproof}

\medskip

We can now conclude the proof: since pairs of characteristic $ 0 $ are isolated by Claim~\ref{Claim:4'}, the first Cantor--Bendixson derivative $ (\mathcal{L}_{2^{n}}^{\ram})^{(1)} $ contains only pairs of positive characteristic, and $ (\mathcal{L}_{2^{n}}^{\ram})^{(1)} $ contains all of them by Corollary~\ref{Cor:variation on fundamental lemma}~($1$). Also, by Corollary~\ref{Cor:variation on fundamental lemma}~($2$) and Lemma~\ref{Lem:further results on mathcal L}, separable pairs are isolated in $ \mathcal{L}_{2^{n}}^{(1)} $, and the inseparable pair is an accumulation point in $ (\mathcal{L}_{2^{n}}^{\ram})^{(1)} $. So that again by \cite{MS20}*{Théorème~1}, we get $ \mathcal{L}_{2^{n}}^{\ram}\cong \hat{\N}^2 $.
\end{proof}

\subsection{Continuity from pairs in \texorpdfstring{$ \mathcal{L}^{\ram}_{\even} $}{L^{ram}_{even}} to subgroups of \texorpdfstring{$ \Aut (T) $}{Aut(T)}}
In this section, we start to vary the ramified pair $ (K,L) $, and look at the variation it produces on the Bruhat--Tits tree of $ \SU_3^{L/K} $. Recall that we introduced a notation to keep track of the dependence on $ (K,L) $ of many of the definitions we made in this section (see Remark~\ref{Rem:dependence on D of the tree rameven} and Remark~\ref{Rem:dependence on D of the local tree rameven}). 

Note however that in Section~\ref{Sec:arithmetic convergence Lrameven} we were forced to consider inseparable pairs, i.e. pairs of the form $ (\mathbf{F}_{2^n}(\!(X)\!), \mathbf{F}_{2^n}(\!(\sqrt{X})\!) ) $, and that we have not yet associated any object to those pairs. The definition of $ \SU_3^{L/K} $ still makes sense for $ L $ an inseparable extension, but this $K$-group scheme is not smooth. In fact, it is the Weil restriction from $ L $ to $ K $ of the naive split special orthogonal group $ \SO_3 $ in characteristic $ 2 $. Instead of $ \SU_3^{L/K} $, the algebraic group associated to the inseparable pair $ (K,L) $ should be the algebraic group $ \SL_2(L) $. It is the group of $ K $-rational point of $ \mathcal{R}^L_K \SL_2 $, the Weil restriction from $ L $ to $ K $ of $ \SL_2 $ (as an aside, note that $ \mathcal{R}^L_K \SL_2 $ is the prototypical example of a non-reductive pseudo-reductive group). The preceding discussion motivates the following definitions.

\begin{definition}\label{Def:data for the inseparable pair}
Let $ (K,L) $ be an inseparable pair in $ \mathcal{L}^{\ram}_{\even} $. We set $ (P_x)_{(K,L)} = (P_x)_L $, $ T_{(K,L)}=T_L $, $ M_{(K,L)}=M_L $, $ N_{(K,L)}=N_L $, $ \nu_{(K,L)}=\nu_L $ and $ \mathcal{I}_{(K,L)}=\mathcal{I}_L $, where objects appearing on the right hand side of an equality were defined in Remark~\ref{Rem:dependence on D of the tree}. And similarly for local objects, we set $ (P_x^{0,2r})_{(K,L)} = (P_x^{0,2r})_{L}  $, $ H^{0,2r}_{(K,L)} = H^{0,2r}_{L} $, $ M^{0,2r}_{(K,L)} = M^{0,2r}_{L} $, $ N^{0,2r}_{(K,L)} = N^{0,2r}_{L} $ and $ \mathcal{I}^{0,2r}_{(K,L)} = \mathcal{I}^{0,2r}_{L} $, where objects appearing on the right hand side of an equality were defined in Remark~\ref{Rem:dependence on D of the local tree}. Finally, we set $ (i_0)_{(K,L)} = \infty $.
\end{definition}

\begin{definition}\label{Def:associating groups to pair}
Let $ (K,L)\in \mathcal{L}^{\ram}_{\even} $.
\begin{enumerate}[$ \bullet $]
\item If $ (K,L) $ is separable, we set $ G_{(K,L)} = \SU_3^{L/K}(K) = \lbrace g\in \SL_3(L)~\vert~^S\bar{g}g=\Id \rbrace $.
\item If $ (K,L) $ is inseparable, we set $ G_{(K,L)} = \SL_2(L) $.
\end{enumerate}
\end{definition}

As in the previous section, we now aim to prove that when two pairs in $ \mathcal{L}^{\ram}_{\even} $ are close, their local Bruhat--Tits tree are equivariantly isomorphic. The appearance of the reduced trace in the local model makes it a bit less straightforward, so that we first need the following lemmas.
\begin{lemma}\label{Lem:proximity implies nearly same i0}
Let $ (K_1,L_1) $ and $ (K_2,L_2) $ be two pairs in $ \mathcal{L}^{\ram}_{\even} $, and assume that that they are $ r $-close for some $ r\in \mathbf{N} $. If $ r\leq (i_0)_{(K_1,L_1)} $, then $ r\leq (i_0)_{(K_2,L_2)} $, while if $ r> (i_0)_{(K_1,L_1)} $, then $ (i_0)_{(K_2,L_2)} = (i_0)_{(K_1,L_1)} $.
\end{lemma}
\begin{proof}
Note that the conjugation is trivial on $ \mathcal{O}_{L_i}/\mathfrak{m}_{L_i}^{r} $ if and only if $ r\leq (i_0)_{(K_i,L_i)} $, so that the first assertion is clear. On the other hand, note that $ (i_0)_{(K_i,L_i)} $ is the largest integer such that for all units $ x\in \mathcal{O}_{L_i} $, $ x+\bar{x} \in \mathfrak{m}_{K_i}^{i_0} $. Indeed, for $ t_i, \alpha_i $ and $ \beta_i $ as in Lemma~\ref{Lem:goodformforL}, we see that either $ [Tr(1+t_i) = 2+\alpha_i] $ or $ [Tr(1) = 2] $ belong to $ \mathfrak{m}_{K_i}^{i_0}\setminus \mathfrak{m}_{K_i}^{i_0+1} $. Hence, the second assertion follows.
\end{proof}

We now give a lemma allowing us to control the reduced norm. Recall that the valuation $ \omega $ on $ \mathcal{O}_L $ induces a map on $ \mathcal{O}_L/\mathfrak{m}_L^r $ that we still denote $ \omega $. By a uniformiser of $ \mathcal{O}_L/\mathfrak{m}_L^r $, we mean a non-invertible element of minimal image under $ \omega $ (amongst non-invertible elements of $ \mathcal{O}_L/\mathfrak{m}_L^r $). Uniformisers of $ \mathcal{O}_K/\mathfrak{m}_K^r $ are defined in a similar way.

Recall that by Lemma~\ref{Lem:reduced trace does not depend on lifting uniformiser}, given a uniformiser $ \pi_K\in \mathcal{O}_K/\mathfrak{m}_K^r $, we get for every $ s<r $ a unique map (not depending on the lift of $ \pi_K $) $ \frac{Tr}{\pi_K^{i_0}}\colon \mathcal{O}_L/\mathfrak{m}_L^{2s}\to\mathcal{O}_K/\mathfrak{m}_K^s $. We use this fact in the statement of the following lemma.

\begin{lemma}\label{Lem:controlling the reduced norm}
Let $ (K_1,L_1) $ and $ (K_2,L_2) $ be two pairs in $ \mathcal{L}^{\ram}_{\even} $. Assume that $ (K_1,L_1) $ is separable and let $ i_0 = (i_0)_{(K_1,L_1)} $. Assume that the two pairs are $ 2r+2i_0 $-close, and let $ \varphi \colon \mathcal{O}_{L_1}/\mathfrak{m}_{L_1}^{2r+2i_0}\to \mathcal{O}_{L_2}/\mathfrak{m}_{L_2}^{2r+2i_0} $ be the given conjugation equivariant isomorphism. We also denote $ \varphi $ the induced equivariant isomorphism $ \mathcal{O}_{L_1}/\mathfrak{m}_{L_1}^{2r}\to \mathcal{O}_{L_2}/\mathfrak{m}_{L_2}^{2r} $. Finally let $ \pi_{K_1} $ be a uniformiser of $ \mathcal{O}_{K_1}/\mathfrak{m}_{K_1}^{r+i_0} $. Then for all $ x\in \mathcal{O}_{L_1}/\mathfrak{m}_{L_1}^{2r} $,  we have $ \varphi (\frac{Tr}{\pi_{K_1}^{i_0}}(x)) = \frac{Tr}{\varphi (\pi_{K_1})^{i_0}}(\varphi (x)) $.  
\end{lemma}
\begin{proof}
First note that $ \varphi (\pi_{K_1}) $ is a uniformiser of $ \mathcal{O}_{K_2}/\mathfrak{m}_{K_2}^{r+i_0} $. Indeed, by Lemma~\ref{Lem:induced proximity on K rameven}, $ \varphi (\pi_{K_1}) $ is an element of $ \mathcal{O}_{K_2}/\mathfrak{m}_{K_2}^{r+i_0} $, and the fact that it is a uniformiser follows from the fact that $ \varphi (\pi_{K_1})^{r+i_0-1}\neq 0 $ but $ \varphi (\pi_{K_1})^{r+i_0}= 0 $.

To simplify notations, let $ \lambda = \frac{Tr}{\pi_{K_1}^{i_0}}(x) $. Let $ x' $ (respectively $ \lambda ' $) be a lift of $ x $ (respectively $ \lambda $) to $ \mathcal{O}_{L_1}/\mathfrak{m}_{L_1}^{2r+2i_0} $ (respectively to $ \mathcal{O}_{K_1}/\mathfrak{m}_{K_1}^{r+i_0} $). We claim that $ \pi_{K_1}^{i_0}\lambda ' = x'+\overline{x'} $ in $ \mathcal{O}_{K_1}/\mathfrak{m}_{K_1}^{r+i_0} $. Indeed, by the definition of $ \frac{Tr}{\pi_{K_1}^{i_0}} $, there exists a lift $ x'' $ (respectively $ \lambda '' $, $ \pi_{K_1} '' $) of $ x $ (respectively $ \lambda $, $ \pi_{K_1} $) to $ \mathcal{O}_{L_1} $ (respectively to $ \mathcal{O}_{K_1} $, $ \mathcal{O}_{K_1} $) such that $ (\pi_{K_1} '')^{i_0}\lambda '' = x''+\overline{x''} $. But note that neither $ \pi_{K_1}^{i_0}\lambda ' $ nor $ x'+\overline{x'} $ depends on the chosen lifts $ x' $ and $ \lambda ' $. Hence, letting $ p_{2r+2i_0}\colon \mathcal{O}_{L_1}\to \mathcal{O}_{L_1}/\mathfrak{m}_{L_1}^{2r+2i_0} $, we can assume that $ x'=p_{2r+2i_0}(x'') $ and $ \lambda ' = p_{2r+2i_0}(\lambda '') $, so that the claim holds.

But now, we get $ \varphi (\pi_{K_1})^{i_0}\varphi (\lambda ') = \varphi (x')+\overline{\varphi (x')} $ in $ \mathcal{O}_{K_2}/\mathfrak{m}_{K_2}^{r+i_0} $. Furthermore, $ \varphi (x') $ (respectively $ \varphi (\lambda ') $) is a lift of $ \varphi (x) $ (respectively $ \varphi (\lambda ) $). We therefore conclude that $ \varphi (\lambda ) = \frac{Tr}{\varphi (\pi_{K_1})^{i_0}}(\varphi (x))  $, as wanted.
\end{proof}

\begin{proposition}\label{Prop:aritmimpliesgeomproximityforSU3 rameven}
Let $ (K_1,L_1) $ and $ (K_2,L_2) $ be two elements in $ \mathcal{L}^{\ram}_{\even} $. Assume that $ (K_1,L_1) $ is $ 2r $-close to $ (K_2,L_2) $, for some $ r \in \N $. Let $ i_0 = (i_0)_{(K_1,L_1)} $ be the parameter associated to $ (K_1,L_1) $ $ ( $see Definition~\ref{Def:parameteri_0} and Definition~\ref{Def:data for the inseparable pair}$ ) $.
\begin{enumerate}
\item If $ 2r\leq 2i_0 $, then $ (P_{0}^{0,2r})_{(K_1,L_1)}\cong (P_{0}^{0,2r})_{(K_2,L_2)} $, and $ \mathcal{I}_{(K_1,L_1)}^{0,2r}$ is equivariantly in bijection with $ \mathcal{I}_{(K_2,L_2)}^{0,2r} $.
\item If $ 2r > 2i_0 $, then $ (P_{0}^{0,2r-2i_0})_{(K_1,L_1)}\cong (P_{0}^{0,2r-2i_0})_{(K_2,L_2)} $, and $ \mathcal{I}_{(K_1,L_1)}^{0,2r-2i_0}$ is equivariantly in bijection with $ \mathcal{I}_{(K_2,L_2)}^{0,2r-2i_0} $.
\end{enumerate}
\end{proposition}
\begin{proof}
When $ 2r\leq 2i_0 $, then $ 2r\leq (2i_0)_{(K_2,L_2)} $ as well by Lemma~\ref{Lem:proximity implies nearly same i0}. In view of Definition~\ref{Def:localdataforbadSU} and Definition~\ref{Def:data for the inseparable pair}, the isomorphism $ \mathcal{O}_{L_1}/\mathfrak{m}_{L_1}^{2r}\cong \mathcal{O}_{L_2}/\mathfrak{m}_{L_2}^{2r} $ induces a group isomorphism $ \varphi \colon (P_0^{0,2r})_{(K_1,L_1)} = \SL_2(\mathcal{O}_{L_1}/\mathfrak{m}_{L_1}^{2r})\cong \SL_2(\mathcal{O}_{L_2}/\mathfrak{m}_{L_2}^{2r}) = (P_0^{0,2r})_{(K_2,L_2)} $.

When $2r>2i_0$, then $ (i_0)_{(K_2,L_2)} = i_0 $ by Lemma~\ref{Lem:proximity implies nearly same i0}. If $ 2r\leq 4i_0 $, then the isomorphism $ \mathcal{O}_{L_1}/\mathfrak{m}_{L_1}^{2r}\cong \mathcal{O}_{L_2}/\mathfrak{m}_{L_2}^{2r} $ induces a group isomorphism $ \varphi \colon (P_0^{0,2r-2i_0})_{(K_1,L_1)} = \SL_2(\mathcal{O}_{L_1}/\mathfrak{m}_{L_1}^{2r-2i_0})\cong \SL_2(\mathcal{O}_{L_2}/\mathfrak{m}_{L_2}^{2r-2i_0}) = (P_0^{0,2r-2i_0})_{(K_2,L_2)} $. On the other hand, if $ 2r> 4i_0 $, then in view of Lemma~\ref{Lem:local stabiliser does not depend on choice of uniformiser} and Lemma~\ref{Lem:controlling the reduced norm}, the isomorphism $ \mathcal{O}_{L_1}/\mathfrak{m}_{L_1}^{2r}\cong \mathcal{O}_{L_2}/\mathfrak{m}_{L_2}^{2r} $ induces a group isomorphism $ \varphi $
\begin{center}
\begin{tikzpicture}[->]
 \node (1) at (0,0) {$ \SL_{3}(\mathcal{O}_{L_1}/\mathfrak{m}_{L_1}^{2r-2i_0}) $};
 \node (2) at (3.6,0) {$ \SL_{3}(\mathcal{O}_{L_2}/\mathfrak{m}_{L_2}^{2r-2i_0}) $};
 \node (7) at (1.8,0) {$ \cong$};
 \node (5) at (-0.5,-0.4) {$ \vee $};
 \node (3) at (-0.2,-0.9) {$ (P_{0}^{0,2r-2i_0})_{(K_1,L_1)} $};
 \node (4) at (3.4, -0.9) {$ (P_{0}^{0,2r-2i_0})_{(K_2,L_2)} $};
 \node (6) at (3.00,-0.4) {$ \vee $};

\path [every node/.style={font=\sffamily\small}]
 (3) edge node [below]{$ \varphi $} (4);

\end{tikzpicture}
\end{center}

Let $\varepsilon = \begin{cases}
 2r \text{ if } 2r\leq 2i_0\\
2r-2i_0 \text{ if } 2r>2i_0
\end{cases} $. In both cases, define a linear map $ f\colon \mathbf{R}\to \mathbf{R}\colon x\mapsto x\frac{\omega (\pi_{L_2})}{\omega (\pi_{L_1})} $. It is clear that for all $ x\in [-\omega (\pi_{L_1}^{\varepsilon}), \omega (\pi_{L_1}^{\varepsilon})] $, $ \varphi $ restricts to an isomorphism $ (P_{x}^{0,\varepsilon})_{(K_1,L_1)}\cong (P_{f(x)}^{0,\varepsilon})_{(K_2,L_2)} $. Furthermore, 
\begin{align*}
\varphi (T^{0,\varepsilon})_{(K_1,L_1)}&= (T^{0,\varepsilon})_{(K_2,L_2)}\\
\varphi (M^{0,\varepsilon})_{(K_1,L_1)}&= (M^{0,\varepsilon})_{(K_2,L_2)}
\end{align*}
and for all $ n\in N^{0,\varepsilon} $, $ f(n.x) = \varphi (n).f(x) $. Hence, the map $ \mathcal{I}_{(K_1,L_1)}^{0,\varepsilon}\to \mathcal{I}_{(K_1,L_1)}^{0,\varepsilon}\colon [(g,x)]^{0,\varepsilon}\mapsto [(\varphi (g), f(x))]^{0,\varepsilon} $ is a $ \varphi $-equivariant bijection.
\end{proof}

We again discuss the homomorphism $ \SU_3^{L/K}(K) \to \Aut (\mathcal{I}_{(K,L)}) $.
\begin{proposition}\label{Prop:Embedding G(K) in Aut(T) rameven}
Let $ \mathcal{I} = \mathcal{I}_{(K,L)} $ be the Bruhat--Tits tree of $ \SU_3^{L/K}(K) $. The homomorphism~ $ \hat{}~\colon \SU_3^{L/K}(K)\to \Aut (\mathcal{I}) $ is continuous with closed image, and the kernel is equal to the centre of $ \SU_3^{L/K}(K) $. 
\end{proposition}
\begin{proof}
The proof is word for word the same as the proof of Proposition~\ref{Prop:Embedding G(K) in Aut(T)}, upon replacing $ \SL_2(D) $ by $ \SU_3^{L/K}(K) $.
\end{proof}

The convergence is then a more or less direct consequence of Theorem~\ref{Thm:localdescriptionoftheball rameven}.
\begin{theorem}\label{Thm:continuity from Lrameven to Chab(Aut(T))}
Let $ ((K_i,L_i))_{i\in \mathbf{N}} $ be a sequence in $ \mathcal{L}^{\ram}_{\even} $ which converges to $ (K,L) $, and let $G_i = G_{(K_i,L_i)} $ $ ( $respectively $ G = G_{(K,L)} )$. For $ N $ big enough and for all $ i\geq N $, there exist isomorphisms $ \mathcal{I}_{(K_i,L_i)} \cong \mathcal{I}_{(K,L)} $ such that the induced embeddings~ $ \hat{G_i}\hookrightarrow \Aut (\mathcal{I}_{(K,L)}) $ make $ (\hat{G_i})_{i\geq N} $ converge to $ \hat{G} $ in the Chabauty topology of $ \Aut (\mathcal{I}_{(K,L)}) $.
\end{theorem}
\begin{proof}
The Bruhat--Tits tree $ \mathcal{I}_{(K_i,L_i)} $ is the regular tree of degree $ 2^n+1 $ if and only if $ (K_i,L_i) $ belongs to $ \mathcal{L}_{2^n}^{\ram} $. Hence there exists $ N $ such that for all $ i\geq N $, $  \mathcal{I}_{(K_i,L_i)}\cong  \mathcal{I}_{(K,L)} $.

Let $ i_0 $ be the parameter associated to $ (K,L) $ as in Definition~\ref{Def:parameteri_0} and Definition~\ref{Def:data for the inseparable pair}. First assume that $ i_0 $ is infinite (or in other words that $ (K,L) $ is an inseparable pair). Then by Lemma~\ref{Lem:proximity implies nearly same i0}, the sequence $ (i_0)_{(K_i,L_i)} $ diverge, and hence up to passing to a subsequence, we can assume that $ (i_0)_{(K_i,L_i)}\geq i $ and that $ (K_i,L_i) $ is $ 2i $-close to $ (K,L) $. On the other hand, when $ i_0 $ is finite, up to passing to a subsequence, we can assume that $ (K_i,L_i) $ is $ 2i+2i_0 $-close to $ (K,L) $.

We now define an explicit isomorphism $ f_i\colon \mathcal{I}_{(K_i,L_i)}\to \mathcal{I}_{(K,L)} $ as follows: let $ \mathcal{I}_{(K_i,L_i)}^{0,i}\cong \mathcal{I}_{(K,L)}^{0,i} $ be the isomorphism given by Proposition~\ref{Prop:aritmimpliesgeomproximityforSU3 rameven}. By Theorem~\ref{Thm:localdescriptionoftheball rameven}, this gives an isomorphism on balls of radius $ i $: $ \mathcal{I}_{(K_i,L_i)}\supset B_0(i)\cong B_0(i)\subset \mathcal{I}_{(K,L)} $ (recall that by Lemma~\ref{Lem:B_0(r) is really the ball of radius r rameven}, $ B_0(i) $ is really the ball of radius $ i $ on the tree $ \mathcal{I}_{(K,L)} $). As $ \mathcal{I}_{(K_i,L_i)} $ is a regular tree of the same degree than $ \mathcal{I}_{(K,L)} $, we can extend this isomorphism of balls to an isomorphism $ f_{i}\colon \mathcal{I}_{(K_i,L_i)}\to \mathcal{I}_{(K,L)} $ (this extension is of course not unique, but we choose one such). By means of $ f_i $, we get an embedding $ \hat{G}_i\hookrightarrow \Aut ( \mathcal{I}_{(K,L)}) $.

Now the end of the proof is word for word the same as the corresponding end of the proof of Theorem~\ref{Thm:continuity from D to Chab(Aut(T))}, upon making the following changes: replace $ D_i $ with $(K_i,L_i)$, replace $D$ with $(K,L)$, replace $d$ with $2$, and replace all references to results in Section~\ref{Sec:SL_2(D)} by their corresponding results in Section~\ref{Sec:SU3L/K ramifiedeven}.
\end{proof}

We then deduce the proof of the main theorem announced in the introduction for regular trees of degree $ 2^n+1 $. Recall the notation $ G_{(K,L)} $ introduced in Definition~\ref{Def:associating groups to pair}. Furthermore, as in Section~\ref{Sec:SL_2(D)}, we set $ G_{K} = \SL_2(K) $. Recall also the notation $ \mathcal{K}_{p^n} $ introduced in Definition~\ref{Def:the space D}. We use this notation in the following proof with $ p=2 $.
\begin{proof}[Proof of Theorem~\ref{Thm:explicit form of main theorem rameven}]
Let $ T $ be the ($ 2^n+1 $)-regular tree. Paralleling the proof of Theorem~\ref{Thm:explicit form of main theorem ur}, we see that the maps $ \mathcal{L}^{\ram}_{2^n}\to\mathcal{S}_T\colon (K,L)\mapsto \hat{G}_{(K,L)} $ and $ \mathcal{K}_{2^n}\to\mathcal{S}_T\colon K\mapsto \hat{G}_{K} $ are injective continuous map whose source is a compact space, hence they are homeomorphism onto their respective image. Now, the explicit description given in Theorem~\ref{Thm:explicit form of main theorem rameven} follows from Remark~\ref{Rem:regularity of the B-T tree SU rameven}, Proposition~\ref{Prop:explicit description of L rameven} and Proposition~\ref{Prop:explicit description of D}.
\end{proof}

\appendix
\section{Comparison with the original Bruhat--Tits definitions}\label{App:A}

The purpose of this appendix is to show that our definitions of the Bruhat--Tits tree agrees with the one in \cite{BrTi1}*{7.4.1 and 7.4.2}. Since the relative rank of $ \SL_2(D) $ and $ \SU_3 $ is $ 1 $, it is already clear that the apartment $ A $ is indeed isomorphic to $ \R $. The main task is to show that our group $ P_x $ is the same as the group $ \hat{P}_{x} $ used to define the equivalence relation in \cite{BrTi1}*{7.4.1}.

In the $ \SL_2(D) $ case, the explicit description of $ \hat{P}_{x} $ is given in \cite{BrTi1}*{Corollaire~10.2.9}, that we take as a definition.
\begin{definition}[\cite{BrTi1}*{Corollaire~10.2.9}]\label{Def:hatPx in B-T for SL_2}
Let $ \lbrace a_1,a_2\rbrace $ be the canonical basis of $ \R^{2} $, and let $ a_{ij} = a_j-a_i $ ($ i,j\in \lbrace 1,2\rbrace $). Identify $ \R $ with a vector space $ V $, whose dual is the vector space $ V^{\ast} = \R .a_{12} $. For $ x\in \R $, we set $ \hat{P}_{x} = \lbrace g\in \SL_{2}(K)~\vert~ \omega(g_{ij})\geq a_{ji}(x), \text{ for all }1\leq i,j\leq 2 \rbrace $.
\end{definition}

Note that we can omit the factor $ (r+1)^{-1}\delta $ appearing in loc. cit. since by definition, $ \delta = \omega (\det (g)) = \omega (1) = 0 $. 

This description obviously depends on the identification of $ \R $ as the dual of $ V^{\ast} $. Now, if we furthermore impose the condition $ a_{12} = \Id\colon \R\to \R $, then $ \hat{P}_{x} $ is indeed equal to the group $ P_x $ of Definition~\ref{Def:stabiliserforSL2}. To end the comparison between \cite{BrTi1}*{Définition~7.4.2} and our definitions, one has to show that $ N=T\sqcup M $ and the maps $ \nu \colon N\to \Aff (\R ) $ are the same. This is easily obtained by comparing \cite{BrTi1}*{Proposition~10.2.5 (ii)} with Definition~\ref{Def:sbgrp N for SL2} and Definition~\ref{Def:affineactionforNSL2}.

We now treat all the $ \SU_3 $ cases at once. As in Definition~\ref{Def:SU(3,f)}, we index the rows and the columns of a $ 3 $-by-$ 3 $ matrix by $ \lbrace -1,0,1\rbrace $. Let $ a_1 $ be a non-trivial element of $ \R^{\ast} $, and set $ a_{-1} = -a_1 $ and $ a_0 = 0 $. We now take some time to spell out the definition of $ \overline{\omega}_{ij} $ as defined in \cite{BrTi1}*{10.1.27}. This requires to extend the definition of $ \gamma $ (see Definition~\ref{Def:gamma}) to all separable quadratic extensions.
\begin{lemma}\label{Lem:goodformforLgeneral}
Let $ L $ be a separable quadratic extension of $ K $. There exists $ t\in L $ and $ \alpha ,\beta \in K $ such that:
\begin{enumerate}
\item  $ L=K[t] $ and $ t^{2}-\alpha t+\beta = 0 $.
\item $ \omega (\beta) = 0 $ when $ L $ is unramified, and $ \beta $ is a uniformiser of $ K $ when $ L $ is ramified.
\item $ \alpha = 0 $, or $ 0 = \omega (\beta) = \omega (\alpha) < \omega (2) $, or $ 0<\omega (\beta)\leq \omega (\alpha)\leq \omega (2) $.
\end{enumerate}
\end{lemma}
\begin{proof}
See \cite{BrTi2}*{Lemme~4.3.3, (ii)}. The fact that $ \alpha $ can be chosen so that $ \omega (\alpha ) = 0 $ in the unramified case is a direct consequence of the theory of unramified extensions of local fields (see for example \cite{Fes02}*{Chapter~II, Section~3.2, Proposition}). With this in mind, the equivalence with \cite{BrTi2}*{Lemme~4.3.3, (ii)} is clear. 
\end{proof}
\begin{remark}\label{Rem:good form for L}
To make Lemma~\ref{Lem:goodformforLgeneral} possibly clearer, let us state what is the valuation of $ \alpha $ on a case-by-case analysis:
\begin{enumerate}
\item If $ L $ is unramified, 
$\begin{cases}
\alpha = 0 \text{ if the residue characteristic is not } 2\\
\omega (\alpha ) = 0 \text{ if the residue characteristic is } 2
\end{cases} $
\item If $ L $ is ramified, 
$\begin{cases}
\alpha = 0 \text{ if the residue characteristic is not } 2\\
\alpha = 0 \text{ or } 0<\omega (\alpha )\leq \omega (2) \text{ if the residue characteristic is } 2
\end{cases} $
\end{enumerate}
The only difference between Remark~\ref{Rem:good form for L} and Lemma~\ref{Lem:goodformforLgeneral} is that the latter allows the possibility that $ \alpha = 0 $ in the unramified residue characteristic $ 2 $ case. But this clearly cannot happen.
\end{remark}

We can now extend our definition of the parameter $ \gamma $ to any separable quadratic extension $L/K$.

\begin{definition}\label{Def:l}
Let $ L/K $ be a separable quadratic extension, and let $ t, \alpha ,\beta $ be as in Lemma~\ref{Lem:goodformforLgeneral}. Let $ l = t\alpha^{-1} \in L $ if $ \alpha \neq 0 $, and $ l = \frac{1}{2} \in L $ if $ \alpha = 0 $, where $ \alpha $ is as in Lemma~\ref{Lem:goodformforLgeneral} (note that $ \alpha = 0 $ implies $ 2\neq 0 $ in $ K $, since $ L $ is assumed to be a separable extension). We then define $ \gamma = -\frac{1}{2}\omega (l) \in \R $.
\end{definition}

Again, we can restate the fact that $ \gamma $ does not depend on the choice of $ \alpha $ and $ \beta $ as in Lemma~\ref{Lem:goodformforLgeneral} in the following way.
\begin{lemma}\label{Lem:lambdaisinL1max}
Let $ L^{1} = \lbrace x \in L~\vert~ x+\bar{x} = 1\rbrace $ and $ L_{\max}^{1} = \lbrace x\in L^{1}~\vert~ \omega (x) = \sup \lbrace \omega (x)~\vert~ x\in L^{1}\rbrace \rbrace $. The element $ l \in L $ in Definition~\ref{Def:l} belongs to $ L_{\max}^{1} $
\end{lemma}
\begin{proof}
See \cite{BrTi2}*{4.3.3 (ii)}.
\end{proof}

It is also important to note that $ \gamma \geq 0 $, and that in view of Remark~\ref{Rem:good form for L}, $ \gamma > 0 $ if and only if the residue characteristic is $ 2 $ and $ L $ is a ramified extension. 

\begin{definition}[{\cite{BrTi1}*{10.1.20}}]
Let $ q $ be the pseudo-quadratic form associated with the hermitian form used to defined $ \SU_3 $ (see Remark~\ref{Rem:thehermitianform}). Explicitly, for $ x\in L^{3} $, $ q(x) = l f(x,x) + L^{0} $, where $ L^{0} = \lbrace x \in L~\vert~ x+\bar{x} = 0\rbrace $ (see \cite{BrTi1}*{10.1.1 (7), (8)}). For $ x\in L $, we define $ \omega_{q}(x)=\frac{1}{2}\sup \lbrace \omega (k)~\vert~ k\in q((0,x,0))\rbrace = \frac{1}{2}\sup \lbrace \omega (k)~\vert~ k\in l \bar{x}x + L^{0}\rbrace $.
\end{definition}

We can actually compute explicitly the value of $ \omega_{q} $.
\begin{lemma}
$ $
\begin{enumerate}
\item $ \omega_{q}(x)=\omega (x)+\omega_{q}(1) $
\item $ \omega_{q}(1) = \frac{1}{2}\omega (l ) $
\end{enumerate}
Hence, $ \omega_{q}(x) = \omega (x)+\frac{1}{2}\omega (l ) $
\end{lemma}
\begin{proof}
The first property follows from the definition, and the second one is Lemma~\ref{Lem:lambdaisinL1max}.
\end{proof}

\begin{definition}[{\cite{BrTi1}*{10.1.27}}]
Let $ \lbrace e_{-1}, e_0, e_1\rbrace $ be the canonical basis of $ L^{3} $. For $ g\in End(L^{3}) $, let $ (g_{ij}) $ be the matrix of $ g $ in the basis $ \lbrace e_{-1}, e_0, e_1\rbrace $. For $ i,j\in \lbrace -1,0,1\rbrace $, we define $ \bar{\omega}_{ij}(g) = \tilde{\omega}_{i}(g_{ij})-\tilde{\omega}_{j}(1) $, where $ \tilde{\omega}_{\pm 1} = \omega $, while $ \tilde{\omega}_{0} = \omega_{q} $.
\end{definition}
\begin{remark}
One readily check that this definition agrees with the one given in \cite{BrTi1}*{10.1.27}. Indeed, we can take advantage of the fact that $ X_0 $ is one dimensional. Let us identify $ \Hom (X_j,X_i) $ with $ L $, through the basis $ \lbrace e_{-1}, e_0, e_1\rbrace $, and define $ \omega_i $ as in \cite{BrTi1}*{10.1.27}. Then, for $ x\in L $ and $ \alpha \in \Hom (X_j,X_i)\cong L $, we have $ \omega_i(\alpha (xe_j))-\omega_{j}(xe_j) = \omega_i((\alpha x)e_i)-\omega_{j}(xe_j) = \tilde{\omega}_i(\alpha x) - \tilde{\omega}_j (x) = \tilde{\omega}_i(\alpha ) - \tilde{\omega}_j (1) $.
\end{remark}

\begin{definition}[{\cite{BrTi1}*{Corollaire~10.1.33}}]\label{Def:stabilisersinTitsforSU}
With these notations, $ \hat{P}_{x} = \lbrace g\in \SU_3(K)~\vert~ \overline{\omega}_{ij}(g)\geq a_{i}(x)-a_{j}(x), i,j\in \lbrace -1,0,1\rbrace \rbrace $.
\end{definition}

Note that we can omit the factor $ \frac{1}{2}\omega c(g) $ appearing in loc. cit. since by definition, $ c(g) $ is the similitude ratio (see \cite{BrTi1}*{Definition~10.1.4}) and is equal to $ 1 $ for $ g\in \SU_3 $.

Again, this description depends on the choice of a non-trivial element in $ \R^{\ast} $. Now, if we choose $ a_{1}\colon \R\to \R \colon x\to \frac{x}{2} $, then for $ x\in \R $, the group $ \hat{P}_x $ of Definition~\ref{Def:stabilisersinTitsforSU} is the following group:

$$ \hat{P}_x = \lbrace g\in \SU_{3}^{L/K}(K)~\vert~ \omega (g)\geq \begin{psmallmatrix}
0 & -\frac{x}{2} - \gamma & -x \\ 
\frac{x}{2} + \gamma & 0 & -\frac{x}{2} + \gamma \\
x & \frac{x}{2} - \gamma & 0
\end{psmallmatrix}\rbrace $$

When $ \gamma = 0 $, i.e. when $ L $ is unramified or when the residue characteristic is not $ 2 $, then this indeed coincides with our definition of $ P_x $ (see Definition~\ref{Def:pointstabiliserforSU unramified} and Definition~\ref{Def:pointstabiliserforSU ramified}). Finally, when $ \gamma > 0 $, the group $ P_x $ of Definition~\ref{Def:pointstabiliserforSU ramifiedeven} also coincides with $ \hat{P}_x $.

To end the comparison between \cite{BrTi1}*{Definition~7.4.2} and our definitions, one has to show that $ N=T\sqcup M $ and the maps $ \nu \colon N\to \Aff (\R ) $ are the same. This is easily obtained by comparing \cite{BrTi1}*{Proposition~10.1.28 (iii)} with our definitions (see Definition~\ref{Def:sbgrp N SU unramified}, Definition~\ref{Def:affineactionforN unramified}, Definition~\ref{Def:sbgrp N SU ramified}, Definition~\ref{Def:affineactionforN ramified}, Definition~\ref{Def:sbgrp N SU ramifiedeven} and Definition~\ref{Def:affineactionforN ramifiedeven}).

\section{A review of the theory of CSA over local fields}\label{App:division algebra theory}
Let $ D $ be a central division algebra of degree $ d $ over a local field $ K $ (recall that the degree of $ D $ over $ K $ is the square root of the dimension of the $ K $-vector space $ D $). It is well known that such division algebras are classified (up to isomorphism) by elements of $ (\mathbf{Z}/d\mathbf{Z})^{\times} $ (see for example \cite{Pie82}*{Corollary~17.7a and Corollary~17.8b}).

To be explicit, for $ r\in (\mathbf{Z}/d\mathbf{Z})^{\times} $, the corresponding division algebra is the cyclic algebra $ (E/K,\sigma^{r}, \pi_{K}) $ where
\begin{itemize}
\item $ E $ is the unramified extension of $ K $ of dimension $ d $.
\item $ \sigma \in \Gal (E/K) $ is the element in $ \Gal (E/K) $ inducing the Frobenius automorphism on $ \overline{E} $.
\end{itemize}
For the reader's  convenience, we recall the definition of a cyclic algebra.
\begin{definition}\label{Def:cyclic algebra}
Let $ K $ be a field and let $ E/K $ be a cyclic extension of degree $ d $. Let $ \sigma $ be a generator of $ \Gal (E/K) $, and let $ a\in K^{\times} $. The cyclic algebra $ (E/K,\sigma ,a) $ is defined as follows:
\begin{itemize}
\item $ (E/K,\sigma ,a) = \bigoplus \limits_{i=0}^{d-1}u^{i}E $
\item $ u^{-1}xu = \sigma (x) $, for all $ x\in E $
\item $ u^{d} = a $
\end{itemize}
\end{definition}

\begin{definition}\label{Def:Hasse invariant}
As in \cite{dlST15}, for a finite central division algebra $ D $ of degree $ d $ over $ K $, we call the corresponding element $ r\in (\Z /d\Z )^{\times} $ the Hasse invariant of $ D $.
\end{definition}

An important fact about such a division algebra $ D $ is that it splits over $ E $. It is important for us to describe explicitly the embedding of $ D $ inside $ M_d(E) $, the algebra of $ d\times d $ matrices with coefficients in $ E $.
\begin{definition}\label{Def:embedding D in Mn(E)}
Let $ D $ be a division algebra isomorphic to the cyclic algebra $ (E/K,\sigma^{r},\pi_K) $ of degree $ d $ over $ K $. Consider the isomorphism of (right) $ E $-vector spaces 
\begin{equation*}
 f\colon E^d\to D\colon v=(v_1,\dots,v_d)\mapsto \sum \limits_{i=0}^{d-1}u^iv_{i+1}
\end{equation*}
Let $ \varphi \colon D\to M_d(E)\colon x\mapsto (v\mapsto f^{-1}(x.f(v))) $. More explicitly,
\begin{equation*}
\varphi (\sum \limits_{i=0}^{d-1}u^ix_{i+1}) = \begin{pmatrix}
x_1 & \pi_K \sigma^{r}(x_d) & \pi_K \sigma^{2r}(x_{d-1}) & \dots & \pi_K \sigma^{(d-1)r}(x_2)\\
x_2 & \sigma^{r}(x_1) & \pi_K \sigma^{2r}(x_{d}) & \dots & \pi_K \sigma^{(d-1)r}(x_{3}) \\
x_3 & \sigma^{r}(x_2) & \sigma^{2r}(x_1) & \dots & \pi_K \sigma^{(d-1)r}(x_{4}) \\
\vdots & \vdots & \vdots & \ddots & \vdots \\
x_d & \sigma^{r}(x_{d-1}) & \sigma^{2r}(x_{d-2}) & \dots & \sigma^{(d-1)r}(x_1)
\end{pmatrix}
\end{equation*}
\end{definition}

We can now properly spell out the definition of the reduced norm.
\begin{definition}\label{Def:reduced norm}
Let $ D $ be a division algebra isomorphic to the cyclic algebra $ (E/K,\sigma^{r},\pi_K) $ of degree $ d $ over $ K $. We define the reduced norm $ \Nrd $ as follows:
\begin{equation*}
\Nrd \colon M_n(D)\to K\colon g\to \det (\varphi (g_{ij}))
\end{equation*}
where $ \varphi (g_{ij}) $ is seen as a $ dn \times dn $ matrix with coefficients in $ E $.
\end{definition}

We end this discussion by an analysis of the ring of integers of $ D $.
\begin{lemma}\label{Lem:analysing O_D}
Let $ D,E,K $ be as in Definition~\ref{Def:reduced norm}, and let $ r\in \mathbf{N}\cup \lbrace \infty \rbrace $. Since $ E $ is unramified, $ \mathcal{O}_{E}/\mathfrak{m}_{E}^{r} \cong \mathcal{O}_{K}/\mathfrak{m}_{K}^{r}\oplus \dots \oplus \mathcal{O}_{K}/\mathfrak{m}_{K}^{r} $. Furthermore, $ \mathcal{O}_{D}/\mathfrak{m}_{D}^{rd}\cong \bigoplus \limits_{i=0}^{d-1}u^{i}.\mathcal{O}_{E}/\mathfrak{m}_{E}^{r} $. This shows that $ \mathcal{O}_{D}/\mathfrak{m}_{D}^{rd} $ is a free $ \mathcal{O}_{E}/\mathfrak{m}_{E}^{r} $-module $ ( $with the convention that $ \mathfrak{m}^{\infty} = (0))$, and that we can define a map $ \overline{\varphi} \colon \mathcal{O}_{D}/\mathfrak{m}_{D}^{rd}\hookrightarrow M_d(\mathcal{O}_{E}/\mathfrak{m}_{E}^{r}) $, which is compatible with the map $ \varphi $ of Definition~\ref{Def:embedding D in Mn(E)}, in the sense that the following diagram commutes

\begin{center}
\begin{tikzpicture}[->]
 \node (1) at (0,0) {$ \mathcal{O}_{D} $};
 \node (2) at (2.0,0) {$ M_d(\mathcal{O}_{E}) $};
 \node (5) at (0.8,0) {$ \hookrightarrow $};
 \node (3) at (-0.3,-1.2) {$ \mathcal{O}_{D}/\mathfrak{m}_{D}^{rd} $};
 \node (4) at (2.3, -1.2) {$ M_d(\mathcal{O}_{E}/\mathfrak{m}_{E}^{r}) $};
 \node (6) at (0.8,-1.2) {$ \hookrightarrow $};
  
 \draw[->] (0,-0.3) to (0,-0.9);
 \draw[->] (1.9,-0.3) to (1.9,-0.9);
\end{tikzpicture}
\end{center}
\end{lemma}
\begin{proof}
This is straightforward from the definitions.
\end{proof}

\section{An integral model of \texorpdfstring{$ \SL_2(D) $}{SL_2(D)}}\label{App:integral models over D}
Recall that the group $ \SL_2(D) $ consists of the $ 2\times 2 $ matrices with coefficients in $ D $ having reduced norm $ 1 $ (Definition~\ref{Def:SL_2(D)}). Recall the definition of the embedding $ \varphi \colon D\to M_d(E) $ given in Definition~\ref{Def:embedding D in Mn(E)}. In view of the definition of the reduced norm (Definition~\ref{Def:reduced norm}), we arrive at the following explicit definition of $ \SL_2(D) $.

\begin{definition}$ \SL_2(D) = \lbrace \begin{psmallmatrix} g_{11} & g_{12}\\g_{21} & g_{22}\end{psmallmatrix}~\vert~g_{ij} \in D,~\det (\begin{psmallmatrix} \varphi(g_{11}) & \varphi(g_{12})\\\varphi(g_{21}) & \varphi(g_{22})\end{psmallmatrix}) = 1 \rbrace $
\end{definition}

Mimicking this definition, we can define a similar group over $ \mathcal{O}_{D}/\mathfrak{m}_{D}^{rd} $.

\begin{definition}\label{Def:SL_2(O_D/m^r)}
Let $ D $ be a central division algebra over $ K $ of degree $ d $, and let $ r\in \mathbf{N}\cup \lbrace \infty \rbrace $. Keeping the notations of Lemma~\ref{Lem:analysing O_D}, we define
\begin{equation*}
\SL_2(\mathcal{O}_{D}/\mathfrak{m}_{D}^{rd}) = \lbrace \begin{psmallmatrix} g_{11} & g_{12}\\g_{21} & g_{22}\end{psmallmatrix}~\vert~g_{ij}\in \mathcal{O}_{D}/\mathfrak{m}_{D}^{rd},~\det (\begin{psmallmatrix} \overline{\varphi}(g_{11}) & \overline{\varphi}(g_{12})\\\overline{\varphi}(g_{21}) & \overline{\varphi}(g_{22})\end{psmallmatrix}) =1 \rbrace
\end{equation*}
\end{definition} 

Let us now discuss the underlying algebraic group of $ \SL_2(D) $. Let $ M_2(D) $ be the algebra of $ 2\times 2 $ matrices with coefficients in $ D $. Using the embedding $ D\hookrightarrow M_d(E) $, we can identify $ M_2(D) $ with a $ K $-linear subspace of $ M_{2d}(E) $. Now, $ \SL_2(D) $ is the closed subspace of $ M_2(D)\cong \mathbf{A}_{K}^{4d^2} $ cut out by the polynomial equation $ \Nrd = 1 $. We can mimic this situation over the ring of integers to define an integral model of $ \SL_2(D) $.

\begin{definition}\label{Def:integral model for SL_2(D)}
Let $ D $ be a central division algebra of degree $ d $ over $ K $, and let $ M_2(\mathcal{O}_D) $ be the $ \mathcal{O}_K $-algebra of $ 2\times 2 $ matrices with coefficients in $ \mathcal{O}_D $. Using the embedding $ \mathcal{O}_D\hookrightarrow M_d(\mathcal{O}_E) $, where $ E $ is the unramified extension of $ K $ of degree $ d $, we can identify $ M_2(\mathcal{O}_D) $ with a free $ \mathcal{O}_K $-submodule of $ M_{2d}(\mathcal{O}_E) $. We define the $ \mathcal{O}_{K} $-scheme $ \underline{\SL}_{2,D} $ to be the closed subscheme of $ M_2(\mathcal{O}_D)\cong \mathbf{A}_{\mathcal{O}_{K}}^{4d^2} $ cut out by the polynomial equation $ \Nrd = 1 $.
\end{definition}

Of course, the crucial point is to check that $ \underline{\SL}_{2,D} $ is in fact smooth.
\begin{theorem}\label{Thm:smmoothness for SL_2(D)}
$ \underline{\SL}_{2,D} $ is a smooth $ \mathcal{O}_{K} $-scheme.
\end{theorem}
\begin{proof}
This is one of the main results in \cite{BrTi84}. Let us explain how to extract it from there. Let $ \varphi $ be the valuation of $ GL_2(D) $ defined in \cite{BrTi84}*{2.2, display (4)}. The valuation $ \varphi $ is thus a point of the enlarged apartment $ A_{1} $. The associated norm is defined as $ \alpha_{\varphi}(e_1x_1+e_2x_2)=\inf \lbrace \omega (x_1),\omega (x_2)\rbrace $ (following the definition in \cite{BrTi84}*{2.8, display (9)}). The corresponding order $ \mathscr{M}_{\alpha_{\varphi}} $ of $ M_2(D) $ defined in \cite{BrTi84}*{1.17} is $ \lbrace\begin{psmallmatrix} g_{11} & g_{12}\\g_{21} & g_{22}\end{psmallmatrix}\in M_2(D)~\vert~\omega (g_{ij})\geq 0 \rbrace $ (this is easily computed using the description of $ \End \alpha (u) $ in \cite{BrTi84}*{1.11, display (17)}). Note that $ \mathscr{M}_{\alpha_{\varphi}} $ is isomorphic to the affine space $ \mathbf{A}_{\mathcal{O}_{K}}^{(2d)^2} $ (being a free $ \mathcal{O}_{K} $-module). Finally, following \cite{BrTi84}*{3.6}, let $ \textswab{G}_{\varphi} $ be the (principal) open subscheme of the affine space  $ \mathscr{M}_{\alpha_{\varphi}} $ defined by the non-vanishing of the reduced norm (see also \cite{BrTi84}*{3.2}).

$ \textswab{G}_{\varphi} $ is actually an integral model for $ GL_2(D) $, and the $ \SL_2(D) $ case is then treated in \cite{BrTi84}*{§5}. Let $ \textswab{G}_{1,\varphi} $ be the schematic adherence of $ \SL_2(D) $ in $ \textswab{G}_{\varphi} $ (following the definition in \cite{BrTi84}*{5.3}). It is mentioned in \cite{BrTi84}*{5.5} that the group $ \textswab{G}_{1,\varphi} $ is the closed subgroup of $ \textswab{G}_{\varphi} $ defined by the equation $ \Nrd = 1 $, and hence it coincides with our group $ \underline{SL}_{2,D} $. But by \cite{BrTi84}*{5.5}, $ \textswab{G}_{1,\varphi} $ is smooth over $ \mathcal{O}_K $, concluding the proof. Note that to apply \cite{BrTi84}*{5.5}, we should check that a finite unramified extension of a local field is étale in the sense of \cite{BrTi84}. But this is clear in view of \cite{BrTi2}*{1.6.1 (f) and Definition~1.6.2}. 
\end{proof}

We conclude our study of the $ \SL_2(D) $ case by identifying the rational points of $ \underline{\SL}_{2,D} $.
\begin{lemma}\label{Lem:rational points for SL_2(D)}
Let $ D $ be a central division algebra over $ K $ of degree $ d $, and let $ r\in \mathbf{N}\cup \lbrace \infty\rbrace $. Then $ \underline{\SL}_{2,D}(\mathcal{O}_{K}/\mathfrak{m}_{K}^{r})\cong \SL_{2}(\mathcal{O}_{D}/\mathfrak{m}_{D}^{rd}) $ $ ( $where by convention, $ \mathfrak{m}^{\infty} = (0))$.
\end{lemma}
\begin{proof}
Because the diagram appearing in Lemma~\ref{Lem:analysing O_D} is commutative, we have
\begin{align*}
\underline{\SL}_{2,D}(\mathcal{O}_{K}/\mathfrak{m}_{K}^{r}) \cong & \lbrace \begin{psmallmatrix} g_{11} & g_{12}\\g_{21} & g_{22}\end{psmallmatrix}\in M_2(\mathcal{O}_{D}/\mathfrak{m}_{D}^{rd})~\vert~ \Nrd (g) = 1\rbrace \\
= & \lbrace\begin{psmallmatrix} g_{11} & g_{12}\\g_{21} & g_{22}\end{psmallmatrix}\in M_2(\mathcal{O}_{D}/\mathfrak{m}_{D}^{rd})~\vert~ \det (\begin{psmallmatrix} \overline{\varphi}(g_{11}) & \overline{\varphi}(g_{12})\\\overline{\varphi}(g_{21}) & \overline{\varphi}(g_{22})\end{psmallmatrix})=1\rbrace
\end{align*}
as wanted.
\end{proof}

\end{document}